\RequirePackage{fix-cm}
\documentclass[smallextended]{svjour3}   
\smartqed  
\usepackage[T1]{fontenc}
\usepackage{graphicx}
\usepackage{amsmath,amssymb,amstext,amsfonts}
\usepackage{graphicx}
\usepackage{epstopdf}
\usepackage{algorithm}
\usepackage{algorithmic}
\usepackage{type1cm}        
\usepackage{makeidx}         
\usepackage{graphicx}        
\usepackage{multicol}        
\usepackage[bottom]{footmisc}

\makeindex             

\usepackage{graphicx}
\usepackage{amsfonts,amsmath, amssymb, amstext,color}
\usepackage{graphicx} 

\usepackage{braket}
\usepackage{mathrsfs}
\usepackage{comment}
\usepackage{mathtools}

\newcommand{\Scal}{\mathcal{S}}

\newcommand{\KL}{{\rm KL}}
\newcommand{\ep}{\epsilon}

\newcommand{\Hrm}{{\rm H}}

\newcommand{\GL}{{\rm GL}}

\newcommand{\Xcal}{\mathcal{X}}

\newcommand{\dettwo}{\mathrm{det_2}}

\newcommand{\Xfrak}{\mathfrak{X}}
\newcommand{\SymTr}{\mathrm{SymTr}}
\newcommand{\varep}{\varepsilon}

\newcommand{\Hcal}{\mathcal{H}}
\newcommand{\Nbb}{\mathbb{N}}

\newcommand{\PC}{\mathscr{PC}}

\newcommand{\Gauss}{\mathrm{Gauss}}

\newcommand{\Sym}{\mathrm{Sym}}
\newcommand{\Tr}{\mathrm{Tr}}
\newcommand{\HS}{\mathrm{HS}}
\newcommand{\Ncal}{\mathcal{N}}
\newcommand{\trace}{\mathrm{tr}}
\newcommand{\R}{\mathbb{R}}
\newcommand{\compose}{\ensuremath{\circ}}
\newcommand{\la}{\ensuremath{\langle}}
\newcommand{\ra}{\ensuremath{\rangle}}
\newcommand{\mapto}{\ensuremath{\rightarrow}}
\newcommand{\Lcal}{\mathcal{L}}
\newcommand{\range}{\mathrm{range}}
\newcommand{\approach}{\ensuremath{\rightarrow}}

\newcommand{\equivalent}{\ensuremath{\iff}}
\newcommand{\tr}{\mathrm{tr}}
\newcommand{\imply}{\ensuremath{\Rightarrow}}
\newcommand{\SymHS}{\mathrm{SymHS}}

\newcommand{\mysqrt}{\mathrm{sqrt}}
\newcommand{\aiHS}{\mathrm{aiHS}}
\newcommand{\Bsc}{\mathscr{B}}
\newcommand{\Exp}{\mathrm{Exp}}
\newcommand{\Log}{\mathrm{Log}}
\title{Fisher-Rao geometry of equivalent Gaussian measures on infinite-dimensional  Hilbert spaces
}
\titlerunning{Fisher-Rao Riemannian geometry of Gaussian measures on Hilbert space}
\author{H\`a Quang Minh}
\institute{RIKEN Center for Advanced Intelligence Project, Tokyo, Japan
\\
\email{minh.haquang@riken.jp}
}
\begin{document}
\maketitle              
\begin{abstract}
This work presents an explicit description of the Fisher-Rao Riemannian metric on the Hilbert manifold
of equivalent centered Gaussian measures on an infinite-dimensional Hilbert space. We show
that the corresponding quantities from the finite-dimensional setting of Gaussian densities on Euclidean space,
including the Riemannian metric, Levi-Civita connection, curvature, geodesic curve, and Riemannian distance,
when properly formulated, directly generalize to this setting. Furthermore, we discuss the connection with
the Riemannian geometry of positive definite unitized Hilbert-Schmidt operators on Hilbert space, which can be viewed as a regularized version
of the current setting. 
\keywords{Fisher-Rao metric  \and Gaussian measures \and Hilbert space \and positive Hilbert-Schmidt operators.}
\end{abstract}

\section{Introduction}
\label{section:intro}

This work studies the Fisher-Rao metric, an object of central importance in information geometry,
on the set of all centered Gaussian measures on a Hilbert space that are equivalent to a fixed one. This is the infinite-dimensional Hilbert space generalization of the Fisher-Rao metric on the set of all Gaussian densities on $\R^n$.
We show that, in this setting, many quantities of interest in the corresponding Riemannian manifold structure admit explicit expressions that directly generalize those in the finite-dimensional setting. We also show that these results are closely connected with the Riemannian geometry of the set of positive definite unitized Hilbert-Schmidt operators on Hilbert space.

We first briefly review the Fisher-Rao metric,
for more detail we refer to e.g. 
\cite{Amari:InformationGeometry2000},  with the focus on the finite-dimensional setting of Gaussian densities on $\R^n$.
Let $\Scal$ be a family of probability density functions $P_{\theta}$ on $\Xcal = \R^n$, parametrized by 
a parameter $\theta = (\theta^1, \ldots, \theta^k) \in \Theta$, where $\Theta$ is an open subset  in $\R^k$, for some $k \in \Nbb$, that is 
$	\Scal = \{P_{\theta} = P(x;\theta) \; | \; \theta = (\theta^1, \ldots, \theta^k) \in \Theta \subset \R^k\}$,
where the mapping $\theta \mapto P_{\theta}$ is assumed to be injective. Such an $\Scal$ is called a $k$-dimensional {\it statistical model}
or  a {\it parametric model} on $\Xcal$. Assume further that for each fixed $x \in \Xcal$,
the mapping $\theta \mapto P(x;\theta)$ is $C^{\infty}$, so that all partial derivatives, such as $\frac{\partial P(x;\theta)}{\partial \theta^i}$, $1 \leq i \leq k$,
are well-defined and continuous.

A $k$-dimensional statistical model $\Scal$ can be considered as a smooth manifold. 
%
At each point $\theta \in \Theta$, 
the {\it Fisher information matrix} \cite{fisher1922mathematical} of $\Scal$ at 
$\theta$ is the $k \times k$ matrix $G(\theta) = [g_{ij}(\theta)]$, $1\leq i,j \leq k$,
with the $(i,j)$th entry given by
\begin{align}
	g_{ij}(\theta) = \int_{\R^n}\frac{\partial \ln P(x;\theta)}{\partial \theta^i}\frac{\partial \ln P(x;\theta)}{\partial \theta^j}P(x;\theta)dx.
\end{align}
It is clear
that the Fisher information matrix $G(\theta)$ is symmetric, positive semi-definite. Assume further that
$G(\theta)$ is strictly positive definite $\forall \theta \in \Theta$, then it defines an inner product on the tangent space $T_{P_{\theta}}(\Scal)$, via the
inner product on the basis $\{\frac{\partial}{\partial \theta^j}\}_{j=1}^k$ of $T_{P_{\theta}}(\Scal)$, by
$	\left\la \frac{\partial}{\partial \theta^i}, \frac{\partial}{\partial \theta^j}\right\ra_{P_{\theta}} = g_{ij}(\theta).
$
This inner product defines a Riemannian metric on $\Scal$, the so-called {\it Fisher-Rao metric}, or {\it Fisher information metric} \cite{Rao1945}, turning $\Scal$ into a Riemannian manifold.

{\bf Gaussian density setting}. 
Let $\Sym(n)$ denote the set of $n \times n$ real symmetric matrices.
Let $\Sym^{++}(n)$ denote the set of $n\times n$ real symmetric, positive definite matrices.
Consider the family $\Scal$ of multivariate Gaussian density functions on $\R^n$ with mean zero
\begin{align}
	\label{equation:Gaussian-family}
	\Scal = \left\{P(x;\theta) = \frac{1}{\sqrt{(2\pi)^n \det(\Sigma(\theta))}}\exp\left(-\frac{1}{2}x^T\Sigma(\theta)^{-1}x\right)\right.,
	\nonumber
	\\ 
	\left.\Sigma(\theta) \in \Sym^{++}(n), \theta \in \R^{k}\right\}.
\end{align}
Here $k = \frac{n(n+1)}{2}$ and $\theta = [\theta^1, \ldots, \theta^k]$, with the $\theta^j$'s corresponding
to the upper triangular entries in $\Sigma(\theta)$ according to the following order:
$\Sigma(\theta)_{11} = \theta^{1}$, $\Sigma(\theta)_{12} = \theta^{2}$, $\ldots, \Sigma(\theta)_{22} = \theta^{n+1}$, 
$\ldots, \Sigma(\theta)_{nn} = \theta^{\frac{n(n+1)}{2}}$.
%
In this case, the Fisher information matrix is given by the following (see e.g \cite{Skov:Riemannian1984,Lenglet:Gaussian2006,Felice:Gaussian2017})
	\begin{align}
		\label{equation:Fisher-Gaussian-finite}
		g_{ij}(\theta) = \frac{1}{2}\trace[\Sigma^{-1}(\partial_{\theta^i}\Sigma) \Sigma^{-1}(\partial_{\theta^j}\Sigma)], \;\; 1 \leq i,j \leq k,
	\end{align}
	where $\partial_{\theta^i} = \frac{\partial}{\partial \theta^i}$, $\partial_{\theta^j} = \frac{\partial}{\partial \theta^j}$.
With the one-to-one correspondence $P_{\theta} \leftrightarrow \Sigma(\theta)$,
we can identify the statistical manifold $\Scal$ with the manifold $\Sym^{++}(n)$ and the corresponding tangent space $T_{P_{\theta}}(\Scal)$ with the tangent space $T_{\Sigma(\theta)}(\Sym^{++}(n)) \cong \Sym(n)$.
The corresponding Riemannian metric on $\Sym^{++}(n)$ is given by  
\begin{align}
	\la A,B\ra_{\Sigma} &= \frac{1}{2}\trace(\Sigma^{-1}A\Sigma^{-1}B), \;\; A,B \in \Sym(n), \Sigma \in \Sym^{++}(n)
	\label{equation:metric-Sym++(n)-1}
	\\
	&= \frac{1}{2}\trace[(\Sigma^{-1/2}A\Sigma^{-1/2})(\Sigma^{-1/2}B\Sigma^{-1/2})].
	\label{equation:metric-Sym++(n)-1/2}
\end{align}
This is precisely $1/2$ the so-called affine-invariant Riemannian metric on $\Sym^{++}(n)$, which has been studied extensively, see e.g. \cite{Bhatia:2007,Pennec:IJCV2006}.

{\bf Infinite-dimensional Gaussian setting}. In this work, we generalize the Fisher-Rao metric for the Gaussian densities in $\R^n$ to the setting of Gaussian measures on an infinite-dimensional separable Hilbert space $\Hcal$. In the general Gaussian setting, this is not possible since no Lebesgue measure exists on $\Hcal$, hence density functions are not well-defined. What we can show is the generalization of the Fisher-Rao metric to the set of Gaussian measures {\it equivalent to a fixed Gaussian measure}, which is a Hilbert manifold,  along with the corresponding Riemannian connection, curvature tensor, geodesic, and Riemannian distance, all in closed form expressions. Furthermore, we show that this
setting is closely related to the geometric setting of positive definite unitized Hilbert-Schmidt operators in \cite{Larotonda:2007}, which can be viewed as a 
regularized version of the equivalent Gaussian measure setting. The current work
thus provides a link between the geometric framework for positive Hilbert-Schmidt operators in \cite{Larotonda:2007} with the information geometry of Gaussian measures on Hilbert space.

{\bf Related work in the infinite-dimensional setting}.
While most work in information geometry is concerned with the finite-dimensional setting, many authors have also considered the infinite-dimensional setting.
In \cite{pistone1995infinite} and subsequent work \cite{gibilisco1998connections,pistone1999exponential,cena2007exponential}, the authors constructed an infinite-dimensional Banach manifold, modeled on Orlicz spaces, for the set of all probability measures equivalent to a given one. 
In \cite{newton2012infinite}, the author constructed a Hilbert manifold of all probability measures equivalent to a given one, with finite entropy, along with the definition of the Fisher-Rao metric, which is generally a pseudo-Riemannian metric.
In \cite{ay2015informationSufficient,ay2017informationGeometry,ay2018parametrized},
the authors constructed general parametrized measure models and statistical models on a given sample space by utilizing the natural immersion of the set of probability measures into the Banach space of all finite signed measures under the total variation norm.
This framework is independent of the reference measure and encompasses that proposed in \cite{pistone1995infinite}.
The previously mentioned work all deal with highly general settings.
Instead, the current work focuses exclusively with the concrete setting
of equivalent Gaussian measures on Hilbert space, where a concrete Hilbert manifold structure exists and many quantities of interest can be computed explicitly.

{\it A short and preliminary version of this work, containing a summary of the main results, but without proofs, was presented at the
{6th International Conference on Geometric Science of Information} \cite{Minh:GSI2023}}.

\section{Background: Gaussian measures and positive Hilbert-Schmidt operators on Hilbert space}
\label{section:finite}

Throughout the following, let $(\Hcal, \la, \ra)$ be a real, separable Hilbert space, with $\dim(\Hcal) = \infty$ unless explicitly stated otherwise.
For two separable Hilbert spaces {\color{black}$(\Hcal_i, \la,\ra_i)$},$i=1,2$, let $\Lcal(\Hcal_1,\Hcal_2)$ denote the Banach space of bounded linear operators from $\Hcal_1$ to $\Hcal_2$, with operator norm $||A||=\sup_{||x||_1\leq 1}||Ax||_2$.
For $\Hcal_1=\Hcal_2 = \Hcal$, we use the notation $\Lcal(\Hcal)$.

Let $\Sym(\Hcal) \subset \Lcal(\Hcal)$ be the set of bounded, self-adjoint linear operators on $\Hcal$. Let $\Sym^{+}(\Hcal) \subset \Sym(\Hcal)$ be the set of
self-adjoint, {\it positive} operators on $\Hcal$, i.e. $A \in \Sym^{+}(\Hcal) \equivalent A^{*}=A, \la Ax,x\ra \geq 0 \forall x \in \Hcal$. 
Let $\Sym^{++}(\Hcal)\subset \Sym^{+}(\Hcal)$ be the set of self-adjoint, {\it strictly positive} operator on $\Hcal$,
i.e $A \in \Sym^{++}(\Hcal) \equivalent A^{*}=A, \la x, Ax\ra > 0$ $\forall x\in \Hcal, x \neq 0$.
We write $A \geq 0$ for $A \in \Sym^{+}(\Hcal)$ and $A > 0$ for $A \in \Sym^{++}(\Hcal)$.
If $\gamma I+A > 0$, where $I$ is the identity operator,$\gamma \in \R,\gamma > 0$, then $\gamma I+A$ is also invertible, in which case it is called
{\it positive definite}. {\color{black}In general, $A \in  \Sym(\Hcal)$ is said to be positive definite if $\exists M_A > 0$ such that $\la x, Ax\ra \geq M_A||x||^2$ $\forall x \in \Hcal$ - this condition is equivalent to $A$ being both strictly positive and invertible, see e.g. \cite{Petryshyn:1962}.}

The Banach space $\Tr(\Hcal)$  of trace class operators on $\Hcal$ is defined by (see e.g. \cite{ReedSimon:Functional})
$\Tr(\Hcal) = \{A \in \Lcal(\Hcal): ||A||_{\tr} = \sum_{k=1}^{\infty}\la e_k, (A^{*}A)^{1/2}e_k\ra < \infty\}$,
for any orthonormal basis {\color{black}$\{e_k\}_{k \in \Nbb} \subset \Hcal$}.
For $A \in \Tr(\Hcal)$, its trace is defined by $\trace(A) = \sum_{k=1}^{\infty}\la e_k, Ae_k\ra$, which is independent of choice of $\{e_k\}_{k\in \Nbb}$. 

The Hilbert space $\HS(\Hcal_1,\Hcal_2)$ of Hilbert-Schmidt operators from $\Hcal_1$ to $\Hcal_2$ is defined by 
(see e.g. \cite{Kadison:1983})
$\HS(\Hcal_1, \Hcal_2) = \{A \in \Lcal(\Hcal_1, \Hcal_2):||A||^2_{\HS} = \trace(A^{*}A) =\sum_{k=1}^{\infty}||Ae_k||_2^2 < \infty\}$,
for any orthonormal basis $\{e_k\}_{k \in \Nbb}$ in $\Hcal_1$,
with inner product $\la A,B\ra_{\HS}=\trace(A^{*}B)$. For $\Hcal_1 = \Hcal_2 = \Hcal$, we write $\HS(\Hcal)$. 
When $\dim(\Hcal) = \infty$,
{\color{black}$\Tr(\Hcal) \subsetneq \HS(\Hcal) \subsetneq \Lcal(\Hcal)$},
 with $||A||\leq ||A||_{\HS}\leq ||A||_{\tr}$.


{\bf Equivalence of Gaussian measures}. 
On $\R^n$, any two Gaussian densities are equivalent, that is they have the same support, which is all of $\R^n$.
The situation is drastically different in the infinite-dimensional setting.
Let $Q,R$ be two self-adjoint, positive trace class operators on $\Hcal$ such that $\ker(Q) = \ker(R) = \{0\}$. Let $m_1, m_2 \in \Hcal$. 
A fundamental result in the theory of Gaussian measures is the Feldman-Hajek Theorem \cite{Feldman:Gaussian1958}, \cite{Hajek:Gaussian1958}, which states that 
two Gaussian measures $\mu = \Ncal(m_1,Q)$ and 
$\nu = \Ncal(m_2, R)$
are either mutually singular or
equivalent, {\color{black}that is either $\mu \perp \nu$ or $\mu \sim \nu$}.
The necessary and sufficient conditions for
the equivalence of the two Gaussian measures $\nu$ and $\mu$ are given by the following.
\begin{theorem}
	[\cite{Bogachev:Gaussian}, Corollary 6.4.11, \cite{DaPrato:PDEHilbert}, Theorems  1.3.9 and 1.3.10]
	\label{theorem:Gaussian-equivalent}
	Let $\Hcal$ be a separable Hilbert space. Consider two Gaussian measures $\mu = \Ncal(m_1, Q)$,
	$\nu = \Ncal(m_2, R)$ on $\Hcal$. Then $\mu$ and $\nu$ are equivalent if and only if 
	the following
	conditions 
	both hold
	\begin{enumerate}
		\item $m_2 - m_1 \in \range(Q^{1/2})$.
		\item There exists  $S \in  \Sym(\Hcal) \cap \HS(\Hcal)$, without the eigenvalue $1$, such that
		$R = Q^{1/2}(I-S)Q^{1/2}$.
	\end{enumerate}
\end{theorem}

{\bf Riemannian geometry of positive definite unitized Hilbert-Schmidt operators}.
Let $\Gauss(\Hcal)$ denote the set of all Gaussian measures on $\Hcal$. Each zero-mean Gaussian measure
$\mu = \Ncal(0,C) \in \Gauss(\Hcal)$ corresponds to an operator $C \in \Sym^{+}(\Hcal) \cap \Tr(\Hcal)$ and vice versa.
The set $\Sym^{+}(\Hcal) \cap \Tr(\Hcal)$ of positive trace class operators is a subset of the set
$\Sym^{+}(\Hcal) \cap \HS(\Hcal)$ of positive Hilbert-Schmidt operators. The generalization
of the affine-invariant Riemannian metric on $\Sym^{++}(n)$ to this set is 
accomplished via extended (unitized) Hilbert-Schmidt operators \cite{Larotonda:2007}, as follows.

{\bf Extended Hilbert-Schmidt operators}. 
In \cite{Larotonda:2007}, the author considered the following set of {\it extended}, or {\it unitized}, Hilbert-Schmidt operators
\begin{align}
	\HS_X(\Hcal) = \{A + \gamma I: A \in \HS(\Hcal), \gamma \in \R\}.
\end{align}
This set is a Hilbert space under the {\it extended Hilbert-Schmidt inner product and extended Hilbert-Schmidt norm}, under which the Hilbert-Schmidt and scalar operators are orthogonal, 
\begin{align}
	\la A+\gamma I, B + \mu I\ra_{\HS_X} = \la A,B\ra_{\HS} + \gamma\mu, \;\; ||A+\gamma I||^2_{\HS_X} = ||A||^2_{\HS} + \gamma^2.
\end{align}
{\bf Manifold of positive definite Hilbert-Schmidt operators}. 
Consider the following subset of {\it (unitized) positive definite Hilbert-Schmidt operators}
\begin{align}
	\PC_2(\Hcal) = \{A+\gamma I > 0: A \in \Sym(\Hcal) \cap\HS(\Hcal), \gamma \in \R\} \subset \HS_X(\Hcal).
\end{align}
The set $\PC_2(\Hcal)$ is an open subset in the Hilbert space $\HS_X(\Hcal)$ and is thus a Hilbert manifold.
It can be equipped with the following Riemannian metric, generalizing the finite-dimensional affine-invariant metric,
\begin{align}
		\label{equation:affine-Riemannian-metric-infinite}
	\la A+\gamma I, B+\mu I\ra_{P} = \la P^{-1/2}(A+\gamma I)P^{-1/2}, P^{-1/2}(B+\mu I)P^{-1/2}\ra_{\HS_X},
\end{align}
for $P \in \PC_2(\Hcal)$, $A+\gamma I,B+\mu I \in T_P(\PC_2(\Hcal)) \cong \Sym(\Hcal) \cap \HS_X(\Hcal))$.
Under this metric, $\PC_2(\Hcal)$ becomes a Cartan-Hadamard manifold, that is it is complete, simply connected, 
with nonpositive sectional curvature. There is a unique geodesic joining every pair 
$(A+\gamma I), (B+\mu I)$, given by
\begin{align}
	\gamma_{AB}(t) =& (A+\gamma I)^{1/2}\exp[t\log((A+\gamma I)^{-1/2}(B+\mu I)(A+\gamma I)^{-1/2})](A+\gamma) I^{1/2},
	\nonumber
	\\
	& 0 \leq t \leq 1.
\end{align}
The Riemannian distance between $(A+\gamma I),(B+\mu I) \in \PC_2(\Hcal)$ is the length of this geodesic and is given by
\begin{align}
		\label{equation:d-aiHS}
	d_{\aiHS}[(A+\gamma I), (B+\mu I)] = ||\log[(A+\gamma I)^{-1/2}(B+\mu I)(A+\gamma I)^{-1/2}]||_{\HS_X}.
\end{align}
The definition of the extended Hilbert-Schmidt norm $||\;||_{\HS_X}$ guarantees that the distance $d_{\aiHS}$ is always
well-defined and finite on $\PC_2(\Hcal)$.
We show below that, when restricted to the subset $\Sym^{+}(\Hcal)\cap \Tr(\Hcal)$, this can be viewed as a regularized version of the exact Fisher-Rao distance between
two equivalent centered Gaussian measures on $\Hcal$. In the next section, we show that the set of 
equivalent Gaussian measures on a separable Hilbert space forms a Hilbert manifold, modeled on the Hilbert space of extended Hilbert-Schmidt operators.

\section{Fisher-Rao metric for equivalent infinite-dimensional Gaussian measures}
\label{section:infinite-gaussian}

Throughout the following,
let $C_0 \in \Sym^{+}(\Hcal) \cap \Tr(\Hcal)$ be fixed, with $\ker(C_0) = {0}$. Let
$\mu_0 =\Ncal(0,C_0)$ be the corresponding Gaussian measure, which is then said to be {\it nondegenerate}.
Consider the set of all zero-mean Gaussian measures on $\Hcal$ equivalent to $\mu_0$,
which, by Theorem \ref{theorem:Gaussian-equivalent}, is given by
\begin{align}
	\label{equation:def-Gauss-HcalC0}
	\Gauss(\Hcal,\mu_0) = \{\mu = \Ncal(0, C), \;\; C = C_0^{1/2}(I-S)C_0^{1/2}, 
	\nonumber
	\\
	S \in  \Sym(\Hcal)\cap \HS(\Hcal), I-S > 0
	\}.
\end{align}
Motivated by Theorem \ref{theorem:Gaussian-equivalent}, we define the following set
\begin{align}
	\label{equation:SymHSHcal<I}
\SymHS(\Hcal)_{<I} = \{S: S \in \Sym(\Hcal) \cap \HS(\Hcal), I-S > 0\}.
\end{align} 
This is an open subset in the Hilbert space $\Sym(\Hcal) \cap \HS(\Hcal)$ and hence is a Hilbert manifold
(for Banach manifolds in general, see e.g. \cite{lang2012fundamentals}).
\begin{lemma}
\label{lemma:Hilbert-manifold-SymHS<I}
The set $\SymHS(\Hcal)_{<I}$ is a Hilbert manifold with tangent space $T_S(\SymHS(\Hcal)_{<I}) \cong \Sym(\Hcal)\cap \HS(\Hcal)$
$\forall S \in  \SymHS(\Hcal)_{<I}$.
\end{lemma}

The set $\Gauss(\Hcal,\mu_0)$
corresponds to the following subset of $\Sym^{+}(\Hcal) \cap \Tr(\Hcal)$
\begin{align}
	\Tr(\Hcal,C_0) = \{C \in \Sym^{+}(\Hcal) \cap \Tr(\Hcal): C = C_0^{1/2}(I-S)C_0^{1/2}
	\nonumber
	\\
	\;\text{for some } S \in \SymHS(\Hcal)_{<I}\}.
	\label{equation:def-TrHC0}
\end{align}
The set $\Tr(\Hcal,C_0)$, equivalently $\Gauss(\Hcal,\mu_0)$, is a Hilbert manifold modeled on $\Sym(\Hcal)\cap\HS(\Hcal)$ via the following bijection
\begin{align}
	\varphi:\Tr(\Hcal,C_0) \mapto \SymHS(\Hcal)_{<I}, \;\;\; \varphi(C) = I-C_0^{-1/2}CC_0^{-1/2}.
\end{align}
We show that $\Tr(\Hcal,C_0)$ can also be embedded as an open subset in a larger Hilbert space, as follows.
Define the following set 
\begin{align}
	\label{equation:def-HSX}
\HS_X(\Hcal,C_0) = \{C_0^{1/2}(A+\gamma I)C_0^{1/2}: A \in \HS(\Hcal), \gamma \in \R\} \subset \Tr(\Hcal).
\end{align}
This is a Hilbert space under the following inner product and norm
\begin{align}
&\la C_0^{1/2}(A+\gamma I)C_0^{1/2}, C_0^{1/2}(B + \mu I)C_0^{1/2}\ra_{\HS_X(\Hcal,C_0)} = \la A,B\ra_{\HS} + \gamma\mu, 
\label{equation:HSX-H-C0-inner-product}
\\
&||C_0^{1/2}(A+\gamma I)C_0^{1/2}||^2_{\HS_X(\Hcal,C_0)} = ||A||^2_{\HS} + \gamma^2.
\label{equation:HSX-H-C0-norm}
\end{align}
$\Tr(\Hcal, C_0)$ is a subset of  the following subspace of self-adjoint operators in $\HS_X(\Hcal,C_0)$
\begin{align}
	\label{equation:def-SymHSX}
	\SymHS_X(\Hcal,C_0) = \{C_0^{1/2}(A+\gamma I)C_0^{1/2}: A \in \Sym(\Hcal) \cap \HS(\Hcal), \gamma \in \R\}.
\end{align}
Since the set $\{S: S \in \Sym(\Hcal)\cap \HS(\Hcal), I-S > 0\}$ is an open subset in the Hilbert space $\Sym(\Hcal) \cap\HS(\Hcal)$
under the $||\;||_{\HS}$ norm, it follows that $\Tr(\Hcal,C_0)$ is an open subset in the Hilbert space
$\SymHS_X(\Hcal,C_0)$ under the $||\;||_{\HS_X(\Hcal,C_0)}$ norm. Thus
$\Tr(\Hcal,C_0)$ is a Hilbert manifold modeled on $\SymHS_X(\Hcal,C_0)$.
By the correspondence $\Ncal(0,C) \in \Gauss(\Hcal,\mu_0)\equivalent C \in \Tr(\Hcal,C_0)$,
we have that the corresponding set of zero-mean Gaussian measures $\Gauss(\Hcal,\mu_0)$ is a Hilbert manifold modeled on $\SymHS_X(\Hcal,C_0)$. Subsequently, we focus on this manifold structure.

{\bf Tangent space}. Let $\Sigma \in \Tr(\Hcal,C_0)$ be fixed, with $\Sigma = C_0^{1/2}(I-S)C_0^{1/2}$, $S \in \SymHS(\Hcal)_{<I}$.
We first specify the tangent space to $\Tr(\Hcal,C_0)$ at $\Sigma$.
Let
\begin{align}
	\label{equation:def-SymHS}
\SymHS(\Hcal,C_0) = \{V = C_0^{1/2}XC_0^{1/2}, X \in \Sym(\Hcal) \cap \HS(\Hcal)\} \subset \SymHS_X(\Hcal,C_0).
\end{align}
\begin{proposition}
	\label{proposition:tangent-space}
	Let $\Sigma \in \Tr(\Hcal,C_0)$ be fixed.
	The tangent space of the Hilbert manifold $\Tr(\Hcal,C_0)$ at $\Sigma$
	is the following Hilbert subspace of $\SymHS_X(\Hcal,C_0)$
	\begin{align}
		T_{\Sigma}(\Tr(\Hcal,C_0)) = \SymHS(\Hcal,C_0) = \SymHS(\Hcal,\Sigma).
	\end{align}
\end{proposition}


We now give the formula for the Fisher metric on $\Gauss(\Hcal,\mu_0)$, using the abstract framework
in \cite{ay2015informationSufficient,ay2017informationGeometry,ay2018parametrized}.
Let $\mu = \Ncal(0,C) \in \Gauss(\Hcal,\mu_0)$, with $C = C_0^{1/2}(I-S)C_0^{1/2}$, $S \in \SymHS(\Hcal)_{<I}$.
Let $\frac{d\mu}{d\mu_0}$ denote its Radon-Nikodym density with respect to $\mu_0$, which has an explicit form, see Section \ref{section:computation-Fisher-metric}.
%
We consider the set
$\Gauss(\Hcal,\mu_0)$ to be parametrized by $S \in \SymHS(\Hcal)_{< I}$.
For a fixed $S \in \SymHS(\Hcal)_{<I}$, the Fisher metric at $S$ is defined to be, for $V_1,V_2 \in T_S(\SymHS(\Hcal)_{<I}) \cong\Sym(\Hcal) \cap \HS(\Hcal)$,
\begin{align}
	\label{equation:Fisher-metric-definition}
	g_S(V_1,V_2) = \int_{\Hcal}D\log\left\{\frac{d\mu}{d\mu_0}(x)\right\}(S)(V_1)D\log\left\{\frac{d\mu}{d\mu_0}(x)\right\}(S)(V_2)d\mu(x).
\end{align} 
Here $D$ denotes the Fr\'echet derivative and the quantity $g_S(V_1,V_2)$ is finite whenever
$D\log\left\{\frac{d\mu}{d\mu_0}\right\}(S)(V_j) \in \Lcal^2(\Hcal,\mu)$, $j=1,2$.
In the current setting, we show that this is always the case, as follows.

\begin{theorem}
[Fr\'echet logarithmic derivative of Radon-Nikodym density]
	\label{theorem:derivative-log-RN-general}
	Let $\mu = \mu(S) = \Ncal(0,C)$, $C = C_0^{1/2}(I-S)C_0^{1/2}$, $S \in \SymHS(\Hcal)_{< I}$.
	Let $S_{*} \in \SymHS(\Hcal)_{< I}$ be fixed but arbitrary, let $C_{*} = C_0^{1/2}(I-S_{*})C_0^{1/2}$ and $\mu_{*} = \Ncal(0,C_{*})$ be the corresponding Gaussian measure. 
	Consider the following three scenarios of increasing generality
	\begin{enumerate}
		\item $S_{*} \in \SymTr(\Hcal)_{< I}$, $V \in \Sym(\Hcal) \cap \Tr(\Hcal)$. In this case
		\begin{align}
			\label{equation:derivative-log-Radon-Nikodym-S-trace-class-1}
			&D\log\left\{\frac{d\mu}{d\mu_0}(x) \right\}(S_{*})(V) 
			\nonumber
			\\
			&= \frac{1}{2}\trace\left[(I-S_{*})^{-1}V\right] - \frac{1}{2}\la C_0^{-1/2}x, (I-S_{*})^{-1}V{(I-S_{*})^{-1}}C_0^{-1/2}x\ra.
		\end{align}
		Here
			$\la C_0^{-1/2}x, (I-S_{*})^{-1}V{(I-S_{*})^{-1}}C_0^{-1/2}x\ra 
			\doteq \lim_{N \approach \infty}\la C_0^{-1/2}P_Nx, (I-S_{*})^{-1}V{(I-S_{*})^{-1}}C_0^{-1/2}P_Nx\ra$, with the limit being taken in $\Lcal^2(\Hcal,\mu_0)$, $P_N = \sum_{k=1}^Ne_k \otimes e_k$, with $\{e_k\}_{k \in \Nbb}$ being the orthonormal eigenvectors of $C_0$,  is the orthogonal projection onto the subspace spanned by $\{e_k\}_{k=1}^N$.
		\item $S_{*} \in \SymHS(\Hcal)_{< I}$, $V \in \Sym(\Hcal) \cap \Tr(\Hcal)$.
		Let $\{S_k\}_{k \in \Nbb}\in \SymTr(\Hcal)_{<I}$ 
		be
		such that $\lim\limits_{k \approach \infty}||S_k-S_{*}||_{\HS} = 0$.
		Then $\forall V \in \Sym(\Hcal) \cap \Tr(\Hcal)$, with the limit being taken in $\Lcal^2(\Hcal,\mu_{*})$,
		\begin{align}
			\label{equation:derivative-log-Radon-Nikodym-S-HS-V-Tr-1}
			D\log\left\{\frac{d\mu}{d\mu_0}(x) \right\}(S_{*})(V) = \lim_{k \approach \infty} D\log\left\{\frac{d\mu}{d\mu_0}(x) \right\}(S_k)(V).
		\end{align}
		
		\item $S_{*} \in \SymHS(\Hcal)_{< I}$, $V \in \Sym(\Hcal) \cap \HS(\Hcal)$.
		Let $\{V_j\}_{j \in \Nbb} \in \Sym(\Hcal) \cap\Tr(\Hcal)$ be
		such that $\lim\limits_{j \approach \infty}||V_j - V||_{\HS} = 0$. Then,
	with the limit being taken in $\Lcal^2(\Hcal,\mu_{*})$,
		\begin{align}
			\label{equation:derivative-log-Radon-Nikodym-S-HS-V-HS-1}
			&D\log\left\{\frac{d\mu}{d\mu_0}(x) \right\}(S_{*})(V) = \lim_{j \approach \infty}D\log\left\{\frac{d\mu}{d\mu_0}(x) \right\}(S_{*})(V_j).
		\end{align}

	\end{enumerate}
	For any $S_{*} \in \SymHS(\Hcal)_{<I}$, for any $V \in \Sym(\Hcal) \cap \HS(\Hcal)$,
	\begin{align}
		\left\| D\log\left\{\frac{d\mu}{d\mu_0}(x) \right\}(S_{*})(V)\right\|^2_{\Lcal^2(\Hcal,\mu_{*})} = \frac{1}{2}||(I-S_{*})^{-1/2}V(I-S_{*})^{-1/2}||^2_{\HS}.
	\end{align}
	For any $S_{*} \in \SymHS(\Hcal)_{<I}$, for any pair $V,W \in \Sym(\Hcal) \cap \HS(\Hcal)$,
\begin{align}
	\label{equation:inner-derivative-log-RN-1}
	&\left\la D\log\left\{\frac{d\mu}{d\mu_0}(x) \right\}(S_{*})(V), D\log\left\{\frac{d\mu}{d\mu_0}(x) \right\}(S_{*})(W)\right\ra_{\Lcal^2(\Hcal,\mu_{*})}
	\nonumber
	\\
	&=\frac{1}{2}\trace[(I-S_{*})^{-1}V(I-S_{*})^{-1}W].
\end{align}

\end{theorem}

Equivalently, the Fisher metric can be obtained by taking the second derivative of 
the Kullback-Leibler (KL) divergence between two equivalent Gaussian measures on $\Hcal$. Let $\mu = \Ncal(m_1,Q)$ and $\nu = \Ncal(m_2,R)$ on $\Hcal$. If $\mu \perp \nu$, then
$\KL(\nu|| \mu) = \infty$. If $\mu \sim \nu$, then we have the following result. 
\begin{theorem}
	[\cite{Minh:2020regularizedDiv}]
	\label{theorem:KL-gaussian}
	Let $\mu = \Ncal(m_1, Q)$, $\nu = \Ncal(m_2, R)$, with $\ker(Q) = \ker{R} = \{0\}$, and
	$\mu \sim \nu$.
	Let $S \in \HS(\Hcal)\cap \Sym(\Hcal)$, $I-S > 0$, 
	be such that $R = Q^{1/2}(I-S)Q^{1/2}$, then
	\begin{align}
		\label{equation:KL-gaussian}
		\KL(\nu ||\mu) = \frac{1}{2}||Q^{-1/2}(m_2-m_1)||^2 -\frac{1}{2}\log\dettwo(I-S).
	\end{align}
\end{theorem}
Here $\dettwo$ is the Hilbert-Carleman determinant, see e.g. \cite{Simon:1977}, with $\dettwo(I+A) = \det[(I+A)\exp(-A)]$
for $A \in \HS(\Hcal)$, where $\det$ is the Fredholm determinant,
given by $\det(I+A) = \prod_{j=1}^{\infty}(1+\lambda_k(A))$, $A \in \Tr(\Hcal)$, $\{\lambda_k(A)\}_{k=1}^{\infty}$ being the eigenvalues of $A$.

Applying either Theorem \ref{theorem:derivative-log-RN-general} or \ref{theorem:KL-gaussian}, we obtain the following 
the explicit expression for the Fisher metric.

\begin{theorem}
	[Riemannian metric]
	\label{theorem:Riemannian-metric}
	The Fisher-Rao metric
	on $\Gauss(\Hcal,\mu_0)$ is given as follows. Let $S \in \SymHS(\Hcal)_{<I}$ be fixed. Then
	\begin{align}
		\label{equation:Fisher-Rao-metric-infinite}
		g_S(V_1,V_2) = \frac{1}{2}\trace[(I-S)^{-1}V_1(I-S)^{-1}V_2], \;\;\; V_1,V_2 \in \Sym(\Hcal)\cap \HS(\Hcal).
	\end{align}
	The corresponding Riemannian metric on $\Tr(\Hcal,C_0)$ is given as follows.
	Let $\Sigma \in \Tr(\Hcal,C_0)$ be fixed.
	For $A_1,A_2 \in T_{\Sigma}(\Tr(\Hcal,C_0)) \cong \SymHS(\Hcal,C_0) = \SymHS(\Hcal,\Sigma)$,
	\begin{align}
		\la A_1,A_2\ra_{\Sigma} &= \frac{1}{2}\la \Sigma^{-1/2}A_1\Sigma^{-1/2}, \Sigma^{-1/2}A_2\Sigma^{-1/2}\ra_{\HS}
		\label{equation:Riemannian-metric-Tr-H-C0-1}
		\\
		& = \frac{1}{2}\trace(\Sigma^{-1/2}A_1\Sigma^{-1}A_2\Sigma^{-1/2}).
		\label{equation:Riemannian-metric-Tr-H-C0-2}
	\end{align}
\end{theorem}

\begin{remark}
	In the finite-dimensional setting, if we characterize the Gaussian density $P_{\theta}$ by $S = S(\theta)$, instead of $\Sigma = \Sigma(\theta)$,
	with $S(\theta)_{11} = \theta^{1}$, $S(\theta)_{12} = \theta^{2}$, $\ldots, S(\theta)_{22} = \theta^{n+1}$, 
	$\ldots, S(\theta)_{nn} = \theta^{\frac{n(n+1)}{2}}$, then we have the following expression for the Fisher metric
	\begin{align}
		g_{ij}(\theta) = g_{ij}(S(\theta)) = \frac{1}{2}\trace[(I-S)^{-1}\partial_{\theta_i}S(I-S)^{-1}\partial_{\theta_j}S].
	\end{align}
	By identifying the basis $\{\partial_{\theta_i}\}$ of $T_{P_{\theta}}(\Scal)$ with a basis for $\Sym(n)$, we
	obtain Eq.\eqref{equation:Fisher-Rao-metric-infinite} for $V_1,V_2 \in \Sym(n)$.
	
	We note that Eq.\eqref{equation:metric-Sym++(n)-1/2} in the finite-dimensional setting directly generalizes to Eqs.\eqref{equation:Riemannian-metric-Tr-H-C0-1} and \eqref{equation:Riemannian-metric-Tr-H-C0-2}
	in the current setting of equivalent Gaussian measures on Hilbert space. However,
	Eq.\eqref{equation:metric-Sym++(n)-1} is generally not well-defined in this setting, since
	for $A \in\SymHS(\Hcal,\Sigma)$, $\Sigma^{-1}A$ is not necessarily bounded.
\end{remark}

\begin{theorem}
	[\textbf{Riemannian structures on $\Tr(\Hcal,C_0)$}]
	Under the Riemannian metric in Eq.\eqref{equation:Riemannian-metric-Tr-H-C0-2}, $\Tr(\Hcal,C_0)$ is an infinite-dimensional Cartan-Hadamard manifold.
	Let $P \in \Tr(\Hcal,C_0)$, let $X,Y,Z$ be smooth vector fields on $\Tr(\Hcal,C_0)$,
	with $X_P,Y_P,Z_P \in T_P(\Tr(\Hcal,C_0))\cong \SymHS(\Hcal,C_0) =  \SymHS(\Hcal,P)$.
	\begin{enumerate}
		\item The Levi-Civita connection is given by
		\begin{align}
				\label{equation:Levi-Civita-infnite}
			(\nabla_{X}Y)_P = D(Y)(P)[X_P] - \frac{1}{2}[X_PP^{-1}Y_P + Y_PP^{-1}X_P].
		\end{align}
		Here $D(Y)(P)$ denotes the Fr\'echet derivative of $Y$ at $P$ in the open subset $\Tr(\Hcal,C_0)$ of the Hilbert space $\SymHS_X(\Hcal,C_0)$.
			\item Let $X_P = P^{1/2}\tilde{X}_PP^{1/2}$, $Y_P = P^{1/2}\tilde{Y}_PP^{1/2}$, $Z_P = P^{1/2}\tilde{Z}_PP^{1/2}$, for
		$\tilde{X}_P$,$\tilde{Y}_P$,$\tilde{Z}_P$ $\in \Sym(\Hcal)\cap\HS(\Hcal)$.
		The Riemannian curvature tensor is given by
		\begin{align}
			\label{equation:Riemann-curvature-tensor}
			[R(X,Y)Z](P) = -\frac{1}{4}P^{1/2}[[\tilde{X}_P, \tilde{Y}_P],\tilde{Z}_P]P^{1/2}.
		\end{align}
		Here $[,]$ denotes the operator commutator
		$[A,B] = AB-BA$.
		\item The sectional curvature is everywhere nonpositive
		%
		\begin{align}
			S_P(X_P,Y_P) =
			-\frac{\trace[(\tilde{X}_P)^2(\tilde{Y}_P)^2 - (\tilde{X}_P\tilde{Y}_P)^2]}{\trace(\tilde{X}_P)^2\trace(\tilde{Y}_P)^2 -\trace(\tilde{X}_P\tilde{Y}_P)^2} \leq 0,
			%
		\end{align}
		where $X_P,Y_P$ are any two linearly independent operators in $\SymHS(\Hcal,P)$.
		\item There is a unique geodesic connecting any pair $A\in \Tr(\Hcal,C_0)$, $B = A^{1/2}(I-S)A^{1/2} \in \Tr(\Hcal,C_0)$, given by
		\begin{align}
			\gamma_{AB}(t) = A^{1/2}\exp[t\log(I-S)]A^{1/2}.
		\end{align}
	Equivalently, with $V \in T_A(\Tr(\Hcal,C_0)) = \SymHS(\Hcal,A)$,
	\begin{align}
		\gamma_V(t) = A^{1/2}\exp[tA^{-1/2}VA^{-1/2}]A^{1/2}, \;\; \gamma_V(0) = A, \;\; \dot{\gamma}_V(0) = V.
	\end{align}
	The exponential map $\Exp_{A}: T_{A}(\Tr(\Hcal,C_0)) \mapto \Tr(\Hcal,C_0)$ is given by
	\begin{align}
		\Exp_{A}(V) = \gamma_V(1) = A^{1/2}\exp(A^{-1/2}VA^{-1/2})A^{1/2}.
	\end{align}
	Its inverse, the logarithm map $\Log_A: \Tr(\Hcal,C_0) \mapto T_{A}(\Tr(\Hcal,C_0))$ is 
	\begin{align}
		\Log_A(B) =  A^{1/2}\log(A^{-1/2}BA^{-1/2})A^{1/2}, \;\; B \in  \Tr(\Hcal,C_0).
	\end{align}
	\item The length of geodesic $\gamma_{AB}(t)$ is the Riemannian distance between $A$ and $B$, or equivalently, the Fisher-Rao distance
		between $\Ncal(0,A)$ and $\Ncal(0,B)$\footnote{There is a typo in the corresponding result in Theorem 4 in \cite{Minh:GSI2023}, with the constant in the Fisher-Rao distance being $1/2$. The correct constant is $1/\sqrt{2}$ as stated here.},
		\begin{align}
			\label{equation:Fisher-Rao-distance-infinite}
			d(A,B) = \frac{1}{\sqrt{2}}||\log(A^{-1/2}BA^{-1/2})||_{\HS} =\frac{1}{\sqrt{2}} ||\log(I-S)||_{\HS}.
		\end{align}
	\end{enumerate}
		%
\end{theorem}
\begin{remark}
	The Levi-Civita connection in Eq.\eqref{equation:Levi-Civita-infnite} is formulated using the fact that
	$\Tr(\Hcal, C_0)$ is an open subset in the Hilbert space $\SymHS_X(\Hcal,C_0)$ under the inner product and norm defined 
	in Eqs.\eqref{equation:HSX-H-C0-inner-product} and \eqref{equation:HSX-H-C0-norm}.
\end{remark}

\begin{remark}
	In the finite-dimensional setting, the Riemannian curvature tensor in Eq.\eqref{equation:Riemann-curvature-tensor} is equivalent to
	\begin{align}
		\label{equation:Riemann-curvature-tensor-finite}
		[R(X,Y)Z](P) = -\frac{1}{4}P[[P^{-1}X_P, P^{-1}Y_P], P^{-1}Z_P].
	\end{align}
	In \cite{lang2012fundamentals} (Theorem 3.9), this formula was given in the case $P=I$, the identity operator, 
	but without the factor $\frac{1}{4}$.
	This formula is also valid on the manifold $\PC_2(\Hcal)$ under the affine-invariant Riemannian metric defined in Eq.\eqref{equation:affine-Riemannian-metric-infinite}
	(see \cite{Larotonda:2007}, Eq.(5)). It is generally not valid in the setting of equivalent infinite-dimensional Gaussian measures, however, since
	$P^{-1}X_P = P^{-1/2}\tilde{X}_PP^{1/2}$, $\tilde{X}_P \in \Sym(\Hcal) \cap\HS(\Hcal)$, is not necessarily bounded.
\end{remark}
{\bf Connection with the Riemannian geometry of positive definite unitized Hilbert-Schmidt operators}.
The following shows that the Fisher-Rao distance in Eq.\eqref{equation:Fisher-Rao-distance-infinite}
can be obtained from the Riemannian distance between positive definite unitized Hilbert-Schmidt operators in Eq.\eqref{equation:d-aiHS}
as $\gamma=\mu \approach 0$. 
\begin{theorem}
	\label{theorem:limit-affineLogHS}
	Let $A,B \in \Tr(\Hcal,C_0)$ with $B = A^{1/2}(I-S)A^{1/2}$. Then
	\begin{align}
		\lim_{\gamma \approach 0^{+}}||\log[(A+ \gamma I)^{-1/2}(B+\gamma I)(A + \gamma I)^{-1/2}]||_{\HS}
		= ||\log(I-S)||_{\HS}.
	\end{align}
\end{theorem}
Thus, for $A,B \in \Sym^{+}(\Hcal) \cap \Tr(\Hcal)$ and $\gamma \in \R, \gamma > 0$, Eq.\eqref{equation:d-aiHS} can be considered as a regularized version of Eq.\eqref{equation:Fisher-Rao-distance-infinite},
up to the multiplicative factor $\frac{1}{\sqrt{2}}$.
The advantage of the distance in Eq.\eqref{equation:d-aiHS} is that it is always finite for all pairs $A,B \in \Sym^{+}(\Hcal) \cap \Tr(\Hcal)$ and all $\gamma \in \R, \gamma > 0$. Furthermore,
Eq.\eqref{equation:d-aiHS} can be applied for estimating the distance between the infinite-dimensional Gaussian measures corresponding to measurable Gaussian processes with squared integrable paths, using distances between the corresponding finite-dimensional Gaussian measures, with explicit finite sample complexity, see \cite{Minh:2021RiemannianEstimation}.

\section{Proofs of main results}
\label{section:proofs}

\subsection{Preliminary technical results}
\label{section:technical-results}
We first note the following technical results. Let $\Sigma = C_0^{1/2}(I-S)C_0^{1/2}$ for some $S \in \SymHS(\Hcal)_{<I}$. Then $\Ncal(0,\Sigma) \sim \Ncal(0,C_0)$.
By symmetry, $C_0 = \Sigma^{1/2}(I-T)\Sigma^{1/2}$ for some $T \in \SymHS(\Hcal)_{<I}$.


\begin{lemma}
	\label{lemma:range-Sigma}
Let $\Sigma = C_0^{1/2}(I-S)C_0^{1/2}$ for some $S \in \SymHS(\Hcal)_{<I}$.
Then $\range(\Sigma^{1/2}) = \range(C_0^{1/2})$ and $\ker(\Sigma) =\ker(C_0) = \{0\}$.
Thus if $\mu=\Ncal(0,C) \sim \mu_0 = \Ncal(0,C_0)$, then $\mu$ is also necessarily nondegenerate.
\end{lemma}
\begin{proof}
As shown in \cite{Fillmore:1971Operator}, as a consequence of their Theorem 2.1, for any bounded operator $A \in \Lcal(\Hcal)$ we have
$\range(A) = \range(AA^{*})^{1/2}$.
Let $A = C_0^{1/2}(I-S)^{1/2}$, then $\Sigma = AA^{*}$ and $\range(\Sigma^{1/2}) = \range(AA^{*})^{1/2} = \range(A) = \range(C_0^{1/2})$, since
$I-S$ is invertible. Thus also $\ker(\Sigma) = \ker(C_0) = \{0\}$.
\qed
\end{proof}

\begin{lemma}
\label{lemma:C0-inverse-Sigma-1}
Assume that $\Sigma = C_0^{1/2}(I-S)C_0^{1/2}$, $S \in \SymHS(\Hcal)_{<I}$. 
Then $C_0^{-1/2}\Sigma C_0^{-1/2}$ is a well-defined, bounded operator on $\Hcal$ and $C_0^{-1/2}\Sigma C_0^{-1/2} = I-S$.
\end{lemma}
\begin{proof}
For $x \in \range(C_0^{1/2})$, $C_0^{-1/2}x$ is well-defined, and 
\begin{align*}
C_0^{-1/2}\Sigma C_0^{-1/2}x = C_0^{-1/2}(C_0^{1/2}(I-S)C_0^{1/2})C_0^{-1/2}x = (I-S)x.
\end{align*}
Since $\ker(C_0) = \{0\}$, $\range(C_0^{1/2})$ is dense in $\Hcal$. For each $x \in \Hcal$, let $\{x_n\}_{n \in \Nbb}$ be a Cauchy sequence in $\range(C_0^{1/2})$,
with $\lim_{n \approach \infty}||x_n - x||=0$, then
\begin{align*}
||C_0^{-1/2}\Sigma C_0^{-1/2}x_n - C_0^{-1/2}\Sigma C_0^{-1/2}x_m|| &= ||(I-S)(x_n-x_m)||
\\
& \leq ||I-S||\;||x_n - x_m||.
\end{align*}
Thus $\{C_0^{-1/2}\Sigma C_0^{-1/2}x_n\}$ is also a Cauchy sequence in $\Hcal$, hence must necessarily converge to $(I-S)x$ $\forall x \in \Hcal$.
\qed
\end{proof}

\begin{lemma}
\label{lemma:Sigma-inverse-C0-1}
Assume that $\Sigma = C_0^{1/2}(I-S)C_0^{1/2}$, $S \in \SymHS(\Hcal)_{<I}$. Then
$\Sigma^{-1/2}C_0\Sigma^{-1/2}$ is a well-defined, bounded operator on $\Hcal$, 
with $\Sigma^{-1/2}C_0\Sigma^{-1/2} = I-T$, for some $T \in \SymHS(\Hcal)_{<I}$.	
\end{lemma}
\begin{proof}
	By assumption, we have $\Ncal(0,C_0) \sim \Ncal(0,\Sigma)$. Thus
by symmetry, we have $C_0 = \Sigma^{1/2}(I-T)\Sigma^{1/2}$ for some $T \in \SymHS(\Hcal)_{<I}$. The conclusion then follows from Lemma \ref{lemma:C0-inverse-Sigma-1}.
\qed
\end{proof}

\begin{lemma}
	\label{lemma:C0-Sigma-1/2-inverse-bounded}
	Assume that
	$\Sigma = C_0^{1/2}(I-S)C_0^{1/2}$, $S \in \SymHS(\Hcal)_{<I}$. Then $\Sigma^{-1/2}C_0^{1/2}$, $C_0^{-1/2}\Sigma^{1/2}$, $\Sigma^{1/2}C_0^{-1/2}$, $C_0^{1/2}\Sigma^{-1/2}$, $C_0^{1/2}\Sigma^{-1}C_0^{1/2}$, $\Sigma^{1/2}C_0^{-1}\Sigma^{1/2} \in \Lcal(\Hcal)$,
	with $(\Sigma^{-1/2}C_0^{1/2})^{*} = C_0^{1/2}\Sigma^{-1/2}$, $(\Sigma^{1/2}C_0^{-1/2})^{*} = C_0^{-1/2}\Sigma^{1/2}$. Furthermore, let $T \in \SymHS(\Hcal)_{<I}$
	such that $C_0 = \Sigma^{1/2}(I-T)\Sigma^{1/2}$, then
	\begin{align}
		&C_0^{1/2}\Sigma^{-1}C_0^{1/2} = (I-S)^{-1},
		\;\;\Sigma^{1/2}C_0^{-1}\Sigma^{1/2} = (I-T)^{-1},
		\\
		&||\Sigma^{1/2}C_0^{-1/2}||\leq\sqrt{||I-S||}, \;\; ||C_0^{1/2}\Sigma^{-1/2}||\leq \sqrt{||I-T||},
		\\
		&||\Sigma^{-1/2}C_0^{1/2}|| \leq \sqrt{||(I-S)^{-1}||},
		\;\;
		||C_0^{-1/2}\Sigma^{1/2}||\leq \sqrt{||(I-T)^{-1}||}.
	\end{align}
\end{lemma}
\begin{proof}
	By the assumption that $\ker(C_0) = \{0\}$, we have that $\range(C_0^{1/2})$ is dense in $\Hcal$. For $x \in \range(C_0^{1/2})$, the following is well-defined
	by Lemma \ref{lemma:C0-inverse-Sigma-1}
	\begin{align*}
		&||\Sigma^{1/2}C_0^{-1/2}x||^2 = \la \Sigma^{1/2}C_0^{-1/2}x, \Sigma^{1/2}C_0^{-1/2}x\ra = \la x, C_0^{-1/2}\Sigma C_0^{-1/2}x\ra 
		\\
		&= \la x, (I-S)x\ra \leq ||I-S||\;||x||^2.
	\end{align*}
	For $x \in \Hcal$, let $\{x_n\}_{n\in \Nbb}$ be a Cauchy sequence in $\range(C_0^{1/2})$,
	with $\lim_{n \approach \infty}||x_n-x||=0$.
	Then
	\begin{align*}
		||\Sigma^{1/2}C_0^{-1/2}x_n - \Sigma^{1/2}C_0^{-1/2}x_m|| = ||\Sigma^{1/2}C_0^{-1/2}(x_n-x_m)|| \leq \sqrt{||I-S||}||x_n -x_m||.
	\end{align*}
	Thus $\{\Sigma^{1/2}C_0^{-1/2}x_n\}_{n \in \Nbb}$ is also a Cauchy sequence in $\Hcal$ that must converge to a unique limit, which is necessarily $\Sigma^{1/2}C_0^{-1/2}x$.
	Thus $\Sigma^{1/2}C_0^{-1/2}x$ is well-defined $\forall x \in \Hcal$
	and $\Sigma^{1/2}C_0^{-1/2} \in \Lcal(\Hcal)$, with $||\Sigma^{1/2}C_0^{-1/2}|| \leq \sqrt{||I-S||}$.
	
	By symmetry, $C_0^{1/2}\Sigma^{-1/2} \in \Lcal(\Hcal)$, with $||C_0^{1/2}\Sigma^{-1/2}||\leq \sqrt{||I-T||}$.
	Furthermore, it is invertible, with $(C_0^{1/2}\Sigma^{-1/2})^{-1} = \Sigma^{1/2}C_0^{-1/2} \in \Lcal(\Hcal)$, since $(C_0^{1/2}\Sigma^{-1/2})(\Sigma^{1/2}C_0^{-1/2}) = (\Sigma^{1/2}C_0^{-1/2})(C_0^{1/2}\Sigma^{-1/2}) = I$.
	
	By Lemma \ref{lemma:C0-inverse-Sigma-1}, $C_0^{-1/2}\Sigma C_0^{-1/2} = (I-S) \in \Lcal(\Hcal)$.
	It follows that $C_0^{-1/2}\Sigma^{1/2} = (I-S)(\Sigma^{1/2}C_0^{-1/2})^{-1} = (I-S)C_0^{1/2}\Sigma^{-1/2} \in \Lcal(\Hcal)$.
	
	By symmetry, we have $\Sigma^{-1/2}C_0^{1/2} \in \Lcal(\Hcal)$, with $(C_0^{-1/2}\Sigma^{1/2})^{-1} = 
	\Sigma^{-1/2}C_0^{1/2} \in \Lcal(\Hcal)$. It follows that
	\begin{align*}
		(I-S)^{-1} = C_0^{1/2}\Sigma^{-1/2}(C_0^{-1/2}\Sigma^{1/2})^{-1} = C_0^{1/2}\Sigma^{-1}C_0^{1/2}.
	\end{align*}
	Thus $\forall x \in \Hcal$,
	\begin{align*}
		&	||\Sigma^{-1/2}C_0^{1/2}x||^2 = \la \Sigma^{-1/2}C_0^{1/2}x, \Sigma^{-1/2}C_0^{1/2}x\ra
		= \la x, C_0^{1/2}\Sigma^{-1}C_0^{1/2}x\ra
		\\
		&	 = \la x, (I-S)^{-1}x\ra \leq ||(I-S)^{-1}||\;||x||^2
	\end{align*}
	from which it follows that $||\Sigma^{-1/2}C_0^{1/2}||\leq \sqrt{||(I-S)^{-1}||}$.
	By symmetry,
	\begin{align*}
		(I-T)^{-1} = \Sigma^{1/2}C_0^{-1}\Sigma^{1/2} 
	\end{align*}
	and $||C_0^{-1/2}\Sigma^{1/2}|| \leq \sqrt{||(I-T)^{-1}||}$.
	
	To show that $(\Sigma^{-1/2}C_0^{1/2})^{*} = C_0^{1/2}\Sigma^{-1/2}$, we note that for $y \in \range(\Sigma^{1/2})$, $\forall x \in \Hcal$,
	\begin{align*}
	\la \Sigma^{-1/2}C_0^{1/2}x,y\ra = \la C_0^{1/2}x, \Sigma^{-1/2}y\ra = \la x, C_0^{1/2}\Sigma^{-1/2}y\ra. 
	\end{align*}
Since $\range(\Sigma^{1/2})$ is dense in $\Hcal$, using a limiting argument as before, we have that
$\la \Sigma^{-1/2}C_0^{1/2}x,y\ra = \la x, C_0^{1/2}\Sigma^{-1/2}y\ra$ $\forall x, y \in \Hcal$,
showing that $(\Sigma^{-1/2}C_0^{1/2})^{*} = C_0^{1/2}\Sigma^{-1/2}$.
Similarly, $(\Sigma^{1/2}C_0^{-1/2})^{*} = C_0^{-1/2}\Sigma^{1/2}$.
	\qed
\end{proof}

\begin{corollary}
	\label{corollary:relation-T-S}
	Assume that $\Sigma = C_0^{1/2}(I-S)C_0^{1/2}$, $S \in \SymHS(\Hcal)_{<I}$.
	Then $C_0 = \Sigma^{1/2}(I-T)\Sigma^{1/2}$, for some $T \in \SymHS(\Hcal)_{<I}$. 
	The operators $T$ and $S$ are related as follows. Let $A = \Sigma^{-1/2}C_0^{1/2}$,
	then $I-T = AA^{*}$ and $(I-S)^{-1} = A^{*}A$.
	In particular, the nonzero eigenvalues of $I-T$ and $(I-S)^{-1}$ are the same.
	Consequently, $\dettwo(I-T) = \dettwo[(I-S)^{-1}]$. 
\end{corollary}
\begin{proof}
	This follows from Lemma \ref{lemma:C0-Sigma-1/2-inverse-bounded}.
	\qed
\end{proof}


\begin{lemma}
Let $\Sigma = C_0^{1/2}(I-S)C_0^{1/2}$, $S \in \SymHS(\Hcal)_{< I}$. Let $A = C_0^{1/2}TC_0^{1/2}$, $T \in \Sym(\Hcal) \cap \HS(\Hcal)$. Then
$\Sigma^{-1/2}A\Sigma^{-1/2} \in \Sym(\Hcal) \cap \HS(\Hcal)$.
\end{lemma}
\begin{proof}
	By Lemma \ref{lemma:C0-Sigma-1/2-inverse-bounded}, we have $\Sigma^{-1/2}C_0^{1/2}\in\Lcal(\Hcal)$, $C_0^{1/2}\Sigma^{-1/2} \in \Lcal(\Hcal)$,
	with $(\Sigma^{-1/2}C_0^{1/2})^{*} = C_0^{1/2}\Sigma^{-1/2}$.
	Thus $\Sigma^{-1/2}A\Sigma^{-1/2} = \Sigma^{-1/2}C_0^{1/2}TC_0^{1/2}\Sigma^{-1/2} \in \Sym(\Hcal) \cap \HS(\Hcal)$.
\qed
\end{proof}

\begin{lemma}
\label{lemma:switch-C0-Sigma}
Assume that 
$\Sigma = C_0^{1/2}(I-S)C_0^{1/2}$, $S \in \SymHS(\Hcal)_{<I}$.
Then $A = C_0^{1/2}B_1C_0^{1/2}$, $B_1 \in \Sym(\Hcal) \cap \HS(\Hcal)$, if and only if
$\exists B_2 \in \Sym(\Hcal) \cap \HS(\Hcal)$ such that $A = \Sigma^{1/2}B_2\Sigma^{1/2}$.
In other words, $\SymHS(\Hcal,C_0) = \SymHS(\Hcal,\Sigma)$.
\end{lemma}
\begin{proof}
	By assumption, $\Ncal(0,\Sigma) \sim \Ncal(0,C_0)$, so that by symmetry, $\exists T \in \SymHS(\Hcal)_{<I}$ such that $C_0 = \Sigma^{1/2}(I-T)\Sigma^{1/2}$.
	Thus if $A = C_0^{1/2}B_1C_0^{1/2}$, then
	\begin{align*}
	A =  (\Sigma^{1/2}(I-T)\Sigma^{1/2})^{1/2}B_1 (\Sigma^{1/2}(I-T)\Sigma^{1/2})^{1/2}.
	\end{align*}
Since $\ker(\Sigma^{1/2}) = \{0\}$ and $I-T > 0$, the operator $(I-T)^{1/2}\Sigma^{1/2}$ admits the polar decomposition
\begin{align*}
(I-T)^{1/2}\Sigma^{1/2} = U(\Sigma^{1/2}(I-T)\Sigma^{1/2})^{1/2},
\end{align*}
where $U$ is unitary. It follows that
\begin{align*}
(\Sigma^{1/2}(I-T)\Sigma^{1/2})^{1/2} = U^{*}(I-T)^{1/2}\Sigma^{1/2} = \Sigma^{1/2}(I-T)^{1/2}U,
\end{align*}
where the second equality follows since $(\Sigma^{1/2}(I-T)\Sigma^{1/2})^{1/2} \in \Sym(\Hcal)$.
Thus
\begin{align*}
	A = \Sigma^{1/2}(I-T)^{1/2}UB_1U^{*}(I-T)^{1/2}\Sigma^{1/2} = \Sigma^{1/2}B_2\Sigma^{1/2},
\end{align*}
where $B_2 = (I-T)^{1/2}UB_1U^{*}(I-T)^{1/2} \in \Sym(\Hcal) \cap \HS(\Hcal)$. The converse follows from symmetry.
\qed
\end{proof}

\subsection{Proofs for the Riemannian metric}
\label{section:Riemannian-metric}

\begin{proof}
	[\textbf{of Lemma \ref{lemma:Hilbert-manifold-SymHS<I}}]
	The set $\SymHS(\Hcal)_{<I}$ is an open subset in the Hilbert space $\Sym(\Hcal) \cap \HS(\Hcal)$ and thus is a Hilbert manifold.
	For $S \in \SymHS(\Hcal)_{<I}$, let $A:\R\mapto \SymHS(\Hcal)_{<I}$ be a differentiable curve, with $A(0) = S$. Then, by definition of the derivative, $\dot{A}(t) \in \Sym(\Hcal) \cap \HS(\Hcal)$ $\forall t \in \R$.	
	
	Conversely, let $S \in \SymHS(\Hcal)_{<I}$ be fixed. Let $V  \in \Sym(\Hcal) \cap \HS(\Hcal)$. 
	Define the curve $B: \R \mapto \Sym(\Hcal) \cap \HS(\Hcal)$ by $B(t) =\sum_{k=1}^{\infty}\frac{(-1)^{k-1}t^k}{k!}V^k 
	= I - \exp(-tV)$.
	We note that $B(t) \in \Sym(\Hcal) \cap \HS(\Hcal)$ since $\Sym(\Hcal) \cap \HS(\Hcal)$ is a Banach algebra.
	Define $A(t) = S + B(t) = I+S - \exp(-tV)$, 
	then
	$A(t)$ is smooth, with $A(t) \in \Sym(\Hcal) \cap \HS(\Hcal)$, $A(0) = S$, $\dot{A}(0) = V$. 
	Since $I-S > 0$, $\exists M_S \in \R$, $M_S > 0$ such that $\la x, (I-S)x\ra \geq M_S||x||^2$ $\forall x \in \Hcal$.
	We have
	\begin{align*}
		\la x, B(t)x\ra \leq \sum_{k=1}^{\infty}\frac{|t|^k||V||^k}{k!}||x||^2 = [\exp(|t|\;||V||) - 1]||x||^2.
	\end{align*}
	Let $\varep = \frac{\log(1+M_S)}{||V||}$, then $M_S - [\exp(|t|\;||V||) - 1] > 0$ $\forall |t| < \varep$ and $\forall x \in \Hcal$,
	\begin{align*}
		\la x, (I-A(t))x\ra = \la x, (I-S-B(t))x\ra \geq (M_S - [\exp(|t|\;||V||) - 1])||x||^2.
	\end{align*}
	This shows that $I-A(t) > 0$ $\forall |t| < \varep$. Thus $A(t) \in \SymHS(\Hcal)_{< I}$
	$\forall |t| < \varep$. Hence $T_S(\SymHS(\Hcal)_{<I}) \cong \Sym(\Hcal) \cap \HS(\Hcal)$
	$\forall S \in \SymHS(\Hcal)_{< I}$.
	\qed
\end{proof}

\begin{proof}
	[\textbf{of Proposition \ref{proposition:tangent-space}}]
	By Lemma \ref{lemma:switch-C0-Sigma}, $V = C_0^{1/2}XC_0^{1/2}$, $X \in \Sym(\Hcal) \cap \HS(\Hcal)$, if and only if
	$V = \Sigma^{1/2}Y\Sigma^{1/2}$ for some $Y \in \Sym(\Hcal) \cap \HS(\Hcal)$, so the two expressions are identical.
	Let $A:\R \mapto \Tr(\Hcal, C_0)$ be a differentiable curve, then it necessarily has the form
	$A(t) = C_0^{1/2}(I-B(t))C_0^{1/2}$, 
	with $B(t) \in \SymHS(\Hcal)_{<I}$.
	Thus $\dot{A}(t) = -C_0^{1/2}\dot{B}(t)C_0^{1/2}$, with $\dot{B}(t) \in \Sym(\Hcal) \cap \HS(\Hcal)$.
	
	Conversely, 
	let $V = \Sigma^{1/2}Y\Sigma^{1/2}$, $Y \in \Sym(\Hcal) \cap \HS(\Hcal)$.
	Consider the following curve $A: \R \mapto \Tr(\Hcal,C_0)$, defined by
	\begin{align*}
		A(t) = \Sigma^{1/2}\exp(t\Sigma^{-1/2}V\Sigma^{-1/2})\Sigma^{1/2} = \Sigma^{1/2}\exp(tY)\Sigma^{1/2}.
	\end{align*}
	Then $A(t) \in \Tr(\Hcal,C_0)$ $\forall t \in \R$, since $\exp(tY) = 
	I + \sum_{k=1}^{\infty}\frac{t^k}{k!}Y^k > 0$ $\forall t \in \R$, 
	with $\sum_{k=1}^{\infty}\frac{t^k}{k!}Y^k \in \Sym(\Hcal)\cap\HS(\Hcal)$, since
	$\HS(\Hcal)$ is a Banach algebra.
	Clearly $A(t)$ is a smooth curve, with $A(0) = \Sigma$, $\dot{A}(0) = \Sigma^{1/2}Y\Sigma^{1/2} = V$. Thus
	$V \in T_{\Sigma}(\Tr(\Hcal,C_0))$.
	\qed
\end{proof}

The following is a special case of Lemma 4 in \cite{Minh:2019AlphaBeta}.
\begin{lemma}
	\label{lemma:derivative-logdet-Fredholm}
	Let $\Omega_1 = \{X \in \Sym(\Hcal) \cap \Tr(\Hcal): I + X > 0\}$. Let $A_0 \in \Omega_1$ be fixed.
	Then $\forall A \in \Sym(\Hcal) \cap \Tr(\Hcal)$,
	\begin{align}
		D\log\det(I+A_0)(A) = \trace[(I+A_0)^{-1}A].
	\end{align}
\end{lemma}

\begin{lemma}
	\label{lemma:Frechet-derivative-logdet2}
	Let $\Omega_2 = \{X \in \Sym(\Hcal) \cap \HS(\Hcal): I + X > 0\}$.
	Let $f: \Omega_2 \mapto \R$ 
	be defined by $f(X) = \log\dettwo(I+X)$. Then $Df(X_0):
	\Sym(\Hcal)\cap \HS(\Hcal) \mapto \R$ is given by, $\forall X_0 \in \Omega_2$,
	\begin{align}
		Df(X_0)(X) = - \trace[(I+X_0)^{-1}X_0X], \;\;\;\forall X \in \Sym(\Hcal) \cap \HS(\Hcal).
	\end{align}
\end{lemma}
\begin{proof}
	(i) For $X \in \Omega_1 = \{X \in \Sym(\Hcal) \cap \Tr(\Hcal): I + X > 0\}$,
	\begin{align*}
		&\dettwo(I+X) = \det[(I+X)\exp(-X)] = \det(I+X)\exp(-\trace(X)), 
		\\
		&\log\dettwo(I+X) = \log\det(I+X) - \trace(X).
	\end{align*}
	It follows that $\forall X_0 \in \Omega_1$,
	 $\forall X \in \Sym(\Hcal) \cap\Tr(\Hcal)$, by Lemma \ref{lemma:derivative-logdet-Fredholm},
	\begin{align*}
		D\log\dettwo(I+X_0)(X) = \trace[(I+X_0)^{-1}X] - \trace(X) = -\trace[(I+X_0)^{-1}X_0X].
	\end{align*}
	(ii) Consider the general case $X \in \Omega_2 = \{X \in \Sym(\Hcal) \cap \HS(\Hcal): I + X > 0\}$.
	By definition,
	\begin{align*}
		\log\dettwo(I+X) = \log\det[(I+X)\exp(-X)].
	\end{align*}
	Expanding $(I+X)\exp(-X)$ gives
	\begin{align*}
		&(I+X)\exp(-X) = (I+X)\left(I - X + \frac{X^2}{2!} - \frac{X^3}{3!} + \frac{X^4}{4!} - \cdots\right)
		\\
		& = I - X^2\left(1-\frac{1}{2!}\right) + X^3\left(\frac{1}{2!} - \frac{1}{3!}\right) - X^4\left(\frac{1}{3!} - \frac{1}{4!}\right) 
		\\
		&\quad + \cdots + (-1)^{k}X^{k+1}\left(\frac{1}{k!} - \frac{1}{(k+1)!}\right) + \cdots
		\\
		& =  I + \sum_{k=1}^{\infty}(-1)^{k}X^{k+1}\frac{k}{(k+1)!}  = I + g(X).
	\end{align*}
	By the product rule, for $h(X) = X^k$, $k \in \Nbb$
	\begin{align*}
		Dh(X_0)(X) = XX_0^{k-1} + X_0XX_0 \ldots X_0 + X_0^2XX_0\ldots X_0 + \cdots + X_0^{k-1}X.
	\end{align*}
	By the chain rule, for $f(X) = \log\dettwo(I+X) = \log\det(I+g(X))$, we have
	\begin{align*}
		&Df(X_0)(X) = D\log\det(I+g(X_0))\compose Dg(X_0)(X) 
		\\
		&= \trace[(I+g(X_0))^{-1}Dg(X_0)(X)]
		\\
		& = \trace\left[\exp(X_0)(I+X_0)^{-1}\left(-X_0 +X_0^2  + \cdots + (-1)^{k}X_0^k\frac{1}{(k-1)!} + \cdots\right)X\right]
		\\
		& = - \trace\left[(I+X_0)^{-1}\exp(X_0)\left(I-X_0 + \cdots + (-1)^{k-1}X_0^{k-1}\frac{1}{(k-1)!} + \cdots\right)X_0X\right]
		\\
		& = -\trace[(I+X_0)^{-1}\exp(X_0)\exp(-X_0)X_0X] = - \trace[(I+X_0)^{-1}X_0X].
	\end{align*}
	Here we have used the cyclic property of the trace.
	\qed
\end{proof}
\begin{corollary}
	\label{corollary:first-derivatve-logdet-2}
	Let $A \in \Sym(\Hcal) \cap \HS(\Hcal)$ be fixed. Let $\Omega = \{t \in \R: I + tA > 0\}$.
	Define $f: \Omega \mapto \R$ by $f(t) = \log\dettwo(I+tA)$.
	Then $\frac{d}{dt}\log\dettwo(I+tA) = - t\trace[A^2(I+tA)]$.
\end{corollary}
\begin{proof}
	Using the chain rule and Lemma \ref{lemma:Frechet-derivative-logdet2}, $\forall t_0 \in \Omega$, $\forall h \in \R$,
	\begin{align*}
		Df(t_0)(h) &= D\log\dettwo(I+t_0A)(hA) = hD\log\dettwo(I+t_0A)(A) 
		\\
		&= -h\trace[(I+t_0A)^{-1}(t_0A)A] = -ht_0\trace[A^2(I+t_0A)^{-1}].
	\end{align*}
	It follows that $\frac{d}{dt}\log\dettwo(I+tA) = Df(t)(1) = -t\trace[A^2(I+tA)^{-1}]$.
	\qed
\end{proof}

\begin{corollary}
	\label{corollary:second-derivative-logdet-2-st}
	Let $A,B \in \Sym(\Hcal) \cap \HS(\Hcal)$ be fixed. Let $\Omega = \{(s,t) \in \R^2: I + sA+tB > 0\}$.
	Define $f: \Omega \mapto \R$ by $f(s,t) = \log\dettwo(I+sA+tB)$.
	Then
	\begin{align}
		\frac{\partial f}{\partial s} &= \trace[A((I+sA+tB)^{-1}-I)],
		\\
		\frac{\partial^2f}{\partial s\partial t} &= -\trace[A(I+sA+tB)^{-1}B(I+sA+tB)^{-1}].
		\\
		\frac{\partial^2f}{\partial s\partial t}\big\vert_{s=0,t=0} &= - \trace(AB).
	\end{align}
\end{corollary}
\begin{proof}
	Let $t$ be fixed. Define $g_1(s) = f(s,t)$.
	By the chain rule and Lemma \ref{lemma:Frechet-derivative-logdet2},
	\begin{align*}
		&Dg_1(s_0)(h) = D\log\dettwo(I+s_0A+tB)(hA) = hD\log\dettwo(I+s_0A+tB)(A)
		\\
		&= -h\trace[(I+s_0A+tB)^{-1}(s_0A+tB)A] = h\trace[A((I+s_0A+tB)^{-1}-I)],
	\end{align*}
	$\forall (s_0,t) \in \Omega$, $\forall h \in \R$. Thus $\frac{\partial f}{\partial s} = Dg_1(s)(1) = \trace[A((I+sA+tB)^{-1}-I)]$.
	
	Differentiating this with respect to $t$ gives
	\begin{align*}
		\frac{\partial^2f}{\partial s\partial t} = -\trace[A(I+sA+tB)^{-1}B(I+sA+tB)^{-1}].
	\end{align*}
	Consequently, $\frac{\partial^2f}{\partial s\partial t}\big\vert_{s=0,t=0} = - \trace(AB)$.
	\qed
\end{proof}

\begin{proof}
[\textbf{of Theorem \ref{theorem:Riemannian-metric} via the KL divergence}]
Let $\Sigma = C_0^{1/2}(I-S)C_0^{1/2}$, 
$\mu_{\Sigma} = \Ncal(0, \Sigma)$. 
The assumption $A_1,A_2 \in T_{\Sigma}(\Tr(\Hcal,C_0))$ means $\exists W_1,W_2 \in \Sym(\Hcal)\cap \HS(\Hcal)$ such that
$A_1 = \Sigma^{1/2}W_1\Sigma^{1/2}$, $A_2 = \Sigma^{1/2}W_2\Sigma^{1/2}$.
Let $\mu = \Ncal(0, \Sigma + sA_1 + tA_2)$,
with $\Sigma +sA_1+tA_2 = \Sigma^{1/2}(I+sW_1 + tW_2)\Sigma^{1/2}$,
where $s,t \in \R$ are sufficiently small so that $I+sW_1 + tW_2 > 0$.
By Theorem \ref{theorem:KL-gaussian},
\begin{align*}
\KL(\mu||\mu_{\Sigma}) = -\frac{1}{2}\log\dettwo(I+sW_1+tW_2).
\end{align*}
By Corollary \ref{corollary:second-derivative-logdet-2-st},
\begin{align*}
&\frac{\partial^2}{\partial s \partial t}\KL(\mu||\mu_{\Sigma})\big\vert_{s=0,t=0}
= -\frac{1}{2}\frac{\partial^2}{\partial s \partial t}\log\dettwo(I+sW_1+tW_2)\big\vert_{s=0,t=0}
\\
& = \frac{1}{2}\trace(W_1W_2) = \frac{1}{2}\trace(\Sigma^{-1/2}A_1\Sigma^{-1}A_2\Sigma^{-1/2}) = \frac{1}{2}\la A_1,A_2\ra_{\Sigma}.
\end{align*}
Let us now find $V_1,V_2\in \Sym(\Hcal)\cap\HS(\Hcal)$ so that 
$\Sigma +sA_1 +tA_2 = C_0^{1/2}(I-S +sV_1 +tV_2)C_0^{1/2}$, where $s,t$ are sufficiently small
so that $I-S+sV_1 +tV_2 > 0$. By the polar decomposition, as in the proof of Lemma \ref{lemma:switch-C0-Sigma},
\begin{align*}
\Sigma^{1/2} = (C_0^{1/2}(I-S)C_0^{1/2})^{1/2} = C_0^{1/2}(I-S)^{1/2}U = U^{*}(I-S)^{1/2}C_0^{1/2},
\end{align*}
where $U$ is unitary. It follows that
\begin{align*}
&\Sigma + sA_1 + tA_2 = \Sigma^{1/2}(I+sW_1 + tW_2)\Sigma^{1/2} 
\\
&= C_0^{1/2}(I-S)C_0^{1/2} + sC_0^{1/2}(I-S)^{1/2}UW_1U^{*}(I-S)^{1/2}C_0^{1/2} 
\\
&\quad + tC_0^{1/2}(I-S)^{1/2}UW_2U^{*}(I-S)^{1/2}C_0^{1/2}
\\
& = C_0^{1/2}(I-S)^{1/2}\left[I + sUW_1U^{*} + t UW_2U^{*}\right](I-S)^{1/2}C_0^{1/2}.
\end{align*}
Thus $\Sigma + sA_1 + tA_2 = C_0^{1/2}(I-S +sV_1 +tV_2)C_0^{1/2}$ if and only if
\begin{align*}
V_j = (I-S)^{1/2}UW_jU^{*}(I-S)^{1/2} \equivalent W_j = U^{*}(I-S)^{-1/2}V_j(I-S)^{-1/2}U.
\end{align*}
Substituting into the second derivative of the KL divergence, we obtain
\begin{align*}
g_S(V_1,V_2) &= \frac{\partial^2}{\partial s \partial t}\KL(\mu||\mu_{\Sigma}) \big\vert_{s=0,t=0}
\\
&= \frac{1}{2}\trace(W_1W_2) = \frac{1}{2}\trace[(I-S)^{-1}V_1(I-S)^{-1}V_2].
\end{align*}
\qed
\end{proof}

\subsection{Proofs for the Riemannian connection and curvature tensor}
\label{section:levi-civita-curvature-tensor}

{\bf Levi-Civita connection}.
Let
$\Xfrak(\Tr(\Hcal, C_0))$ denote the set of all smooth vector fields on $\Tr(\Hcal, C_0)$.
Let $X \in \Xfrak(\Tr(\Hcal, C_0))$, which is a mapping $X: \Tr(\Hcal,C_0) \mapto \SymHS(\Hcal,C_0)$. 
For any differentiable function $f: \Tr(\Hcal,C_0) \mapto \R$, we have the corresponding function $Xf:\Tr(\Hcal,C_0)\mapto \R$, defined by
\begin{align}
	\label{equation:def-Xf}
(Xf)(P) = X_P(f) = Df(P)(X_P),
\end{align}
where $Df(P):\SymHS_X(\Hcal,C_0)\mapto \R$ is the Fr\'echet derivative of $f$ at $P \in \Tr(\Hcal,C_0)$, regarded as an open subset of the Hilbert space $\SymHS_X(\Hcal,C_0)$.

Consider $X:\Tr(\Hcal,C_0) \mapto \SymHS(\Hcal,C_0)$, with $DX(P):\SymHS_X(\Hcal,C_0) \mapto \SymHS(\Hcal,C_0)$ denoting its Fr\'echet derivative at $P$.
Then for $Y_P \in \SymHS(\Hcal,C_0)$, $DX(P)(Y_P) \in \SymHS(\Hcal,C_0)$ and thus for $f:\Tr(\Hcal,C_0) \mapto \R$,
\begin{align*}
[DX(P)(Y_P)](f) = Df(P)[DX(P)(Y_P)] \in \R.
\end{align*}


\begin{lemma}
	\label{lemma:Frechet-derivative-Xf}
	Let $X$ be a smooth vector field on $\Tr(\Hcal,C_0)$.
	Let $f \in C^2(\Tr(\Hcal,C_0))$.  Then $\forall P \in \Tr(\Hcal,C_0)$, $\forall H \in \SymHS(\Hcal,C_0)$,
	\begin{align}
	D(Xf)(P)(H) = Df(P)(DX(P)(H)) + D^2f(P)(H)(X(P)).
	\end{align}
For two smooth vector fields $X,Y$ on $\Tr(\Hcal,C_0)$, their Lie bracket is given by
\begin{align}
&	[X,Y](f)(P) = [X,Y]_P(f) = D(Yf)(X_P) - D(Xf)(Y_P)
\nonumber
\\
&= Df(P)(DY(P)(X(P))) - Df(P)(DX(P)(Y(P)))
\nonumber
\\
&= [DY(P)(X_P) - DX(P)(Y_P)](f).
\end{align}
Thus $[X,Y]_P = DY(P)(X_P) - DX(P)(Y_P)$ $\forall P \in \Tr(\Hcal,C_0)$.
\end{lemma}
\begin{proof}
	For simplicity, here we denote the norm $||\;||_{\HS_X(\Hcal,C_0)}$ by $||\;||$.
By definition, 
for a fixed 
$P \in \Tr(\Hcal,C_0)$, 
$D(Xf)(P): \SymHS_X(\Hcal,C_0) \mapto \R$. For $H \in \SymHS(\Hcal,C_0)$,
$\lim\limits_{H \approach 0}\frac{|(Xf)(P+H) - (Xf)(P) - D(Xf)(P)(H)|}{||H||} = 0$.
We have
\begin{align*}
	&(Xf)(P+H) - (Xf)(P) - Df(P)(DX(P)(H)) - D^2f(P)(H)(X(P))
	\\
	&= Df(P+H)(X(P+H)) - Df(P)(X(P)) - Df(P)(DX(P)(H)) 
	\\
	&\quad -D^2f(P)(H)(X(P))
	\\
	&= [Df(P+H)-Df(P)-D^2f(P)(H)](X(P+H))
	\\
	&\quad + Df(P)(X(P+H) - X(P) - DX(P)(H)) + D^2f(P)(H)(X(P+H) - X(P)).
\end{align*}
Since $\lim\limits_{H \approach 0}\frac{||Df(P+H)-Df(P)-D^2f(P)(H)||}{||H||} = 0$,
$\lim\limits_{H \approach 0}\frac{||X(P+H) - X(P) - DX(P)(H)||}{||H||} = 0$,
$|D^2f(P)(H)(X(P+H) - X(P))| \leq ||D^2f(P)||\;||H||\;||X(P+H)-X(P)||$,
it follows that 
\begin{align*}
\lim_{H\approach 0}\frac{|(Xf)(P+H) - (Xf)(P) - Df(P)(DX(P)(H)) - D^2f(P)(H)(X(P))|}{||H||} = 0.
\end{align*}
This shows that $D(Xf)(P)(H) = Df(P)(DX(P)(H)) + D^2f(P)(H)(X(P))$.
For two smooth vector fields $X,Y$ on $\Tr(\Hcal,C_0)$,
\begin{align*}
D(Xf)(P)(Y(P)) = Df(P)(DX(P)(Y(P))) + D^2f(P)(Y(P))(X(P)),
\\
D(Yf)(P)(X(P)) = Df(P)(DY(P)(X(P))) + D^2f(P)(X(P))(Y(P)).
\end{align*}
Since $D^2f(P): \SymHS_X(\Hcal,C_0) \times \SymHS_X(\Hcal,C_0) \mapto \R$ is a symmetric bilinear form,  $D^2f(P)(Y(P))(X(P)) = D^2f(P)(X(P))(Y(P))$.
It follows that 
\begin{align*}
&D(Xf)(P)(Y(P)) - D(Yf)(P)(X(P))
\\
& = Df(P)(DX(P)(Y(P))) - Df(P)(DY(P)(X(P)))
\\
& = [DX(P)(Y_P) - DY(P)(X_P)](f).
\end{align*}
On the other hand, by definition of the Lie bracket, 
$[X,Y]f = X(Yf) - Y(Xf)$,
where the functions $X(Yf), Y(Xf):\Tr(\Hcal,C_0) \mapto \R$ are given by
\begin{align*}
	[X(Yf)](P) = X_P(Yf) = D(Yf)(P)(X_P) \in \R, \; P \in \Tr(\Hcal,C_0),
	\\
	[Y(Xf)](P) = Y_P(Xf) = D(Xf)(P)(Y_P) \in \R, \; P \in \Tr(\Hcal,C_0).
\end{align*}
It follows that $\forall f \in C^{\infty}(\Tr(\Hcal,C_0))$, $\forall P \in \Tr(\Hcal,C_0)$, by the preceding part,
\begin{align*}
	&([X,Y]f)(P) = [X,Y]_P(f) = X_P(Yf) - Y_P(Xf) = D(Yf)(P)(X_P) - D(Xf)(P)(Y_P)
	\\
	&=	Df(P)(DY(P)(X(P))) - Df(P)(DX(P)(Y(P))) 
	\\
	&	= (D(Y)(P)[X_P] - D(X)(P)[Y_P])(f).
	\end{align*}
Thus $[X,Y]_P = D(Y)(P)[X_P] - D(X)(P)[Y_P]$.
\qed
\end{proof}

The following is Lemma 26 in \cite{Minh2020:EntropicHilbert}.
\begin{lemma}
	\label{lemma:cancellation-product-strictly-positive}
	Let $A,B \in \Lcal(\Hcal)$. Assume that $B \in \Sym^{++}(\Hcal)$ is compact. Then
	\begin{align}
		AB = 0 \equivalent A = 0,\;\;\; BA = 0 \equivalent A = 0.
	\end{align}
\end{lemma}

\begin{lemma}
	\label{lemma:trace-derivative-square-root-product}
	Consider the function $\mysqrt(X) = X^{1/2}$ on $\Tr(\Hcal,C_0)$.
	Let $P \in \Tr(\Hcal,C_0)$ be fixed. 
	Consider the Fr\'echet derivative $D\mysqrt(P):\SymHS_X(\Hcal,C_0) \mapto \Sym(\Hcal)\cap\HS(\Hcal)$.
	For $X \in \SymHS(\Hcal,C_0)$, 
	\begin{align}
		D\mysqrt(P)(X) = P^{1/2}ZP^{1/2}  \text{for some $Z \in \Sym(\Hcal)$}.
	\end{align}
	Thus the operators $D\mysqrt(P)(X)P^{-1/2}\in \HS(\Hcal)$, $P^{-1/2}D\mysqrt(P)(X) \in \HS(\Hcal)$,
	$P^{-1/2}D\mysqrt(P)(X)P^{-1/2} \in \Sym(\Hcal)$ are well-defined.
	$\forall Y \in \HS(\Hcal)$,
	\begin{align}
		\trace[D\mysqrt(P)(X)P^{-1/2}Y + P^{-1/2}D\mysqrt(P)(X)Y] = \trace[P^{-1/2}XP^{-1/2}Y].
	\end{align}
\end{lemma}
\begin{proof}
	For the function $\mysqrt(X) = X^{1/2}$ on $\Tr(\Hcal,C_0)$, which is an open subset 
	in $\SymHS_X(\Hcal,C_0)$, we have $\mysqrt:\Tr(\Hcal,C_0) \mapto \Sym^{+}(\Hcal)\cap \HS(\Hcal)$. Thus $D\mysqrt(P): \SymHS_X(\Hcal,C_0) \mapto \Sym(\Hcal) \cap \HS(\Hcal)$
	$\forall P \in \Tr(\Hcal,C_0)$.
	Using
	the identity $(\mysqrt(X))^2 = X$ and the chain rule, we have $\forall P \in \Tr(\Hcal,C_0)$, $\forall X \in \SymHS_X(\Hcal,C_0)$,
	\begin{align*}
		P^{1/2}D\mysqrt(P)(X) + D\mysqrt(P)(X)P^{1/2} = X.
	\end{align*}
	In particular, for $X \in \SymHS(\Hcal,C_0) = \SymHS(\Hcal,P)$ by Lemma \ref{lemma:switch-C0-Sigma}, we have $X = P^{1/2}\tilde{X}P^{1/2}$
	with $\tilde{X}\in \Sym(\Hcal) \cap \HS(\Hcal)$.
	Thus $D\mysqrt(P)(X)$ must necessarily have the form $D\mysqrt(P)(X) = P^{1/2}ZP^{1/2}$ for some
	$Z\in \Sym(\Hcal)$. Thus $P^{-1/2}D\mysqrt(P)(X) = ZP^{1/2} \in \HS(\Hcal)$, $D\mysqrt(P)(X)P^{-1/2} = P^{1/2}Z\in \HS(\Hcal)$,
	and $P^{-1/2}D\mysqrt(P)(X)P^{-1/2} = Z \in \Sym(\Hcal)$ are all
	well-defined. 
	By Lemma \ref{lemma:cancellation-product-strictly-positive}, pre- and post-canceling $P^{1/2}$ on both sides gives
	\begin{align*}
		D\mysqrt(P)(X)P^{-1/2} + P^{-1/2}D\mysqrt(P)(X) = \tilde{X}  = P^{-1/2}XP^{-1/2}.
	\end{align*}
	By multiplying both sides with $Y \in \HS(\Hcal)$ and taking the trace,
	\begin{align*}
		\trace[D\mysqrt(P)(X)P^{-1/2}Y + P^{-1/2}D\mysqrt(P)(X)Y] = \trace[P^{-1/2}XP^{-1/2}Y].
	\end{align*}
	\qed
\end{proof}

\begin{lemma}
	\label{lemma:Frechet-derivative-product-Banach-algebra}
	Let $(W, ||\;||_W)$ be a Banach algebra and $\Omega \subset W$ be an open subset.
	Let $g,h: \Omega \mapto W$ be Fr\'echet differentiable at $X_0 \in \Omega$.
	Let $f(X) = g(X)h(X)$. Then $f$ is Fr\'echet differentiable at $X_0$ and
	$Df(X_0)(X) = Dg(X_0)(X)h(X_0) + g(X_0)Dh(X_0)(X)$ $\forall X \in W$.
\end{lemma}
\begin{proof}
	For a fixed $X \in W$, let $t$ be sufficiently close to zero so that $X_0+tX \in \Omega$.
	With $f(X_0 +tX) = g(X_0+tX)h(X_0+tX)$, $f(X_0) = g(X_0)h(X_0)$, we have
	\begin{align*}
		&||f(X_0+tX) - f(X_0) - (Dg(X_0)(tX)h(X_0) + g(X_0)Dh(X_0)(tX)||_W
		\\
		& \leq ||g(X_0+tX) - g(X_0) - tDg(X_0)(X)||_W\;||h(X_0)||_W
		\\
		& \quad + ||g(X_0+tX)||h(X_0+tX) - h(X_0) - tDh(X_0)(X)||_W
		\\
		&\quad + ||g(X_0+tX) - g(X_0)||_W|t|||Dh(X_0)||\;||X||_W.
	\end{align*}
	Thus $\lim_{t \approach 0}\frac{||f(X_0+tX) - f(X_0) - (Dg(X_0)(tX)h(X_0) + g(X_0)Dh(X_0)(tX)||_W}{|t|\;||X||_W} = 0$.
	\qed
\end{proof}

\begin{proposition}
	[Levi-Civita connection]
	\label{proposition:Levi-Civita}
	For $\Tr(\Hcal,C_0)$ under the Riemannian metric in Theorem \ref{theorem:Riemannian-metric}, the Levi-Civita connection is, 
	$\forall P \in \Tr(\Hcal,C_0)$,
	\begin{align}
		\label{equation:Levi-Civita}
	(\nabla_{X}Y)_P = D(Y)(P)[X_P] - \frac{1}{2}[X_PP^{-1}Y_P + Y_PP^{-1}X_P].
	\end{align}
Here $D(Y)(P)$ denotes the Fr\'echet derivative of $Y$ at $P$ in the open subset $\Tr(\Hcal,C_0)$ of the Hilbert space $\SymHS_X(\Hcal,C_0)$.
\end{proposition}
We first note that for $X_P,Y_P,Z_P \in T_P(\Tr(\Hcal,C_0)) \cong \SymHS(\Hcal,C_0) = \SymHS(\Hcal,P)$, one can write
$X_P = P^{1/2}\tilde{X}_PP^{1/2}$, $Y_P = P^{1/2}\tilde{Y}_PP^{1/2}$, $Z_P = P^{1/2}\tilde{Z}_PP^{1/2}$ for $\tilde{X}_P,\tilde{Y}_P, \tilde{Z}_P \in \Sym(\Hcal)\cap\HS(\Hcal)$.
Thus the following expressions are well-defined
\begin{align}
&X_PP^{-1}Y_P + Y_PP^{-1}X_P = P^{1/2}(\tilde{X}_P\tilde{Y}_P+\tilde{Y}_P\tilde{X}_P)P^{1/2} \in \SymTr(\Hcal,P),
\\
&\trace(P^{-1/2}X_PP^{-1}Y_PP^{-1/2}) = \trace(\tilde{X}_P\tilde{Y}_P),
\\
&\trace(P^{-1/2}XP^{-1}Y_PP^{-1}Z_PP^{-1/2}) = \trace(\tilde{X}_P\tilde{Y}_P\tilde{Z}_P).
\end{align}
Here $\SymTr(\Hcal,P) = \{V = P^{1/2}XP^{1/2}, \; X \in \Sym(\Hcal) \cap \Tr(\Hcal)\}$.
\begin{proof}
	We need to either check the symmetry and compatibility with the Riemannian metric or verify that the Koszul formula is satisfied.
	
	(i) {\it Symmetry}, i.e. $(\nabla_XY) - (\nabla_YX) = [X,Y]$ $\forall X,Y\in \Xfrak(\Tr(\Hcal,C_0))$. From Eq.\eqref{equation:Levi-Civita}, $\forall P \in \Tr(\Hcal,C_0)$,
	\begin{align*}
	(\nabla_{X}Y)_P = D(Y)(P)[X_P] - \frac{1}{2}[X_PP^{-1}Y_P + Y_PP^{-1}X_P],
	\\
	(\nabla_{Y}X)_P = D(X)(P)[Y_P] - \frac{1}{2}[Y_PP^{-1}X_P + X_PP^{-1}Y_P],
	\\
	(\nabla_XY)_P - (\nabla_YX)_P = D(Y)(P)[X_P] - D(X)(P)[Y_P] = [X,Y]_P.
	\end{align*}
Here the last equality comes from Lemma \ref{lemma:Frechet-derivative-Xf}.

(ii) {\it Compatibility with the Riemannian metric}, i.e. $\forall X,Y,Z \in \Xfrak(\Tr(\Hcal,C_0))$, $X\la Y,Z\ra = \la \nabla_XY,Z\ra + \la Y, \nabla_XZ\ra$. From the expressions for $\nabla_XY, \nabla_XZ$, namely
\begin{align*}
	(\nabla_{X}Y)_P = D(Y)(P)(X_P) - \frac{1}{2}[X_PP^{-1}Y_P + Y_PP^{-1}X_P],
	\\
	(\nabla_{X}Z)_P = D(Z)(P)(X_P) - \frac{1}{2}[X_PP^{-1}Z_P + Z_PP^{-1}X_P],
\end{align*}
and the expression for the Riemannian metric $\la,\ra_P$, it follows that
({\it for simplicity, we skip the multiplicative factor $\frac{1}{2}$ in the metric here})
\begin{align*}
	&\la \nabla_XY,Z\ra_P = \trace(P^{-1/2}(\nabla_XY)_PP^{-1}Z_PP^{-1/2}) 
	\\
	&=\trace[P^{-1/2}D(Y)(P)(X_P)P^{-1}Z_PP^{-1/2}] 
	\\
	&\quad - \frac{1}{2}\trace[P^{-1/2}(X_PP^{-1}Y_P + Y_PP^{-1}X_P)P^{-1}Z_PP^{-1/2}],
	\\
	&\la \nabla_XZ,Y\ra_P
	=\trace[P^{-1/2}D(Z)(P)(X_P)P^{-1}Y_PP^{-1/2}] 
	\\
	&\quad - \frac{1}{2}\trace[P^{-1/2}(X_PP^{-1}Z_P + Z_PP^{-1}X_P)P^{-1}Y_PP^{-1/2}],
	\\
	&\la \nabla_XY,Z\ra_P + \la \nabla_XZ,Y\ra_P 
	\\
	&= \trace[P^{-1/2}D(Y)(P)(X_P)P^{-1}Z_PP^{-1/2}] 
	+\trace[P^{-1/2}D(Z)(P)(X_P)P^{-1}Y_PP^{-1/2}] 
	\\
	& \quad- \trace[P^{-1/2}(X_PP^{-1}Y_P + Y_PP^{-1}X_P)P^{-1}Z_PP^{-1/2}].
\end{align*}
where the last equality follows from the cyclic property of the trace operation.

To calculate $X\la Y,Z\ra$, define $\tilde{Y} = P^{-1/2}YP^{-1/2}$, $\tilde{Z}= P^{-1/2}ZP^{-1/2}$, $\tilde{Y}, \tilde{Z} \in \Sym(\Hcal) \cap \HS(\Hcal)$, so that $Y = P^{1/2}\tilde{Y}P^{1/2}$, $Z = P^{1/2}\tilde{Z}P^{1/2}$. Then $\la Y,Z\ra_P = \trace(P^{-1/2}Y_PP^{-1}Z_PP^{-1/2}) = \trace(\tilde{Y}_P\tilde{Z}_P)$. 
By Lemma \ref{lemma:Frechet-derivative-product-Banach-algebra}, applied to the Banach algebra $(\HS(\Hcal), ||\;||_{\HS})$,
\begin{align*}
&(X\la Y,Z\ra)(P) = X_P(\la Y,Z\ra) = D(\la Y,Z\ra)(P)(X_P) 
\\
& = \trace[D\tilde{Y}(P)(X_P)\tilde{Z}_P + \tilde{Y}_PD\tilde{Z}(P)(X_P)].
\end{align*}
With $Y = P^{1/2}\tilde{Y}P^{1/2}$, by Lemma \ref{lemma:Frechet-derivative-product-Banach-algebra},
for $H \in \SymHS(\Hcal,C_0) = \SymHS(\Hcal,P)$
\begin{align*}
DY(P)(H)=D\mysqrt(P)(H)\tilde{Y}P^{1/2} + P^{1/2}D\tilde{Y}(P)(H)P^{1/2} + P^{1/2}\tilde{Y}D\mysqrt(P)(H).
\end{align*}
It follows from Lemma \ref{lemma:trace-derivative-square-root-product} that for $H \in \SymHS(\Hcal,C_0) = \SymHS(\Hcal,P)$
\begin{align*}
D\tilde{Y}(P)(H) &= P^{-1/2}DY(P)(H)P^{-1/2}
\\
& \quad - P^{-1/2}D\mysqrt(P)(H)\tilde{Y} - \tilde{Y}D\mysqrt(P)(H)P^{-1/2},
\end{align*}
Similarly,
\begin{align*}
D\tilde{Z}(P)(H) & = 
P^{-1/2}DZ(P)(H)P^{-1/2}
\\
& \quad - P^{-1/2}D\mysqrt(P)(H)\tilde{Z} - \tilde{Z}D\mysqrt(P)(H)P^{-1/2}.
\end{align*}
Combining these two expressions, letting $H = X_P$ we obtain
\begin{align*}
&(X\la Y,Z\ra)(P) = \trace[D\tilde{Y}(P)(X_P)\tilde{Z}_P + D\tilde{Z}(P)(X_P)\tilde{Y}_P]
\\
&=\trace[P^{-1/2}DY(P)(X_P)P^{-1}Z_PP^{-1/2} + P^{-1/2}DZ(P)(X_P)P^{-1}Y_PP^{-1/2}]
\\
&\quad - \trace[P^{-1/2}D\mysqrt(P)(X_P)(\tilde{Y}_P\tilde{Z}_P +\tilde{Z}_P\tilde{Y}_P)+ D\mysqrt(P)(X_P)P^{-1/2}(\tilde{Y}_P\tilde{Z}_P + \tilde{Z}_P\tilde{Y}_P)].
\end{align*}
By Lemma \ref{lemma:trace-derivative-square-root-product},
\begin{align*}
&\trace[P^{-1/2}D\mysqrt(P)(X_P)(\tilde{Y}_P\tilde{Z}_P +\tilde{Z}_P\tilde{Y}_P)+ D\mysqrt(P)(X_P)P^{-1/2}(\tilde{Y}_P\tilde{Z}_P + \tilde{Z}_P\tilde{Y}_P)]
\\
&=\trace[P^{-1/2}X_PP^{-1/2}(\tilde{Y}_P\tilde{Z}_P + \tilde{Z}_P\tilde{Y}_P)]
\\
&= \trace[P^{-1/2}X_PP^{-1/2}(P^{-1/2}Y_PP^{-1}Z_PP^{-1/2} + P^{-1/2}Z_PP^{-1}Y_PP^{-1/2})]
\\
& = \trace[P^{-1/2}(X_PP^{-1}Y_P + Y_PP^{-1}X_P)P^{-1}Z_PP^{-1/2}].
\end{align*}
Combining all expressions, we obtain $(X\la Y,Z\ra)(P) =\la \nabla_XY,Z\ra_P + \la \nabla_XZ,Y\ra_P$
$\forall P \in \Tr(\Hcal,C_0)$.
\qed
\end{proof}

\begin{lemma}
	\label{lemma:derivative-composition-vector-fields}
	Let $X,Z \in \Xfrak(\Tr(\Hcal,C_0))$. Let $Y:\Tr(\Hcal,C_0) \mapto \SymHS(\Hcal,C_0)$ be defined by
	$Y(P) = DZ(P)(X(P))$. Then $\forall H \in \SymHS(\Hcal,C_0)$,
	\begin{align}
		DY(P)(H) = D^2Z(P)(H)(X(P)) + DZ(P)(DX(P)(H)).
	\end{align}
\end{lemma}
\begin{proof}
	For simplicity, here we denote the norm $||\;||_{\HS_X(\Hcal,C_0)}$ by $||\;||$.
	From the definition $Y(P) = DZ(P)(X(P))$,
	\begin{align*}
		&Y(P+H) - Y(P) - D^2Z(P)(H)(X(P)) - DZ(P)(DX(P)(H))
		\\
		&= [DZ(P+H) - DZ(P) - D^2Z(P)(H)](X(P+H))  
		\\
		& \quad+ DZ(P)[X(P+H) - X(P)-DX(P)(H)] + D^2Z(P)(H)[X(P+H)-X(P)].
	\end{align*}
	The result then follows from the limits $\lim\limits_{H \approach 0}\frac{||DZ(P+H) - DZ(P) - D^2Z(P)(H)||}{||H||} = 0$,
	$\lim\limits_{H \approach 0}\frac{||X(P+H) - X(P)-DX(P)(H)||}{||H||} = 0$, 
	and $\frac{||D^2Z(P)(H)[X(P+H)-X(P)]||}{||H||}\leq ||D^2Z(P)||\;\;\;$
	$||X(P+H)-X(P)|| \approach 0$ as $H \approach 0$.
	\qed
\end{proof}

\begin{lemma}
	\label{lemma:derivative-Levi-Civita}
	Let $\nabla$ be the Riemannian connection given in Proposition \ref{proposition:Levi-Civita}.
	Then $\forall X,Y,Z \in \Xfrak(\Tr(\Hcal,C_0))$, $\forall P \in \Tr(\Hcal,C_0)$,
\begin{align}
	&D(\nabla_XZ)(P)(Y_P) = D^2Z(P)(X_P,Y_P) + DZ(P)(DX(P)(Y_P))
	\nonumber
	\\
	&-\frac{1}{2}\left(DX(P)(Y_P)P^{-1}{Z}_P 
	+ {X}_PP^{-1}DZ(P)(Y_P) - X_PP^{-1}Y_PP^{-1}Z_P\right)
	\nonumber
	\\
	&-\frac{1}{2}\left(DZ(P)(Y_P)P^{-1}{X}_P 
	+ {Z}_PP^{-1}DX(P)(Y_P) - Z_PP^{-1}Y_PP^{-1}X_P\right).
\end{align}
\end{lemma}
\begin{proof}
	Let $X_P = P^{1/2}\tilde{X}_PP^{1/2}$, $Z_P = P^{1/2}\tilde{Z}_PP^{1/2}$. Then for $H \in \SymHS(\Hcal,C_0)$,
	\begin{align*}
	DX(P)(H) = D\mysqrt(P)(H)\tilde{X}_PP^{1/2} + P^{1/2}D\tilde{X}(P)(H)P^{1/2} + P^{1/2}\tilde{X}_PD\mysqrt(P)(H),
	\end{align*}
from which it follows that for $H \in \SymHS(\Hcal,C_0) = \SymHS(\Hcal,P)$
\begin{align*}
D\tilde{X}(P)(H) &= P^{-1/2}DX(P)(H)P^{-1/2}
\\
&\quad - P^{-1/2}D\mysqrt(P)(H)\tilde{X}_P 
- \tilde{X}_PD\mysqrt(P)(H)P^{-1/2}.
\end{align*}
Similarly, for $H \in \SymHS(\Hcal,C_0) = \SymHS(\Hcal,P)$,
\begin{align*}
D\tilde{Z}(P)(H) &= P^{-1/2}DZ(P)(H)P^{-1/2}
\\
&\quad - P^{-1/2}D\mysqrt(P)(H)\tilde{Z}_P 
- \tilde{Z}_PD\mysqrt(P)(H)P^{-1/2}.
\end{align*}
Recalling the expression for the Riemannian connection,
\begin{align*}
&	(\nabla_{X}Z)_P = D(Z)(P)(X_P) - \frac{1}{2}[X_PP^{-1}Z_P + Z_PP^{-1}X_P]
\\
&	 = D(Z)(P)(X_P) - \frac{1}{2}[P^{1/2}\tilde{X}_P\tilde{Z}_PP^{1/2} + P^{1/2}\tilde{Z}_P\tilde{X}_PP^{1/2}].
\end{align*}
Differentiating this, invoking Lemma \ref{lemma:derivative-composition-vector-fields}, gives
\begin{align*}
	D(\nabla_XZ)(P)(Y_P) &= D^2Z(P)(X_P,Y_P) + DZ(P)(DX(P)(Y_P))
	\\
	&\quad- \frac{1}{2}D(P^{1/2}\tilde{X}_P\tilde{Z}_PP^{1/2})(P)(Y_P) -\frac{1}{2} D(P^{1/2}\tilde{Z}_P\tilde{X}_PP^{1/2})(P)(Y_P).
\end{align*}
By the product rule, for $H \in \SymHS(\Hcal,C_0) = \SymHS(\Hcal,P)$,
\begin{align*}
&D(P^{1/2}\tilde{X}_P\tilde{Z}_PP^{1/2})(P)(H) 
\\
&= D\mysqrt(P)(H)\tilde{X}_P\tilde{Z}_PP^{1/2} + P^{1/2}D\tilde{X}(P)(H)\tilde{Z}_PP^{1/2}
\\
&\quad + P^{1/2}\tilde{X}_PD\tilde{Z}(P)(H)P^{1/2} + P^{1/2}\tilde{X}_P\tilde{Z}_PD\mysqrt(P)(H)
\\
& =  D\mysqrt(P)(H)\tilde{X}_P\tilde{Z}_PP^{1/2} + P^{1/2}\tilde{X}_P\tilde{Z}_PD\mysqrt(P)(H)
\\
&\quad + P^{1/2}\left(P^{-1/2}DX(P)(H)P^{-1/2} - P^{-1/2}D\mysqrt(P)(H)\tilde{X}_P \right.
\\
&\quad\quad \quad\quad \left.- \tilde{X}_PD\mysqrt(P)(H)P^{-1/2}\right)\tilde{Z}_PP^{1/2} 
\\
&\quad+ P^{1/2}\tilde{X}_P\left(P^{-1/2}DZ(P)(H)P^{-1/2}
- P^{-1/2}D\mysqrt(P)(H)\tilde{Z}_P \right.
\\
&\quad\quad\quad \quad\left.- \tilde{Z}_PD\mysqrt(P)(H)P^{-1/2}\right)P^{1/2} 
\\
& = DX(P)(H)P^{-1/2}\tilde{Z}_PP^{1/2} - P^{1/2}\tilde{X}_PD\mysqrt(P)(H)P^{-1/2}\tilde{Z}_PP^{1/2}
\\
&\quad + P^{1/2}\tilde{X}_PP^{-1/2}DZ(P)(H) - P^{1/2}\tilde{X}_PP^{-1/2}D\mysqrt(P)(H)\tilde{Z}_PP^{1/2}
\\
& = DX(P)(H)P^{-1}{Z}_P - {X}_PP^{-1/2}D\mysqrt(P)(H)P^{-1}{Z}_P
\\
&\quad + {X}_PP^{-1}DZ(P)(H) - {X}_PP^{-1}D\mysqrt(P)(H)P^{-1/2}{Z}_P
\\
& = DX(P)(H)P^{-1}{Z}_P + {X}_PP^{-1}DZ(P)(H) 
\\
&\quad - X_PP^{-1}[P^{1/2}D\mysqrt(P)(H) + D\mysqrt(P)(H)P^{1/2}]P^{-1}Z_P
\\
& = DX(P)(H)P^{-1}{Z}_P + {X}_PP^{-1}DZ(P)(H) -X_PP^{-1}HP^{-1}Z_P.
\end{align*}
Similarly, for $H \in \SymHS(\Hcal,C_0) = \SymHS(\Hcal,P)$,
\begin{align*}
&D(P^{1/2}\tilde{Z}_P\tilde{X}_PP^{1/2})(P)(H) 
\\
& = DZ(P)(H)P^{-1}{X}_P +{Z}_PP^{-1}DX(P)(H)-Z_PP^{-1}HP^{-1}X_P.
\end{align*}
Combining all expressions, letting $H=Y_P$, we obtain
\begin{align*}
&D(\nabla_XZ)(P)(Y_P) = D^2Z(P)(X_P,Y_P) + DZ(P)(DX(P)(Y_P))
\\
&-\frac{1}{2}\left(DX(P)(Y_P)P^{-1}{Z}_P 
+ {X}_PP^{-1}DZ(P)(Y_P) - X_PP^{-1}Y_PP^{-1}Z_P\right)
\\
&-\frac{1}{2}\left(DZ(P)(Y_P)P^{-1}{X}_P 
 + {Z}_PP^{-1}DX(P)(Y_P) - Z_PP^{-1}Y_PP^{-1}X_P\right).
\end{align*}
\qed
\end{proof}

\begin{proposition}
[Riemannian curvature tensor]
\label{proposition:curvature-tensor}
For the manifold $\Tr(\Hcal,C_0)$ under the Riemannian metric in Theorem \ref{theorem:Riemannian-metric}, the Riemannian curvature tensor is given by
\begin{align}
[R(X,Y)Z](P) = -\frac{1}{4}P^{1/2}[[P^{-1/2}{X}_PP^{-1/2}, P^{-1/2}{Y}_PP^{-1/2}],P^{-1/2}{Z}_PP^{-1/2}]P^{1/2}
\end{align}
$\forall X,Y,Z \in \Xfrak(\Tr(\Hcal,C_0))$, $\forall P \in \Tr(\Hcal,C_0)$. Here $[,]$ denotes the operator commutator
$[A,B] = AB-BA$.
\end{proposition}
\begin{proof}
By definition of the Riemannian curvature tensor, for $X,Y,Z \in \Xfrak(\Tr(\Hcal,C_0))$,
\begin{align*}
R(X,Y)Z= \nabla_X\nabla_YZ - \nabla_Y\nabla_XZ - \nabla_{[X,Y]}Z
\end{align*}
with $\nabla$ being the Riemannian connection, as given in Proposition \ref{proposition:Levi-Civita}.
Recalling
\begin{align*}
(\nabla_{X}Z)_P = D(Z)(P)(X_P) - \frac{1}{2}[X_PP^{-1}Z_P + Z_PP^{-1}X_P],
\\
(\nabla_{Y}Z)_P = D(Z)(P)(Y_P) - \frac{1}{2}[Y_PP^{-1}Z_P + Z_PP^{-1}Y_P].
\end{align*}
It follows that $\nabla_Y\nabla_XZ$ is given by
\begin{align*}
&(\nabla_Y\nabla_XZ)(P) = D(\nabla_XZ)(P)(Y_P) - \frac{1}{2}[Y_PP^{-1}(\nabla_XZ)(P) + (\nabla_XZ)(P)P^{-1}Y_P]
\\
& =  D(\nabla_XZ)(P)(Y_P)-\frac{1}{2}Y_PP^{-1}\left( D(Z)(P)(X_P) - \frac{1}{2}[X_PP^{-1}Z_P + Z_PP^{-1}X_P]\right)
\\
& \quad - \frac{1}{2}\left( D(Z)(P)(X_P) - \frac{1}{2}[X_PP^{-1}Z_P + Z_PP^{-1}X_P]\right)P^{-1}Y_P
\\
& = D(\nabla_XZ)(P)(Y_P) - \frac{1}{2}Y_PP^{-1}D(Z)(P)(X_P) + \frac{1}{4}Y_PP^{-1}X_PP^{-1}Z_P 
\\
&\quad + \frac{1}{4}Y_PP^{-1}Z_PP^{-1}X_P - \frac{1}{2} D(Z)(P)(X_P)P^{-1}Y_P + \frac{1}{4}X_PP^{-1}Z_PP^{-1}Y_P
\\
&\quad + \frac{1}{4}Z_PP^{-1}X_PP^{-1}Y_P.
\end{align*}
Similarly, $\nabla_X\nabla_YZ$ is given by
\begin{align*}
&(\nabla_X\nabla_YZ)(P) = D(\nabla_YZ)(P)(X_P)-\frac{1}{2}[X_PP^{-1}(\nabla_YZ)(P) + (\nabla_YZ)(P)P^{-1}X_P]
\\
& =  D(\nabla_YZ)(P)(X_P) -\frac{1}{2}X_PP^{-1}\left(D(Z)(P)(Y_P) - \frac{1}{2}[Y_PP^{-1}Z_P + Z_PP^{-1}Y_P]\right)
\\
&\quad-\frac{1}{2}\left(D(Z)(P)(Y_P) - \frac{1}{2}[Y_PP^{-1}Z_P + Z_PP^{-1}Y_P]\right)P^{-1}X_P
\\
& = D(\nabla_YZ)(P)(X_P) - \frac{1}{2}X_PP^{-1}D(Z)(P)(Y_P) + \frac{1}{4}X_PP^{-1}Y_PP^{-1}Z_P 
\\
&\quad + \frac{1}{4}X_PP^{-1}Z_PP^{-1}Y_P - \frac{1}{2}D(Z)(P)(Y_P)P^{-1}X_P + \frac{1}{4}Y_PP^{-1}Z_PP^{-1}X_P 
\\
&\quad + \frac{1}{4} Z_PP^{-1}Y_PP^{-1}X_P.
\end{align*}
By Lemma \ref{lemma:derivative-Levi-Civita},
\begin{align*}
	&D(\nabla_XZ)(P)(Y_P) = D^2Z(P)(X_P,Y_P) + DZ(P)(DX(P)(Y_P))
	\\
	&-\frac{1}{2}\left(DX(P)(Y_P)P^{-1}{Z}_P 
	+ {X}_PP^{-1}DZ(P)(Y_P) - X_PP^{-1}Y_PP^{-1}Z_P\right)
	\\
	&-\frac{1}{2}\left(DZ(P)(Y_P)P^{-1}{X}_P 
	+ {Z}_PP^{-1}DX(P)(Y_P) - Z_PP^{-1}Y_PP^{-1}X_P\right).
\end{align*}
\begin{align*}
	&D(\nabla_YZ)(P)(X_P) = D^2Z(P)(Y_P,X_P) + DZ(P)(DY(P)(X_P))
	\\
	&-\frac{1}{2}\left(DY(P)(X_P)P^{-1}{Z}_P 
	+ {Y}_PP^{-1}DZ(P)(X_P) - Y_PP^{-1}X_PP^{-1}Z_P\right)
	\\
	&-\frac{1}{2}\left(DZ(P)(X_P)P^{-1}{Y}_P 
	+ {Z}_PP^{-1}DY(P)(X_P) - Z_PP^{-1}X_PP^{-1}Y_P\right).
\end{align*}
Combining the last two expressions gives
\begin{align*}
	&D(\nabla_YZ)(P)(X_P) - D(\nabla_XZ)(P)(Y_P)
	\\
	&=
	DZ(P)(DY(P)(X_P)) - DZ(P)(DX(P)(Y_P))
	\\
	&\quad + \frac{1}{2}\left(Y_PP^{-1}X_PP^{-1}Z_P + Z_PP^{-1}X_PP^{-1}Y_P\right)
	\\
	&\quad-\frac{1}{2}\left(DY(P)(X_P)P^{-1}{Z}_P + {Z}_PP^{-1}DY(P)(X_P)\right)
	\\
	&\quad -\frac{1}{2}\left({Y}_PP^{-1}DZ(P)(X_P) + DZ(P)(X_P)P^{-1}{Y}_P\right)
	\\
	& \quad- \frac{1}{2}\left(X_PP^{-1}Y_PP^{-1}Z_P + Z_PP^{-1}Y_PP^{-1}X_P\right)
	\\
	&\quad + \frac{1}{2}\left(DX(P)(Y_P)P^{-1}{Z}_P + {Z}_PP^{-1}DX(P)(Y_P) \right)
	\\
	&\quad + \frac{1}{2}\left({X}_PP^{-1}DZ(P)(Y_P) + DZ(P)(Y_P)P^{-1}{X}_P\right).
\end{align*}
Combining the previous expressions, we obtain
\begin{align*}
&(\nabla_X\nabla_YZ - \nabla_Y\nabla_XZ)(P) = D(\nabla_YZ)(P)(X_P)-D(\nabla_XZ)(P)(Y_P)
\\
&\quad -  \frac{1}{2}X_PP^{-1}D(Z)(P)(Y_P)-\frac{1}{2}D(Z)(P)(Y_P)P^{-1}X_P
\\
&\quad +\frac{1}{2} Y_PP^{-1}D(Z)(P)(X_P)+\frac{1}{2}D(Z)(P)(X_P)P^{-1}Y_P
\\
& \quad + \frac{1}{4}X_PP^{-1}Y_PP^{-1}Z_P - \frac{1}{4}Z_PP^{-1}X_PP^{-1}Y_P
\\
& \quad + \frac{1}{4}Z_PP^{-1}Y_PP^{-1}X_P -\frac{1}{4}Y_PP^{-1}X_PP^{-1}Z_P
\\
& = DZ(P)(DY(P)(X_P)) - DZ(P)(DX(P)(Y_P))
\\
& \quad + \frac{1}{4}Y_PP^{-1}X_PP^{-1}Z_P + \frac{1}{4}Z_PP^{-1}X_PP^{-1}Y_P
- \frac{1}{4}X_PP^{-1}Y_PP^{-1}Z_P
\\
&\quad-\frac{1}{4}Z_PP^{-1}Y_PP^{-1}X_P
 -\frac{1}{2}\left(DY(P)(X_P)P^{-1}{Z}_P + {Z}_PP^{-1}DY(P)(X_P)\right)
\\
&\quad + \frac{1}{2}\left(DX(P)(Y_P)P^{-1}{Z}_P + {Z}_PP^{-1}DX(P)(Y_P) \right).
\end{align*}
For the quantity $\nabla_{[X,Y]}Z$, by definition
\begin{align*}
(\nabla_{[X,Y]}Z)(P) = D(Z)(P)([X,Y]_P) - \frac{1}{2}[[X,Y]_PP^{-1}Z_P + Z_PP^{-1}[X,Y]_P].
\end{align*}
By Lemma \ref{lemma:Frechet-derivative-Xf},
$[X,Y]_P = DY(P)(X_P) - DX(P)(Y_P)$.
It follows that
\begin{align*}
&(\nabla_{[X,Y]}Z)(P) = D(Z)(P)[DY(P)(X_P) - DX(P)(Y_P)] 
\\
&- \frac{1}{2}\left([DY(P)(X_P) - DX(P)(Y_P)]P^{-1}Z_P + Z_PP^{-1}[DY(P)(X_P) - DX(P)(Y_P)] \right).
\end{align*}
Combining this with the expression for $(\nabla_X\nabla_YZ - \nabla_Y\nabla_XZ)(P)$, we get
\begin{align*}
&(\nabla_X\nabla_YZ - \nabla_Y\nabla_XZ)(P) - (\nabla_{[X,Y]}Z)(P)
\\
&= \frac{1}{4}Y_PP^{-1}X_PP^{-1}Z_P + \frac{1}{4}Z_PP^{-1}X_PP^{-1}Y_P - \frac{1}{4}X_PP^{-1}Y_PP^{-1}Z_P
\\
&\quad -\frac{1}{4}Z_PP^{-1}Y_PP^{-1}X_P
\\
& = \frac{1}{4}P^{1/2}\tilde{Y}_P\tilde{X}_P\tilde{Z}_PP^{1/2}+\frac{1}{4}P^{1/2}\tilde{Z}_P\tilde{X}_P\tilde{Y}_PP^{1/2}
- \frac{1}{4}P^{1/2}\tilde{X}_P\tilde{Y}_P\tilde{Z}_PP^{1/2}
\\
&\quad - \frac{1}{4}P^{1/2}\tilde{Z}_P\tilde{Y}_P\tilde{X}_PP^{1/2}
\\
& = \frac{1}{4}P^{1/2}[\tilde{Z}_P,\tilde{X}_P\tilde{Y}_P]P^{1/2}+\frac{1}{4}P^{1/2}[\tilde{Y}_P\tilde{X}_P, \tilde{Z}_P]P^{1/2}
\\
& = -\frac{1}{4}P^{1/2}[\tilde{X}_P\tilde{Y}_P - \tilde{Y}_P\tilde{X}_P, \tilde{Z}_P]P^{1/2}
= -\frac{1}{4}P^{1/2}[[\tilde{X}_P, \tilde{Y}_P],\tilde{Z}_P]P^{1/2}.
\end{align*}
\qed
\end{proof}

\begin{proposition}
[Sectional curvature]
\label{proposition:sectional-curvature}
Under the Riemannian metric in Theorem \ref{theorem:Riemannian-metric}, the manifold $\Tr(\Hcal,C_0)$ has everywhere nonpositive sectional curvature.
Let $P\in \Tr(\Hcal,C_0)$, then the sectional curvature at $P$ is given by
\begin{align}
S_P(X,Y) =
-\frac{\trace[(P^{-1/2}XP^{-1/2})^2(P^{-1/2}YP^{-1/2})^2 - (P^{-1/2}XP^{-1}YP^{-1/2})^2]}{\trace(P^{-1/2}XP^{-1/2})^2\trace(P^{-1/2}YP^{-1/2})^2 -\trace(P^{-1/2}XP^{-1}YP^{-1/2})^2}
%
\end{align}
where $X,Y$ are any two linearly independent operators in $\SymHS(\Hcal,C_0) = \SymHS(\Hcal,P)$.
\end{proposition}
\begin{proof}
By definition, the sectional curvature at $P$ is
given by
\begin{align}
S_P(X,Y) = \frac{\la R(X,Y)Y, X\ra_P}{||X||^2_P||Y||^2_P - \la X,Y\ra_P^2}.
\end{align}
Let $X = P^{1/2}\tilde{X}P^{1/2}$, $Y = P^{1/2}\tilde{Y}P^{1/2}$. 
By Proposition \ref{proposition:curvature-tensor}, 
\begin{align*}
R(X,Y)Y = -\frac{1}{4}P^{1/2}[[\tilde{X},\tilde{Y}],\tilde{Y}]P^{1/2}.
\end{align*}
It follows that
\begin{align*}
&\la R(X,Y)Y,X\ra_P 
 = \frac{1}{2}\trace(P^{-1/2}R(X,Y)YP^{-1}XP^{-1/2}) 
\\
&= -\frac{1}{8}\trace([[\tilde{X},\tilde{Y}],\tilde{Y}]\tilde{X}) = -\frac{1}{8}\trace([\tilde{X}\tilde{Y}-\tilde{Y}\tilde{X},\tilde{Y}]\tilde{X})
\\
& = -\frac{1}{4}\trace([\tilde{X}\tilde{Y}^2 - \tilde{Y}\tilde{X}\tilde{Y}]\tilde{X}) = -\frac{1}{4}\trace(\tilde{X}^2\tilde{Y}^2 - (\tilde{X}\tilde{Y})^2).
\end{align*}
Since $||X||^2_P = \frac{1}{2}\trace(\tilde{X}^2)$, $||Y||^2_P = \frac{1}{2}\trace(\tilde{Y}^2)$, $\la X,Y\ra_P = \frac{1}{2}\trace(\tilde{X}\tilde{Y})$, we have
\begin{align*}
S_P(X,Y)= -\frac{\trace(\tilde{X}^2\tilde{Y}^2 - (\tilde{X}\tilde{Y})^2)}{\trace(\tilde{X}^2)\trace(\tilde{Y}^2) -[\trace(\tilde{X}\tilde{Y})]^2}.
\end{align*}
By the Cauchy-Schwarz inequality,
\begin{align*}
\trace(AB)^2 = \la BA, AB\ra_{\HS}\leq ||BA||_{\HS}||AB||_{\HS} &= \sqrt{\trace(ABBA)}\sqrt{\trace(BAAB)} 
\\
&= \trace(A^2B^2).
\end{align*}
Thus we always have $S_P(X,Y) \leq 0$.
\qed
\end{proof}

\begin{proposition}[Simply connectedness]
	\label{proposition:convex}
	The sets $\SymHS(\Hcal)_{< I}$ 
	and $\Tr(\Hcal,C_0)$ 
	are convex and hence simply connected. Thus $\SymHS(\Hcal)_{< I}$ and $\Tr(\Hcal,C_0)$ are simply connected Hilbert manifolds.
\end{proposition}
\begin{proof}
	(i) Let $S_1,S_2 \in \SymHS(\Hcal)_{<I}$ and $0 < \alpha < 1$. The assumption
	$I - S_1 > 0$, $I - S_2 > 0$ means that $\exists M_{S_1}, M_{S_2} \in \R$, $M_{S_1} > 0, M_{S_2} > 0$
	such that $\la x, (I-S_1)x\ra \geq M_{S_1}||x||^2$, $\la x, (I-S_2)x\ra \geq M_{S_2}||x||^2$ $\forall x \in \Hcal$.
	Then the convex combination $(1-\alpha)S_1 + \alpha S_2 \in \Sym(\Hcal)\cap \HS(\Hcal)$ and 
	\begin{align*}
		&\la x, [I-(1-\alpha)S_1 - \alpha S_2]x\ra = (1-\alpha)\la x, (I-S_1)x\ra + \alpha \la x, (I- S_2)x\ra  
		\\
		&\geq [(1-\alpha)M_{S_1}+ \alpha M_{S_2}]||x||^2
	\end{align*}
	$\forall x \in \Hcal$. Thus $(1-\alpha)S_1 + \alpha S_2 \in \SymHS(\Hcal)_{< I}$. Hence $\SymHS(\Hcal)_{< I}$ is a convex set in the Hilbert space
	$\Sym(\Hcal) \cap \HS(\Hcal)$ and thus is simply connected.
	
	(ii) Let $C_1 = C_0^{1/2}(I-S_1)C_0^{1/2}$, $C_2 = C_0^{1/2}(I-S_2)C_0^{1/2}$ $\in \Tr(\Hcal,C_0)$. For $0 < \alpha < 1$, we have
	$(1-\alpha)S_1 + \alpha S_2 \in \SymHS(\Hcal)_{< I}$ by part (i) and thus
	\begin{align*}
		&(1-\alpha)C_1 + \alpha C_2 =  (1-\alpha)C_0^{1/2}(I-S_1)C_0^{1/2} + \alpha C_0^{1/2}(I-S_2)C_0^{1/2}
		\\
		&= C_0^{1/2}[I - (1-\alpha)S_1 - \alpha S_2]C_0^{1/2} \in \Tr(\Hcal,C_0).
	\end{align*}
	Thus $\Tr(\Hcal,C_0)$ is a convex set in the Hilbert space $\SymHS_X(\Hcal,C_0)$ and hence is simply connected.
	\qed
\end{proof}

\subsection{Geodesic and Riemannian distance}
\label{section:geodesic}

\begin{lemma}
	\label{lemma:positive-definite-AA*-invertible}
	Let $A \in \Lcal(\Hcal)$ be invertible. Then $A^{*}A$ is positive definite.
\end{lemma}
\begin{proof} $A$ is invertible means that $A^{-1}\in \Lcal(\Hcal)$ and we have $||x|| = ||A^{-1}Ax|| \leq ||A^{-1}||\;||Ax|| \imply ||Ax|| \geq \frac{1}{||A^{-1}||}||x||$ $\forall x \in \Hcal$.
	Thus $\la x, A^{*}Ax\ra = ||Ax||^2 \geq \frac{1}{||A^{-1}||^2}||x||^2$ $\forall x \in \Hcal$, showing
	that $A^{*}A$ is positive definite.
	\qed
\end{proof}

\begin{lemma}
	Let $A\in \Sym(\Hcal)$ be positive definite. Let $B \in \Lcal(\Hcal)$ be invertible. Then $B^{*}AB$ is positive definite.
\end{lemma}
\begin{proof}
	Since $B$ is invertible, $B^{*}B$ is positive definite by Lemma \ref{lemma:positive-definite-AA*-invertible}, i.e. $\exists M_B > 0$ such that $\la x, B^{*}Bx\ra \geq M_B||x||^2$ $\forall x \in \Hcal$.
	$A$ is positive definite means that $\exists M_A > 0$ such that $\la x, Ax\ra \geq M_A||x||^2$ $\forall x \in \Hcal$. Thus $\forall x \in \Hcal$,
	\begin{align*}
		\la x, B^{*}ABx\ra = \la Bx, ABx\ra \geq M_A||Bx||^2 = M_A\la x, B^{*}Bx\ra \geq M_AM_B||x||^2,
	\end{align*}
	showing that $B^{*}AB \in \Sym(\Hcal)$ is positive definite.
	\qed
\end{proof}

\begin{lemma}
	\label{lemma:A-inverse-1/2-B}
	Let $A,B \in \Tr(\Hcal,C_0)$. Then $A^{-1/2}BA^{-1/2} = I-S$ for some $S\in \SymHS(\Hcal)_{<I}$.
	Thus $\log(A^{-1/2}BA^{-1/2}) \in \SymHS(\Hcal)$ is well-defined.
\end{lemma}
\begin{proof}
	By assumption, 
	the Gaussian measures $\Ncal(0,A)$ and $\Ncal(0,B)$ are equivalent.
	Thus $\exists S \in \SymHS(\Hcal)_{<I}$ such that $B = A^{1/2}(I-S)A^{1/2}$.
	By Lemma \ref{lemma:C0-inverse-Sigma-1}, $A^{-1/2}BA^{-1/2}$ is well-defined and $A^{-1/2}BA^{-1/2} = I-S$. It follows 
	that $\log(A^{-1/2}BA^{-1/2}) = \log(I-S) \in \SymHS(\Hcal)$ is well-defined.
	\qed
\end{proof}
The expression $A^{-1/2}BA^{-1/2}$ in Lemma \ref{lemma:A-inverse-1/2-B} can be equivalently expressed via
$C_0$ as follows. By assumption, 
$B = C_0^{1/2}(I-S_B)C_0^{1/2}$, $C_0 = A^{1/2}(I-T_A)C_0^{1/2}$, where
$S_B,T_A \in \SymHS(\Hcal)_{<I}$. By Lemma \ref{lemma:C0-Sigma-1/2-inverse-bounded}, $A^{-1/2}C_0^{1/2}\in \Lcal(\Hcal)$ and $(A^{-1/2}C_0^{1/2})^{*} = C_0^{1/2}A^{-1/2}$.
By Lemma \ref{lemma:Sigma-inverse-C0-1}, $A^{-1/2}C_0A^{-1/2} = I-T_A$.
Thus
\begin{align*}
	A^{-1/2}BA^{-1/2} &= (A^{-1/2}C_0^{1/2})(I-S_B)(C_0^{1/2}A^{-1/2})
	\\
	& = A^{-1/2}C_0A^{-1/2} - 
	(A^{-1/2}C_0^{1/2})S_B(A^{-1/2}C_0^{1/2})^{*}
	\\
	&= I-T_A - (A^{-1/2}C_0^{1/2})S_B(A^{-1/2}C_0^{1/2})^{*}.
\end{align*}
Thus in Lemma \ref{lemma:A-inverse-1/2-B}, $S = T_A + (A^{-1/2}C_0^{1/2})S_B(A^{-1/2}C_0^{1/2})^{*} \in \SymHS(\Hcal)_{<I}$.

\begin{proposition}
	[Geodesic curve and exponential map]
	\label{proposition:geodesic-curve}
	The Riemannian manifold $(\Tr(\Hcal,C_0), \la \;\ra_{\Sigma})$ is geodesically complete and is consequently a Cartan-Hadamard manifold.
	There is a unique geodesic curve $\gamma_{AB}:[0,1] \mapto \Tr(\Hcal,C_0)$ connecting $A,B \in \Tr(\Hcal,C_0)$, with $\gamma_A(0) = A$, $\gamma_B(1) = B$, given by
	\begin{align}
		\gamma_{AB}(t) = A^{1/2}\exp[t\log(A^{-1/2}BA^{-1/2})]A^{1/2}.
	\end{align}
	Equivalently, with $V \in T_A(\Tr(\Hcal,C_0)) = \SymHS(\Hcal,A)$,
	\begin{align}
		\gamma_V(t) = A^{1/2}\exp[tA^{-1/2}VA^{-1/2}]A^{1/2}, \;\; \gamma_V(0) = A, \;\; \dot{\gamma}_V(0) = V.
	\end{align}
The exponential map $\Exp_{A}: T_{A}(\Tr(\Hcal,C_0)) \mapto \Tr(\Hcal,C_0)$ is given by
\begin{align}
	\Exp_{A}(V) = \gamma_V(1) = A^{1/2}\exp(A^{-1/2}VA^{-1/2})A^{1/2},
\end{align}
Its inverse, the logarithm map $\Log_A: \Tr(\Hcal,C_0) \mapto T_{A}(\Tr(\Hcal,C_0))$ is given by
\begin{align}
	\Log_A(B) =  A^{1/2}\log(A^{-1/2}BA^{-1/2})A^{1/2}.
\end{align}
\end{proposition}
\begin{proof}
	By definition, a smooth curve $\gamma :\R\mapto\Tr(\Hcal,C_0)$ is a geodesic if and only if
	$\nabla_{\dot{\gamma}}\dot{\gamma}  = 0$. 
Assume that $\gamma$ is nontrivial, so that it is necessarily regular. For each fixed $t_0 \in \R$, let $\ep > 0$ be sufficiently small so that for $J = (t_0-\ep, t_0 + \ep)$ there exists an open set $U \subset \Tr(\Hcal,C_0)$ and a smooth vector field $\xi$ on $U$ with $\xi(\gamma(t)) = \dot{\gamma}(t)$ for $t \in  J$.
%
By Proposition \ref{proposition:Levi-Civita},
\begin{align*}
	(\nabla_{\dot{\gamma}}\xi)(\gamma(t_0)) = D\xi(\gamma(t_0))[\dot{\gamma}(t_0)] - \dot{\gamma}(t_0)\gamma(t_0)^{-1}\dot{\gamma}(t_0).
\end{align*}	
Consider the function $g: J \mapto \SymHS(\Hcal,C_0)$ defined by $g(t) = \xi({\gamma(t)}) = \dot{\gamma}(t):J \mapto \Tr(\Hcal,C_0) \mapto \SymHS(\Hcal,C_0)$.
For any $h \in \R$, by the chain rule,
\begin{align*}
	Dg(t_0)(h) = D\xi(\gamma(t_0)) \compose D\gamma(t_0)(h) = D\xi(\gamma(t_0))(D\gamma(t_0)(h)).
\end{align*}
Here $Dg(t_0): \R\mapto \SymHS(\Hcal,C_0)$, $D\xi(\gamma(t_0)): \SymHS(\Hcal,C_0) \mapto \SymHS(\Hcal,C_0)$, $D\gamma(t_0):\R \mapto \SymHS(\Hcal,C_0)$. We have, $\forall h \in \R$, $Dg(t_0)(h) = h\dot{g}(t_0) = h\ddot{\gamma}(t_0)$,
$D\gamma(t_0)(h) = h\dot{\gamma}(t_0)$. Thus
\begin{align*}
	D\xi(\gamma(t_0))[\dot{\gamma}(t_0)] = D\xi(\gamma(t_0))(D\gamma(t_0)(1)) = Dg(t_0)(1) = \ddot{\gamma}(t_0).
\end{align*}
It follows that the equation for the geodesic curve is
\begin{align}
	\label{equation:geodesic-curve}
\ddot{\gamma}(t) - \dot{\gamma}(t)\gamma(t)^{-1}\dot{\gamma}(t) = 0.
\end{align}
%
Since $\gamma(t) \in \Tr(\Hcal,C_0) = \Tr(\Hcal,A)$, it must have the form
$\gamma(t) = A^{1/2}(I-R(t))A^{1/2}$ for some $R(t) \in \SymHS(\Hcal)_{< I}$. The conditions $\gamma(0) = A$, $\gamma(1) = B$ means that
$R(0) = 0$ and $I-R(1) = A^{-1/2}BA^{-1/2}$.
Substituting $\dot{\gamma}(t) = -A^{1/2}\dot{R}(t)A^{1/2}$ and $\ddot{\gamma}(t) = -A^{1/2}\ddot{R}(t)A^{1/2}$
into Eq.\eqref{equation:geodesic-curve} gives
\begin{align*}
	-A^{1/2}\ddot{R}(t)A^{1/2} - A^{1/2}\dot{R}(t)(I-R(t))^{-1}\dot{R}(t)A^{1/2} = 0.
\end{align*}
Since $A \in \Sym^{++}(\Hcal)$, by Lemma \ref{lemma:cancellation-product-strictly-positive},
\begin{align*}
	\ddot{R}(t) +\dot{R}(t)(I-R(t))^{-1}\dot{R}(t) = 0.
\end{align*}
Right multiplying by $(I-R(t))^{-1}$ gives
\begin{align*}
	\ddot{R}(t)(I-R(t))^{-1} + \dot{R}(t)(I-R(t))^{-1}\dot{R}(t)(I-R(t))^{-1} &= 0
	\\ \equivalent \frac{d}{dt}[\dot{R}(t)(I-R(t))^{-1}] &= 0.
\end{align*}
This means there exists a constant operator $C \in \Lcal(\Hcal)$ such that
\begin{align*}
	\dot{R}(t)(I-R(t))^{-1} = C.
\end{align*}
By Lemma \ref{lemma:operator-equation} applied to $P(t) = I-R(t)$, so that $\dot{P}(t)P(t)^{-1} = - C$,
this equation has solution, with $D \in \Lcal(\Hcal)$ being constant and invertible,
\begin{align*}
	I-R(t) = \exp(-tC)D \equivalent R(t) = I-\exp(-tC)D.
\end{align*}
The first condition $R(0) = 0$ implies that $D = I$ and thus $I-R(t) = \exp(-tC)$.
The second condition $I-R(1) = \exp(-C) = A^{-1/2}BA^{-1/2}$ implies that
$C = - \log(A^{-1/2}BA^{-1/2})$. Thus $I-R(t) =  \exp[t\log(A^{-1/2}BA^{-1/2})]$
and finally $\gamma(t) = A^{1/2} \exp[t\log(A^{-1/2}BA^{-1/2})]A^{1/2}$.

Differentiating $\gamma_{AB}(t) = \gamma(t)$ gives 
$\dot{\gamma}_{AB}(0) = A^{1/2}\log(A^{-1/2}BA^{-1/2})A^{1/2}$ $\in T_A(\Tr(\Hcal,C_0)) = \SymHS(\Hcal,A)$.
Thus letting $V = A^{1/2}\log(A^{-1/2}BA^{-1/2})A^{1/2}$ gives the following equivalent expression for the geodesic curve
\begin{align*}
\gamma_{V}(t) = A^{1/2}\exp(tA^{-1/2}VA^{-1/2})A^{1/2}, \;\;\; \gamma_V(0) = A, \dot{\gamma}_V(0) = V.
\end{align*}
The exponential map $\Exp_{A}: T_{A}(\Tr(\Hcal,C_0)) \mapto \Tr(\Hcal,C_0)$ is then given by
\begin{align*}
\Exp_{A}(V) = \gamma_V(1) = A^{1/2}\exp(A^{-1/2}VA^{-1/2})A^{1/2},
\end{align*}
which is clearly well-defined over all $T_{A}(\Tr(\Hcal,C_0))$ and is moreover surjective.
Its inverse, the logarithm map $\Log_A: \Tr(\Hcal,C_0) \mapto T_{A}(\Tr(\Hcal,C_0))$ is given by
\begin{align*}
\Log_A(B) =  A^{1/2}\log(A^{-1/2}BA^{-1/2})A^{1/2}.
\end{align*}
The Riemannian manifold $(\Tr(\Hcal,C_0), \la \; \ra_{\Sigma})$ is thus geodesically complete. It is simply connected by Proposition \ref{proposition:convex} and with nonpositive sectional curvature everywhere by Proposition \ref{proposition:sectional-curvature}.
Thus it is a Cartan-Hadamard manifold.
\qed
\end{proof}

\begin{lemma}
\label{lemma:operator-equation}
	Let $\GL(\Hcal)$ be the set of invertible operators on $\Hcal$. Let $C \in \Lcal(\Hcal)$ be constant.
	Let $P:\R \mapto \GL(\Hcal)$. 
The operator equation
\begin{align}
\dot{P}(t)P(t)^{-1} = C
\end{align}
has solution of the form, with $D \in \Lcal(\Hcal)$ being constant and invertible,
\begin{align}
P(t) = \exp(tC)D.
\end{align}
\end{lemma}
\begin{proof}
	First, $P(t) = \exp(tC)D$ is clearly a solution, since $\dot{P}(t) = C\exp(tC)D$, $P(t)^{-1} = D^{-1}\exp(-tC)$, so that $\dot{P}(t)P(t)^{-1} = C$.
 Second, let $S(t)$ be any solution, i.e., $\dot{S}(t) = CS(t)$. Define $R(t) = \exp(-tC)S(t)$, then
 \begin{align*}
 \dot{R}(t) = -\exp(-tC)CS(t) + \exp(-tC)\dot{S}(t) = 0.
 \end{align*}
Thus we must have $R(t) = \exp(-tC)S(t) = D$ for a constant, invertible operator $D \in \Lcal(\Hcal)$, so that $S(t) = \exp(tC)D$.
	\qed
\end{proof}

\begin{proposition}
	[Riemannian distance]
	\label{proposition:Riemannian distance}
	The Riemannian distance between $A,B \in (\Tr(\Hcal,C_0), \la \;\ra_{\Sigma})$ is given by
	\begin{align}
	d(A,B) = \frac{1}{\sqrt{2}}||\log(A^{-1/2}BA^{-1/2})||_{\HS}.
	\end{align}
\end{proposition}
\begin{proof}
	By the assumption $A,B \in \Tr(\Hcal,C_0)$, $\exists S \in \SymHS(\Hcal)_{<I}$ 
	such that $B = A^{1/2}(I-S)A^{1/2}$, with $A^{-1/2}BA^{-1/2} = I-S$.
	Consider the geodesic curve $\gamma(t) = A^{1/2}\exp[t\log(A^{-1/2}BA^{-1/2})]A^{1/2} = A^{1/2}\exp(t\log(I-S))A^{1/2} = A^{1/2}(I-S)^tA^{1/2}$
	connecting $A,B$, with
	$\gamma(0) = A$, $\gamma(1) = B$. 
	As in the proof of Lemma \ref{lemma:switch-C0-Sigma}, consider the polar decomposition $(I-S)^{t/2}A^{1/2} = U(A^{1/2}(I-S)^tA^{1/2})^{1/2}$, where
	$U$ is unitary since $\ker(A) = \{0\}$. Then
	\begin{align*}
		\gamma(t)^{1/2} = (A^{1/2}(I-S)^tA^{1/2})^{1/2} = U^{*}(I-S)^{t/2}A^{1/2} = A^{1/2}(I-S)^{t/2}U.
	\end{align*}
The tangent curve $\dot{\gamma}(t) \in T_{\gamma(t)}(\Tr(\Hcal,C_0))$ is then given by
\begin{align*}
&\dot{\gamma}(t) = A^{1/2}(I-S)^t\log(I-S)A^{1/2} = A^{1/2}\log(I-S)(I-S)^tA^{1/2}
\\
& = A^{1/2}(I-S)^{t/2}\log(I-S)(I-S)^{t/2}A^{1/2}
=\gamma(t)^{1/2}U^{*}\log(I-S)U\gamma(t)^{1/2}.
\end{align*}
Here $U^{*}\log(I-S)U \in \Sym(\Hcal)\cap \HS(\Hcal)$, since $\log(I-S) \in \Sym(\Hcal) \cap\HS(\Hcal)$.
Since $(\Tr(\Hcal,C_0), \la \; \ra_{\Sigma})$ is a Cartan-Hadamard manifold by Proposition \ref{proposition:geodesic-curve},
the Riemannian distance between $A$ and $B$ is the length of the geodesic curve $\gamma(t)$ joining them (see e.g. \cite{lang2012fundamentals}, Corollary IX.3.11)
\begin{align*}
&	d(A,B) = L(\gamma(t)) = \int_{0}^1||\dot{\gamma}(t)||_{\gamma(t)}dt = \int_{0}^1\sqrt{\frac{1}{2}\trace[(\gamma(t)^{-1/2}\dot{\gamma}(t)\gamma(t)^{-1/2})^2]}dt
\\
& 	= \frac{1}{\sqrt{2}}\sqrt{\trace[U^{*}(\log(I-S))^2U]} = \frac{1}{\sqrt{2}}||\log(I-S)||_{\HS}.
\end{align*}
\qed
\end{proof}

\subsection{Proof of the connection with the Riemannian distance of positive definite Hilbert-Schmidt operators}
\label{section:positive-definite-HS-operators}


The following is a special case of Theorem 1 in \cite{Minh:2021RiemannianEstimation}.
\begin{lemma}
	\label{lemma:logHS-convergence}
Let $\{A_k\}$, $A$ $\in \HS(\Hcal$) be  such
that $I+A, I+A_k > 0$ $\forall k \in \Nbb$.
Assume that $\lim\limits_{k \approach \infty}||A_k - A||_{\HS} = 0$. Then
\begin{align}
	\lim_{k \approach \infty}||\log(I+A_k) - \log(I+A)||_{\HS} = 0.
\end{align}
\end{lemma}
\begin{proof}
	[\textbf{of Theorem \ref{theorem:limit-affineLogHS}}]
	For $B = A^{1/2}(I-S)A^{1/2} = A - A^{1/2}SA^{1/2}$, we have $B+ \gamma I = A+\gamma I - A^{1/2}SA^{1/2}$. It follows that
	\begin{align*}
	(A+\gamma I)^{-1/2}(B+\gamma I)(A+\gamma I)^{-1/2} = I - (A+\gamma I)^{-1/2}A^{1/2}SA^{1/2}(A+\gamma I)^{-1/2}.
	\end{align*}
	By Proposition 4.6 in \cite{Minh:2020regularizedDiv}, we have
\begin{align*}
	\lim_{\gamma \approach 0^{+}} ||(A + \gamma I)^{-1/2}A^{1/2}SA^{1/2}(A + \gamma I)^{-1/2} - S||_{\HS} = 0.
\end{align*}
By Lemma \ref{lemma:logHS-convergence}, it follows that
\begin{align*}
\lim_{\gamma \approach 0^{+}}||\log[I - (A+\gamma I)^{-1/2}A^{1/2}SA^{1/2}(A+\gamma I)^{-1/2}]-\log(I-S)||_{\HS} = 0,
\end{align*}
from which the desired result follows.
\qed
\end{proof}





\subsection{Direct computation of the Fisher-Rao metric}
\label{section:computation-Fisher-metric}

This section
presents the direct proof of Theorem \ref{theorem:Riemannian-metric}, i.e. the
direct computation of the Fisher-Rao metric on $\Gauss(\Hcal,\mu_0)$, as defined in 
Eq.\eqref{equation:Fisher-metric-definition}. To this end, we first prove Theorem \ref{theorem:derivative-log-RN-general}, that is $\forall \mu \in \Gauss(\Hcal,\mu_0)$,
the Fr\'echet logarithmic derivative $D\log\frac{d\mu}{d\mu_0}$ exists, and present its explicit expression.

For the expression of the Radon-Nikodym density between two equivalent Gaussian measures on 
$\Hcal$, 
we utilize the concept of {\it white noise mapping}, see e.g. \cite{DaPrato:2006,DaPrato:PDEHilbert}.
For $\mu = \Ncal(m, Q)$, $\ker(Q) = \{0\}$, we  define
$\Lcal^2(\Hcal, \mu) = \Lcal^2(\Hcal, \Bsc(\Hcal),\mu) = \Lcal^2(\Hcal, \Bsc(\Hcal), \Ncal(m,Q))$.
Consider the following mapping
\begin{align}
	&W:Q^{1/2}(\Hcal) \subset \Hcal \mapto \Lcal^2(\Hcal,\mu), \;\; z  \in Q^{1/2}(\Hcal) \mapto W_z \in \Lcal^2(\Hcal, \mu),
	\\
	&W_z(x) = \la x -m, Q^{-1/2}z\ra,  \;\;\; z \in Q^{1/2}(\Hcal), x \in \Hcal.
\end{align}
For any pair $z_1, z_2 \in Q^{1/2}(\Hcal)$, it follows from the definition of covariance operator
that $\la W_{z_1}, W_{z_2}\ra_{\Lcal^2(\Hcal,\mu)} = \la z_1, z_2\ra_{\Hcal}$ 
so that the map $W:Q^{1/2}(\Hcal) \mapto \Lcal^2(\Hcal, \mu)$ is an isometry, that is
	$||W_z||_{\Lcal^2(\Hcal,\mu)} = ||z||_{\Hcal}$, $z \in Q^{1/2}(\Hcal)$.
Since $\ker(Q) = \{0\}$, the subspace $Q^{1/2}(\Hcal)$ is dense in $\Hcal$ and the map $W$ can be uniquely extended to all of $\Hcal$, as follows.
For any $z \in \Hcal$, let $\{z_n\}_{n\in \Nbb}$ be a sequence in $Q^{1/2}(\Hcal)$ with $\lim_{n \approach \infty}||z_n -z||_{\Hcal} = 0$.
Then $\{z_n\}_{n \in \Nbb}$ is a Cauchy sequence in $\Hcal$, so that by isometry, $\{W_{z_n}\}_{n\in \Nbb}$ is also
a Cauchy sequence in $\Lcal^2(\Hcal, \mu)$, thus converging to a unique element in $\Lcal^2(\Hcal, \mu)$.
Thus
we can define
\begin{align}
	W: \Hcal \mapto \Lcal^2(\Hcal, \mu),  \;\;\; z \in \Hcal \mapto \Lcal^2(\Hcal, \mu)
\end{align}
by the following unique limit in $\Lcal^2(\Hcal, \mu)$
\begin{align}
	W_z(x) = \lim_{n \approach \infty}W_{z_n}(x) = \lim_{n \approach \infty}\la x-m, Q^{-1/2}z_n\ra.
\end{align}
The map $W: \Hcal \mapto \Lcal^2(\Hcal, \mu)$ is called the {\it white noise mapping}
associated with the measure $\mu = \Ncal(m,Q)$.

Let $\mu = \Ncal(m_1, Q)$, $\nu = \Ncal(m_2,R)$ be equivalent, with $R = Q^{1/2}(I-S)Q^{1/2}$, $S \in \SymHS(\Hcal)_{<I}$.
Let $\{\alpha_k\}_{k \in \Nbb}$ be the eigenvalues of $S$, with corresponding orthonormal eigenvectors $\{\phi_k\}_{k \in \Nbb}$, which form an orthonormal basis in $\Hcal$.
The following result expresses the Radon-Nikodym density $\frac{d\nu}{d\mu}$ in terms of the $\alpha_k$'s and
$\phi_k$'s. Here the white noise mapping is $W: \Hcal \mapto \Lcal^2(\Hcal, \mu) = \Lcal^2(\Hcal, \Ncal(m_1,Q))$,
with $\{W_{\phi_k}\}_{k=1}^{\infty}$ forming an orthonormal sequence in $\Lcal^2(\Hcal,\mu)$.


\begin{theorem}
	[\cite{Minh:2020regularizedDiv}, Theorem 11 and Corollary 2]
	\label{theorem:radon-nikodym-infinite}
	Let $\mu = \Ncal(m_1, Q)$, $\nu = \Ncal(m_2,R)$, with $m_2 - m_1 \in \range(Q^{1/2})$, $R = Q^{1/2}(I-S)Q^{1/2}$, $S \in \SymHS(\Hcal)_{<I}$.
	The Radon-Nikodym density $\frac{d\nu}{d\mu}$
	is given by
	\begin{align}
		\label{equation:RN-infinite}
		\frac{d\nu}{d\mu}(x) = \exp\left[-\frac{1}{2}\sum_{k=1}^{\infty}\Phi_k(x)\right]\exp\left[-\frac{1}{2}||(I-S)^{-1/2}Q^{-1/2}(m_2 - m_1)||^2\right],
	\end{align}
	where for each $k \in \Nbb$
	\begin{align}
		\label{equation:Phik}
		\Phi_k = \frac{\alpha_k}{1-\alpha_k}W^2_{\phi_k} - \frac{2}{1-\alpha_k}\la Q^{-1/2}(m_2-m_1), \phi_k\ra W_{\phi_k}+ \log(1-\alpha_k).
	\end{align}
	The series $\sum_{k=1}^{\infty}\Phi_k$ converges in 
	$\Lcal^2(\Hcal,\mu)$ and the function $s(x) = \exp\left[-\frac{1}{2}\sum_{k=1}^{\infty}\Phi_k(x)\right] \in \Lcal^1(\Hcal, \mu)$.
\end{theorem}
For the case $m_1 = m_2 = 0$, this is Corollary 6.4.11 in \cite{Bogachev:Gaussian}.
In particular, for the case $S \in \Sym(\Hcal) \cap \Tr(\Hcal)$, $I -S > 0$, we have
\begin{align}
	\label{equation:RN-traceclass}
	\frac{d\nu}{d\mu}(x) &= [\det(I-S)]^{-1/2}
	\\
	& \times \exp\left\{-\frac{1}{2}\la Q^{-1/2}(x-m_1), S(I-S)^{-1}Q^{-1/2}(x-m_1)\ra \right\}
	\nonumber
	\\
	&\times \exp(\la Q^{-1/2}(x-m_1), (I-S)^{-1}Q^{-1/2}(m_2 - m_1)\ra)
	\nonumber
	\\
	&\times \exp\left[-\frac{1}{2}||(I-S)^{-1/2}Q^{-1/2}(m_2 - m_1)||^2\right].
	\nonumber
\end{align}
Let $P_N = \sum_{k=1}^Ne_k \otimes e_k$, $N \in \Nbb$, be the orthogonal projection onto the $N$-dimensional subspace of $\Hcal$ spanned
by $\{e_k\}_{k=1}^N$, where $\{e_k\}_{k=1}^{\infty}$ are the orthonormal eigenvectors of $Q$.
In the above expression,
\begin{align}
	&\la Q^{-1/2}(x-m_1), S(I-S)^{-1}Q^{-1/2}(x-m_1)\ra 
	\nonumber
	\\
	& \doteq \lim_{N \approach \infty} \la Q^{-1/2}P_N(x-m_1), S(I-S)^{-1}Q^{-1/2}P_N(x-m_1)\ra 
	\\
	&
	\la Q^{-1/2}(x-m_1), (I-S)^{-1}Q^{-1/2}(m_2 - m_1)\ra 
	\nonumber
	\\
	&\doteq \lim_{N \approach \infty} \la Q^{-1/2}P_N(x-m_1), (I-S)^{-1}Q^{-1/2}(m_2 - m_1)\ra,
\end{align}
with the limits being in the 
$\Lcal^2(\Hcal, \mu)$ sense (see Proposition \ref{proposition:RN-S-trace-class-quadratic-form-convergence}).
For the case $m_1=m_2 = 0$, we obtain the expression given in Proposition 1.3.11 in \cite{DaPrato:PDEHilbert}
\begin{align}
	\label{equation:RN-m1m2-0}
	\frac{d\nu}{d\mu}(x) = [\det(I-S)]^{-1/2}\exp\left\{-\frac{1}{2}\la Q^{-1/2}x, S(I-S)^{-1}Q^{-1/2}x\ra \right\}.
\end{align}
It follows that for $S \in \SymHS(\Hcal)_{<I}$,
\begin{align}
	\log\left\{{\frac{d\nu}{d\mu}(x)}\right\} = -\frac{1}{2}\sum_{k=1}^{\infty}\Phi_k(x)-\frac{1}{2}||(I-S)^{-1/2}Q^{-1/2}(m_2 - m_1)||^2.
\end{align}
In the case $m_1 = m_2 = 0$,
\begin{align}
\label{equation:log-Radon-Nikodym-S-HS-zero-mean}
\log\left\{{\frac{d\nu}{d\mu}(x)}\right\} = -\frac{1}{2}\sum_{k=1}^{\infty}\Phi_k(x) = -\frac{1}{2}\sum_{k=1}^{\infty}\left[\frac{\alpha_k}{1-\alpha_k}W^2_{\phi_k}(x) + \log(1-\alpha_k)\right].
\end{align}
In particular, for $S \in \SymTr(\Hcal)_{<I}$, with $m_1 = m_2 = 0$,
\begin{align}
\label{equation:log-Radon-Nikodym-S-trace-class-zero-mean}
\log\left\{{\frac{d\nu}{d\mu}(x)}\right\} = - \frac{1}{2}\log\det(I-S) - \frac{1}{2}\la Q^{-1/2}x, S(I-S)^{-1}Q^{-1/2}x\ra.
\end{align}

The following is a stronger and more general version of Lemma 17 and Corollary 2 in \cite{Minh:2020regularizedDiv}, which are obtained by setting $T = S(I-S)^{-1}$, $S \in \SymTr(\Hcal)_{< I}$, with convergence in $\Lcal^1(\Hcal,\mu_0)$ being strengthened to convergence in $\Lcal^2(\Hcal,\mu_0)$.

\begin{proposition}
	\label{proposition:RN-S-trace-class-quadratic-form-convergence}
	Let $\mu_0 = \Ncal(m_0, C_0)$ be fixed, with the associated white noise mapping $W:\Hcal \mapto \Lcal^2(\Hcal,\mu_0)$.
	Assume that $T \in \Sym(\Hcal) \cap \Tr(\Hcal)$, with eigenvalues $\{\alpha_k\}_{k \in \Nbb}$ and
	corresponding orthonormal eigenvectors $\{\phi_k\}_{k \in \Nbb}$. Then in the $\Lcal^2(\Hcal,\mu_0)$ sense,
	\begin{align}
		\lim_{N \approach \infty}\sum_{k=1}^{\infty}{\alpha_k}W^2_{P_N\phi_k} = \sum_{k=1}^{\infty}{\alpha_k}W^2_{\phi_k}.
	\end{align}
	Let $P_N = \sum_{j=1}^Ne_j \otimes e_j$, with $\{e_j\}_{j\in \Nbb}$ being the orthonormal eigenvectors of $C_0$, then, consequently, with the limit being taken in the $\Lcal^2(\Hcal,\mu_{0})$ sense, 
	\begin{align*}
		\sum_{k=1}^{\infty}{\alpha_k}W^2_{\phi_k}(x)
		& = \lim_{N \approach \infty}\la C_0^{-1/2}P_N(x-m_0), TC_0^{-1/2}P_N(x-m_0)\ra
		\\
		&\doteq\la C_0^{-1/2}(x-m_0), TC_0^{-1/2}(x-m_0)\ra.
	\end{align*}
\end{proposition}
\begin{proof}
	(i) For the first part, we first note that, since $T \in \Sym(\Hcal) \cap \Tr(\Hcal)$, 
	\begin{align*}
		&||T||_{\tr} = \sum_{j=1}^{\infty}\la e_j, |T|e_j\ra 
		= \sum_{j=1}^{\infty}\la e_j, \sum_{k=1}^{\infty}\left|{\alpha_k}\right| (\phi_k \otimes \phi_k) e_j\ra
		\\
		& = \sum_{j=1}^{\infty}\sum_{k=1}^{\infty}\left|{\alpha_k}\right| \la \phi_k, e_j\ra^2 
		= \sum_{k=1}^{\infty}\left|{\alpha_k}\right|< \infty
		\\
		&\imply \sum_{j=N+1}^{\infty}\sum_{k=1}^{\infty}\left|{\alpha_k}\right| \la \phi_k, e_j\ra^2 \approach 0 \; \text{as} \;  N \approach \infty.
	\end{align*}
	Furthermore, with $T \in \Tr(\Hcal)$, we can rewrite the series
	\begin{align*}
		\sum_{k=1}^{\infty}{\alpha_k}W^2_{\phi_k} &= \sum_{k=1}^{\infty}{\alpha_k}(W^2_{\phi_k}-1)
		+ \sum_{k=1}^{\infty}{\alpha_k} 
		= \sum_{k=1}^{\infty}{\alpha_k}(W^2_{\phi_k}-1) + \trace[T].
	\end{align*}
	By Proposition 9 in \cite{Minh:2020regularizedDiv}, the functions $\{\frac{1}{\sqrt{2}}(W^2_{\phi_k}-1)\}_{k \in \Nbb}$ are orthonormal in $\Lcal^2(\Hcal,\mu_0)$. Thus
	\begin{align*}
		\left\|\sum_{k=1}^{\infty}{\alpha_k}(W^2_{\phi_k}-1)\right\|^2_{\Lcal^2(\Hcal,\mu_0)} = 2 \sum_{k=1}^{\infty}{\alpha_k^2} = 2||T||^2_{\HS} < \infty.
	\end{align*}
	It follows that $\sum_{k=1}^{\infty}{\alpha_k}W^2_{\phi_k} \in \Lcal^2(\Hcal,\mu_0)$.
	By Lemma 13 in \cite{Minh:2020regularizedDiv}, $\forall a,b \in \Hcal$,
	\begin{align}
		\int_{\Hcal}W^2_a(x)W^2_b(x)\mu_0(dx) &= ||a||^2||b||^2 + 2\la a,b\ra^2,
		\\
		\int_{\Hcal}W^4_a(x)\mu_0(dx) &= 3||a||^4.
	\end{align}
	Applying these identities, noting that $\la P_N\phi_k, \phi_k\ra = ||P_N\phi_k||^2 \leq ||\phi_k||^2 =1$,
	\begin{align*}
		&\int_{\Hcal}|W^2_{P_N\phi_k}(x) - W^2_{\phi_k}(x)|^2\mu_0(dx) = \int_{\Hcal}[W^4_{P_N\phi_k} - 2 W^2_{P_N\phi_k}W^2_{\phi_k} + W^4_{\phi_k}]\mu_0(dx)
		\\
		& = 3||P_N\phi_k||^4 - 2[||P_N\phi_k||^2||\phi_k||^2 + 2\la P_N\phi_k, \phi_k\ra^2] + 3||\phi_k||^4
		\\
		& = 3||\phi_k||^4 - 2||P_N\phi_k||^2||\phi_k||^2 - ||P_N\phi_k||^4 \leq 3(||\phi_k||^4 - ||P_N\phi_k||^4)
		\\
		&  =3(||\phi_k||^2 - ||P_N\phi_k||^2)(||\phi_k||^2 + ||P_N\phi_k||^2) \leq 6(||\phi_k||^2 - ||P_N\phi_k||^2)
		\\
		& = 6\sum_{j=N+1}^{\infty}\la \phi_k, e_j\ra^2.
	\end{align*}
	It follows that
	\begin{align*}
		&\int_{\Hcal}\left|\sum_{k=1}^{\infty}{\alpha_k}W^2_{P_N\phi_k}(x)- \sum_{k=1}^{\infty}{\alpha_k}W^2_{\phi_k}(x)\right|^2\mu_0(dx) 
		\\
		&\leq \int_{\Hcal}\left(\sum_{k=1}^{\infty}\left|{\alpha_k}\right|\left|[W^2_{P_N\phi_k}(x) - W^2_{\Phi_k}(x)]\right|\right)^2\mu_0(dx)
		\\
		&\leq \sum_{k=1}^{\infty}\left|{\alpha_k}\right|\sum_{k=1}^{\infty}\left|{\alpha_k}\right|\int_{\Hcal}\left|[W^2_{P_N\phi_k}(x) - W^2_{\Phi_k}(x)]\right|^2\mu_0(dx)
		\\
		&
		\leq 6\sum_{k=1}^{\infty}\left|{\alpha_k}\right|\sum_{k=1}^{\infty}\left|{\alpha_k}\right|\;
		\sum_{j=N+1}^{\infty}\la \phi_k,e_j\ra^2
		\\
		& = 6||T||_{\tr}\sum_{j=N+1}^{\infty}\sum_{k=1}^{\infty}|\alpha_k|\la \phi_k, e_j\ra^2 \approach 0 \;\text{as $N \approach \infty$}.
	\end{align*}
	(ii) For the second part, 
	from 
	$T = \sum_{k=1}^{\infty}{\alpha_k} \phi_k \otimes \phi_k$,
	we have $\forall N \in \Nbb$,
	\begin{align*}
		& \la C_0^{-1/2}P_N(x-m_0), TC_0^{-1/2}P_N(x-m_0)\ra 
		\\
		&= \sum_{k=1}^{\infty}{\alpha_k}\la C_0^{-1/2}P_N(x-m_0), \phi_k\ra^2
		\\
		& = \sum_{k=1}^{\infty}{\alpha_k}\la x- m_0, C_0^{-1/2}P_N\phi_k\ra^2 
		= \sum_{k=1}^{\infty}{\alpha_k}W^2_{P_N\phi_k}(x).
	\end{align*}
	By the first part,
	taking limit as $N \approach \infty$ gives, where the limit is in $\Lcal^2(\Hcal,\mu_0)$, 
	\begin{align*}
		&\sum_{k=1}^{\infty}{\alpha_k}W^2_{\phi_k}(x)
		= \lim_{N \approach \infty} \sum_{k=1}^{\infty}{\alpha_k}W^2_{P_N\phi_k}(x)
		\\
		& = \lim_{N \approach \infty}\la C_0^{-1/2}P_N(x-m_0), TC_0^{-1/2}P_N(x-m_0)\ra
		\\
		&\doteq\la C_0^{-1/2}(x-m_0), TC_0^{-1/2}(x-m_0)\ra.
	\end{align*}
	\qed
\end{proof}

In the case $m_0=0$, the following generalizes Proposition \ref{proposition:RN-S-trace-class-quadratic-form-convergence}
so that the convergence is in $\Lcal^2(\Hcal, \mu_{*})$ for any Gaussian measure $\mu_{*}$ equivalent to $\mu_0$.
\begin{proposition}
	\label{proposition:RN-S-trace-class-quadratic-form-convergence-L2-mu-star}
	Let $\mu_0= \Ncal(0, C_0)$ be fixed, with the associated white noise mapping $W:\Hcal \mapto \Lcal^2(\Hcal,\mu_0)$.
	Let $S_{*} \in \SymHS(\Hcal)_{<I}$ be fixed but arbitrary, with corresponding Gaussian measure $\mu_{*} = \Ncal(0, C_{*})$, with $C_{*} = C_0^{1/2}(I-S_{*})C_0^{1/2}$.
	Assume that $T \in \Sym(\Hcal) \cap \Tr(\Hcal)$, with eigenvalues $\{\alpha_k\}_{k \in \Nbb}$ and
	corresponding orthonormal eigenvectors $\{\phi_k\}_{k \in \Nbb}$. Then in the $\Lcal^2(\Hcal,\mu_{*})$ sense,
	\begin{align}
		\lim_{N \approach \infty}\sum_{k=1}^{\infty}{\alpha_k}W^2_{P_N\phi_k} = \sum_{k=1}^{\infty}{\alpha_k}W^2_{\phi_k}.
	\end{align}
	Let $P_N = \sum_{j=1}^Ne_j \otimes e_j$, with $\{e_j\}_{j\in \Nbb}$ being the orthonormal eigenvectors of $C_0$, then, consequently, with the limit being taken in the $\Lcal^2(\Hcal,\mu_{*})$ sense,
	\begin{align*}
		\sum_{k=1}^{\infty}{\alpha_k}W^2_{\phi_k}(x)
		& = \lim_{N \approach \infty}\la C_0^{-1/2}P_Nx, TC_{0}^{-1/2}P_Nx\ra
		\doteq\la C_0^{-1/2}x, TC_0^{-1/2}x\ra.
	\end{align*}
\end{proposition}

\begin{proof}
	(i) As in the proof of Proposition \ref{proposition:RN-S-trace-class-quadratic-form-convergence},  
$\sum_{j=N+1}^{\infty}\sum_{k=1}^{\infty}\left|{\alpha_k}\right| \la \phi_k, e_j\ra^2 \approach 0$ as $N \approach \infty$
	for $T \in \Sym(\Hcal) \cap \Tr(\Hcal)$.
By Lemma 22 in \cite{Minh:2020regularizedDiv}, for any $a,b \in \Hcal$,
\begin{align}
	&\int_{\Hcal}W^2_a(x)W^2_b(x)\mu_{*}(dx) = \la a, (I-S_{*})a\ra \la b, (I-S_{*})b\ra + 2 \la a, (I-S_{*})b\ra^2, 
	\nonumber
	\\
	& = ||(I-S_{*})^{1/2}a||^2||(I-S_{*})^{1/2}b||^2 + 2\la (I-S_{*})^{1/2}a,(I-S_{*})^{1/2}b\ra^2,
	\label{equation:W2aW2b-mu-star}
	\\
&	\int_{\Hcal}W^4_a(x)\mu(dx) = 3||(I-S_{*})^{1/2}a||^4.
\label{equation:W4a-mu-star}
\end{align}
Applying Eq. \eqref{equation:W4a-mu-star} and noting that $||\phi_k||=1$, we obtain
\begin{align*}
\left\|\sum_{k=1}^{\infty}\alpha_kW^2_{\phi_k}\right\|_{\Lcal^2(\Hcal,\mu_{*})} &\leq 3\sum_{k=1}^{\infty}|\alpha_k|\;||(I-S_{*})^{1/2}\phi_k||^4\leq 3||I-S_{*}||^2\sum_{k=1}^{\infty}|\alpha_k|\;||\phi_k||^4
\\
& = 3||I-S_{*}||^2||T||_{\trace} < \infty.
\end{align*}
Thus  $\sum_{k=1}^{\infty}\alpha_kW^2_{\phi_k} \in \Lcal^2(\Hcal,\mu_{*})$.
Applying Eqs.\eqref{equation:W2aW2b-mu-star} and \eqref{equation:W4a-mu-star} and Lemma \ref{lemma:fourth-power-quadratic-bound},
noting that $||P_N\phi_k|| \leq ||\phi_k|| =1$, we obtain
\begin{align*}
	&\int_{\Hcal}|W^2_{P_N\phi_k}(x) - W^2_{\phi_k}(x)|^2\mu_{*}(dx) = \int_{\Hcal}[W^4_{P_N\phi_k} - 2 W^2_{P_N\phi_k}W^2_{\phi_k} + W^4_{\phi_k}]\mu_{*}(dx)
	\\
	& = 3||(I-S_{*})^{1/2}P_N\phi_k||^4 + 3||(I-S_{*})^{1/2}\phi_k||^4
	\\
	& \quad - 2[||(I-S_{*})^{1/2}P_N\phi_k||^2||(I-S_{*})^{1/2}\phi_k||^2 + 2\la (I-S_{*})^{1/2}P_N\phi_k, (I-S_{*})^{1/2}\phi_k\ra^2] 
	\\
	& \leq 3[||(I-S_{*})^{1/2}P_N\phi_k||^4 + ||(I-S_{*})^{1/2}\phi_k||^4 - 2\la (I-S_{*})^{1/2}P_N\phi_k, (I-S_{*})^{1/2}\phi_k\ra^2]
	\\
	& \leq 3||(I-S_{*})^{1/2}P_N\phi_k - (I-S_{*})^{1/2}\phi_k||^2[||(I-S_{*})^{1/2}P_N\phi_k|| + ||(I-S_{*})^{1/2}\phi_k||]^2
	\\
	& \leq 3||I-S_{*}||^2||P_N\phi_k - \phi_k||^2[||P_N\phi_k|| + ||\phi_k||]^2
	\\
	& \leq 12 ||I-S_{*}||^2||P_N\phi_k - \phi_k||^2 = 12||I-S_{*}||^2\sum_{j=N+1}^{\infty}\la \phi_k, e_j\ra^2.
\end{align*}
It follows that
\begin{align*}
	&\int_{\Hcal}\left|\sum_{k=1}^{\infty}{\alpha_k}W^2_{P_N\phi_k}(x)- \sum_{k=1}^{\infty}{\alpha_k}W^2_{\phi_k}(x)\right|^2\mu_{*}(dx) 
	\\
	&\leq \int_{\Hcal}\left(\sum_{k=1}^{\infty}\left|{\alpha_k}\right|\left|[W^2_{P_N\phi_k}(x) - W^2_{\Phi_k}(x)]\right|\right)^2\mu_{*}(dx)
	\\
	&\leq \sum_{k=1}^{\infty}\left|{\alpha_k}\right|\sum_{k=1}^{\infty}\left|{\alpha_k}\right|\int_{\Hcal}\left|[W^2_{P_N\phi_k}(x) - W^2_{\Phi_k}(x)]\right|^2\mu_{*}(dx)
	\\
	&
	\leq 12||I-S_{*}||^2\sum_{k=1}^{\infty}\left|{\alpha_k}\right|\sum_{k=1}^{\infty}\left|{\alpha_k}\right|\;
	\sum_{j=N+1}^{\infty}\la \phi_k,e_j\ra^2
	\\
	& = 12||I-S_{*}||^2||T||_{\tr}\sum_{j=N+1}^{\infty}\sum_{k=1}^{\infty}|\alpha_k|\la \phi_k, e_j\ra^2 \approach 0 \;\text{as $N \approach \infty$}.
\end{align*}
	(ii) For the second part, as in the proof of Proposition \ref{proposition:RN-S-trace-class-quadratic-form-convergence},
	taking limit in $\Lcal^2(\Hcal,\mu_{*})$ as $N \approach \infty$ gives
	\begin{align*}
		&\sum_{k=1}^{\infty}{\alpha_k}W^2_{\phi_k}(x)
		= \lim_{N \approach \infty} \sum_{k=1}^{\infty}{\alpha_k}W^2_{P_N\phi_k}(x)
		\\
		& = \lim_{N \approach \infty}\la C_0^{-1/2}P_Nx, TC_0^{-1/2}P_Nx\ra
		\doteq\la C_0^{-1/2}x, TC_0^{-1/2}x\ra.
	\end{align*}
	\qed
\end{proof}

We first show two cases in which $D\log\left\{\frac{d\mu}{d\mu_0}(x) \right\}(S)(V)$ admits an explicit expression
in terms of $S$ and $V$, namely
\begin{enumerate}
	\item $S \in \SymTr(\Hcal)_{<I}$, $V \in \Sym(\Hcal) \cap \Tr(\Hcal)$.
	\item $S \in \SymHS(\Hcal)_{<I}$, $V \in \Sym(\Hcal) \cap \HS(\Hcal)$, and $S$ and $V$ commute.
\end{enumerate}

Consider the first case, namely $S \in \SymTr(\Hcal)_{<I}$, $V \in \Sym(\Hcal) \cap \Tr(\Hcal)$.

\begin{proposition}
	\label{proposition:derivative-log-Radon-Nikodym-S-trace-class}
	Let $\mu =\Ncal(0,C), \mu_0 = \Ncal(0,C_0)$, with $C = C_0^{1/2}(I-S)C_0^{1/2}$, $S \in \SymTr(\Hcal)_{<I}$. Then
	$\forall V \in \Sym(\Hcal) \cap \Tr(\Hcal)$,
	\begin{align}
		\label{equation:derivative-log-Radon-Nikodym-S-trace-class}
		&D\log\left\{\frac{d\mu}{d\mu_0}(x) \right\}(S)(V) 
		\nonumber
		\\
		&= \frac{1}{2}\trace\left[(I-S)^{-1}V\right] - \frac{1}{2}\la C_0^{-1/2}x, (I-S)^{-1}V{(I-S)^{-1}}C_0^{-1/2}x\ra.
	\end{align}
Here, with the limit being taken in $\Lcal^2(\Hcal,\mu_0)$
\begin{align}
&\la C_0^{-1/2}x, (I-S)^{-1}V{(I-S)^{-1}}C_0^{-1/2}x\ra 
\nonumber
\\
&\doteq \lim_{N \approach \infty}\la C_0^{-1/2}P_Nx, (I-S)^{-1}V{(I-S)^{-1}}C_0^{-1/2}P_Nx\ra.
\end{align}
\end{proposition}

\begin{proof}
	[\textbf{of Proposition \ref{proposition:derivative-log-Radon-Nikodym-S-trace-class}}]
For $S \in \SymTr(\Hcal)_{< I}$, by Eq.\eqref{equation:log-Radon-Nikodym-S-trace-class-zero-mean},
	$\log\left\{{\frac{d\mu}{d\mu_0}(x)}\right\} = - \frac{1}{2}\log\det(I-S) - \frac{1}{2}\la C_0^{-1/2}x, S(I-S)^{-1}C_0^{-1/2}x\ra$,
where $\la C_0^{-1/2}x, S(I-S)^{-1}C_0^{-1/2}x\ra$ $\doteq \lim\limits_{N \approach \infty}
		\la C_0^{-1/2}P_Nx, S(I-S)^{-1}C_0^{-1/2}P_Nx\ra$, 	with the limit being taken in $\Lcal^2(\Hcal,\mu_0)$.
	By Lemma \ref{lemma:derivative-logdet-Fredholm}, for $h:\SymTr(\Hcal)_{<I} \mapto \R$ defined by $h(S) = \log\det(I-S)$, 
	$Dh(S)(V) = -\trace[(I-S)^{-1}V]$ $\forall V \in \Sym(\Hcal) \cap \Tr(\Hcal)$.
	This gives the first term in Eq.\eqref{equation:derivative-log-Radon-Nikodym-S-trace-class}.
	
		Let $\Omega = \{S \in \Lcal(\Hcal), I-S \text{ is invertible}\}$.
By the product rule, for $f: \Omega \mapto \Lcal(\Hcal)$ defined by $f(S) = S(I-S)^{-1}$,
$Df(S)(V) = V(I-S)^{-1} + S(I-S)^{-1}V(I-S)^{-1} = (I-S)^{-1}V(I-S)^{-1}$ $\forall V \in \Lcal(\Hcal)$.

For each fixed $N \in \Nbb$, define $g_N: \SymTr(\Hcal)_{<I} \mapto \Lcal^2(\Hcal, \mu_0)$ by
$g_N(S)(x) = \la C_0^{-1/2}P_Nx, S(I-S)^{-1}C_0^{-1/2}P_Nx\ra$.
By Lemma \ref{lemma:gaussian-integral-double-quadratic-form},
\begin{align*}
&\int_{\Hcal}|g_N(S)(x)|^2d\mu_0(x) = \int_{\Hcal}\la C_0^{-1/2}P_Nx, S(I-S)^{-1}C_0^{-1/2}P_Nx\ra^2 d\Ncal(0,C_0)(x)
\\
& = \int_{\Hcal}\la x, C_0^{-1/2}P_NS(I-S)^{-1}C_0^{-1/2}P_Nx\ra^2 d\Ncal(0,C_0)(x)
\\
&=[\trace(C_0A_N)]^2 + 2\trace[(C_0A_N)^2],
\end{align*}
where $A_N = C_0^{-1/2}P_NS(I-S)^{-1}C_0^{-1/2}P_N$. Since $C_0^{-1/2}P_NC_0^{1/2} = P_N$,
\begin{align*}
&\trace(C_0A_N) = \trace[C_0^{1/2}P_NS(I-S)^{-1}C_0^{-1/2}P_N]
\\
& = \trace[P_NS(I-S)^{-1}C_0^{-1/2}P_NC_0^{1/2}] = \trace[P_NS(I-S)^{-1}P_N],
\\
&\trace[(C_0A)^2] = \trace[(C_0^{1/2}P_NS(I-S)^{-1}C_0^{-1/2}P_N)(C_0^{1/2}P_NS(I-S)^{-1}C_0^{-1/2}P_N)]
\\
& = \trace[P_NS(I-S)^{-1}P_NS(I-S)^{-1}P_N].
\end{align*}
Thus $g_N(S) \in \Lcal^2(\Hcal,\mu_0)$ 
$\forall S \in \SymTr(\Hcal)_{<I}$.

Let us show that for a fixed $S \in \SymTr(\Hcal)_{< I}$, $Dg_N(S):\Sym(\Hcal) \cap \Tr(\Hcal) \mapto \Lcal^2(\Hcal,\mu_0)$ is given by $Dg_N(S)(V)=h_N(S)(V)$, $V \in \Sym(\Hcal) \cap \Tr(\Hcal)$, where $h_N(S)(V)(x) = \la C_0^{-1/2}P_Nx, (I-S)^{-1}V{(I-S)^{-1}}C_0^{-1/2}P_Nx\ra$.
For each fixed pair $(S,V)$, by Lemma \ref{lemma:gaussian-integral-double-quadratic-form}, similar to $g_N$ before,
\begin{align*}
&\int_{\Hcal}\la C_0^{-1/2}P_Nx, (I-S)^{-1}V{(I-S)^{-1}}C_0^{-1/2}P_Nx\ra^2 d\Ncal(0,C_0)(x)
\\
&= [\trace(P_N(I-S)^{-1}V(I-S)^{-1}P_N)]^2 
\\
&\quad + 2\trace[P_N(I-S)^{-1}V(I-S)^{-1}P_N (I-S)^{-1}V(I-S)^{-1}P_N] < \infty.
\end{align*}
Thus $h_N(S)(V) \in \Lcal^2(\Hcal,\mu_0)$.
It follows that
\begin{align*}
&||g_N(S+tV) - g_N(S) - th_N(S)(V)||^2_{\Lcal^2(\Hcal,\mu_0)}
\\
&\int_{\Hcal}|g_N(S+tV)(x) - g_N(S)(x) - th_N(S)(V)(x)|^2d\Ncal(0,C_0)(x)
\\
&=\int_{\Hcal} |\la C_0^{-1/2}P_Nx, [f(S+tV) - f(S) - tDf(S)(V)]C_0^{-1/2}P_Nx\ra|^2d\Ncal(0,C_0)(x)
\\
& \leq ||f(S+tV) - f(S) - tDf(S)(V)||^2\int_{\Hcal}||C_0^{-1/2}P_Nx||^4d\Ncal(0,C_0)(x)
\\
& = ||f(S+tV) - f(S) - tDf(S)(V)||^2
\\
&\quad \times (\trace[C_0(C_0^{-1/2}P_N)^2] \trace[C_0(C_0^{-1/2}P_N)^2] + 2\trace[(C_0(C_0^{-1/2}P_N)^2)^2])
\\
& = ||f(S+tV) - f(S) - tDf(S)(V)||^2[(\trace(P_N))^2 + 2\trace(P_N)]
\\
& = (N^2+2N) ||f(S+tV) - f(S) - tDf(S)(V)||^2.
\end{align*}
It follows that for each fixed $N \in \Nbb$, since $||V||_{\HS} \geq ||V||$,
\begin{align*}
&\lim_{t \approach 0}\frac{||g_N(S+tV) - g_N(S) - th_N(S)(V)||_{\Lcal^2(\Hcal,\mu_0)}}{|t|\;||V||_{\HS}} 
\\
&\leq \lim_{t \approach 0}\sqrt{N^2+2N}\frac{||f(S+tV) - f(S) - tDf(S)(V)||}{|t|\;||V||} = 0
\end{align*}
Thus $Dg_N(S)(V) = h_N(S)(V)$ by definition of the Fr\'echet derivative.
Taking limits as $N \approach \infty$ of $g_N$ and $h_N$ in $\Lcal^2(\Hcal,\mu_0)$ gives the second term in 
Eq.\eqref{equation:derivative-log-Radon-Nikodym-S-trace-class}.
\qed
\end{proof}

We next present the explicit expression for $D\log\left\{\frac{d\mu}{d\mu_0}(x)\right\}(S)(V)$ 
in the case both $S$ and $V$ are Hilbert-Schmidt, with the additional assumption that $V$ commutes with $S$.
In this case $S$ and $V$ have the same eigenspaces.

\begin{proposition}
	Let $S \in \SymHS(\Hcal)_{<I}$ be fixed, with eigenvalues $\{\alpha_k\}_{k \in \Nbb}$ corresponding to orthonormal eigenvectors
	$\{\phi_k\}_{k \in \Nbb}$.
	Assume further that $V \in \Sym(\Hcal)\cap \HS(\Hcal)$ commutes with $S$. Let the eigenvalues of $V$ corresponding to $\{\phi_k\}_{k \in \Nbb}$ be $\{\beta_k\}_{k \in \Nbb}$. Then
	\begin{align}
		&D\log\left\{\frac{d\mu}{d\mu_0}(x)\right\}(S)(V)
		= -\frac{1}{2}\sum_{k=1}^{\infty}\frac{\beta_k}{1-\alpha_k} \left[\frac{1}{1-\alpha_k}W^2_{\phi_k} - 1\right] \in \Lcal^2(\Hcal,\mu).
	\end{align} 
	Let $V_1,V_2 \in \Sym(\Hcal) \cap \HS(\Hcal)$ be commuting with $S$. Then
	\begin{align}
		&\int_{\Hcal}D\log\left\{\frac{d\mu}{d\mu_0}(x)\right\}(S)(V_1)D\log\left\{\frac{d\mu}{d\mu_0}(x)\right\}(S)(V_2)d\mu(x)
		\nonumber
		\\
		&= \frac{1}{2}\trace[(I-S)^{-1}V_1(I-S)^{-1}V_2].
	\end{align}
	
\end{proposition}
\begin{proof}
	Let $t \in \R$ be sufficiently small so that  $I-(S\pm tV) > 0$. The eigenvalues of
	$S \pm tV$ corresponding to $\{\phi_k\}_{k \in \Nbb}$ are $\{\alpha_k \pm  t\beta_k\}_{k \in \Nbb}$
	and $\alpha_k \pm t\beta_k < 1$ $\forall k \in \Nbb$.
	Define $
	f(S) = -\frac{1}{2}\sum_{k=1}^{\infty}\left[\frac{\alpha_k}{1-\alpha_k}W^2_{\phi_k}(x) + \log(1-\alpha_k)\right]
	$.
	Since $1-\alpha_k > 0$ and $\alpha_k \pm t\beta_k < 1$ $\forall k \in \Nbb$, we have $\frac{|t\beta_k|}{1-\alpha_k} < 1$.
	By using the absolutely convergent series $\log(1-x) = - x - \frac{x^2}{2} - \frac{x^3}{3} - \cdots $ for $|x|<1$, 
	\begin{align*}
		&f(S+tV) - f(S) 
		\\
		&=
		-\frac{1}{2}\sum_{k=1}^{\infty}\left[\left(\frac{(\alpha_k + t\beta_k)}{1-(\alpha_k + t\beta_k)} - \frac{\alpha_k}{1-\alpha_k}\right)W^2_{\phi_k} + \log(1-(\alpha_k +t\beta_k)) - \log(1-\alpha_k)\right]
		\\
		& = -\frac{1}{2}\sum_{k=1}^{\infty}\left[\frac{t\beta_k}{[1-(\alpha_k + t\beta_k)](1-\alpha_k)}W^2_{\phi_k}+\log\left(1-\frac{t\beta_k}{1-\alpha_k}\right)\right]
		\\
		& = -\frac{1}{2}\sum_{k=1}^{\infty}\left[\frac{t\beta_k}{[1-(\alpha_k + t\beta_k)](1-\alpha_k)}W^2_{\phi_k} - \frac{t\beta_k}{1-\alpha_k} - \frac{t^2\beta_k^2}{2(1-\alpha_k)^2} - \frac{t^3\beta_k^3}{3(1-\alpha_k)^3} - \cdots  \right]
		\\
		& = -\frac{t}{2} \sum_{k=1}^{\infty}\left[\frac{\beta_k}{[1-(\alpha_k + t\beta_k)](1-\alpha_k)}W^2_{\phi_k} - \frac{\beta_k}{1-\alpha_k} - \frac{t\beta_k^2}{2(1-\alpha_k)^2} - \frac{t^2\beta_k^3}{3(1-\alpha_k)^3} - \cdots  \right].
	\end{align*}
	Define $g(S)(V) = -\frac{1}{2}\sum_{k=1}^{\infty}\left[\frac{\beta_k}{(1-\alpha_k)^2}W^2_{\phi_k} - \frac{\beta_k}{1-\alpha_k}\right]
	= -\frac{1}{2}\sum_{k=1}^{\infty}\frac{\beta_k}{1-\alpha_k} \left[\frac{1}{1-\alpha_k}W^2_{\phi_k} - 1\right]
	$. By Proposition 9 in \cite{Minh:2020regularizedDiv}, the functions $\left\{\frac{1}{\sqrt{2}}\left(\frac{1}{1-\alpha_k}W^2_{\phi_k}-1\right)\right\}_{k \in \Nbb}$
	are orthonormal in $\Lcal^2(\Hcal,\mu)$. Thus 
	\begin{align*}
		||g(S)(V)||^2_{\Lcal^2(\Hcal, \mu)} = \frac{1}{2}\sum_{k=1}^{\infty}\frac{\beta_k^2}{(1-\alpha_k)^2} =\frac{1}{2} ||(I-S)^{-1}V||^2_{\HS} < \infty
		.
	\end{align*}
	Therefore $g(S)(V) \in \Lcal^2(\Hcal,\mu)$ and
	\begin{align*}
		&\frac{1}{t}[f(S+tV) - f(S) - tg(S)(V)] = 
		-\frac{t}{2}\sum_{k=1}^{\infty}\frac{\beta_k^2}{[1-(\alpha_k + t\beta_k)](1-\alpha_k)^2}W^2_{\phi_k}
		\\
		& + \frac{t}{2}\sum_{k=1}^{\infty}\left[\frac{\beta_k^2}{2(1-\alpha_k)^2} + \frac{t\beta_k^3}{3(1-\alpha_k)^3} + \cdots \right].
	\end{align*}
	By the orthonormality of $\left\{\frac{1}{\sqrt{2}}\left(\frac{1}{1-\alpha_k}W^2_{\phi_k}-1\right)\right\}_{k \in \Nbb}$
	in $\Lcal^2(\Hcal,\mu)$,
	\begin{align*}
		\left\|\frac{1}{(1-\alpha_k)}W^2_{\phi_k}\right\|_{\Lcal^2(\Hcal,\mu)} \leq
		\sqrt{2}\left\|\frac{1}{\sqrt{2}}\left(\frac{1}{(1-\alpha_k)}W^2_{\phi_k}-1\right)\right\|_{\Lcal^2(\Hcal,\mu)} + 1
		= 1+\sqrt{2},
	\end{align*}
	from which it follows that
	\begin{align*}
		&\left\|\sum_{k=1}^{\infty}\frac{\beta_k^2}{[1-(\alpha_k + t\beta_k)](1-\alpha_k)^2}W^2_{\phi_k}\right\|_{\Lcal^2(\Hcal,\mu)}
		\leq (1+\sqrt{2})\sum_{k=1}^{\infty}\frac{\beta_k^2}{[1-(\alpha_k + t\beta_k)](1-\alpha_k)}.
	\end{align*}
	By the Monotone Convergence Theorem,
	\begin{align*}
		&\lim_{t \approach 0}\sum_{k=1}^{\infty}\frac{\beta_k^2}{[1-(\alpha_k + t\beta_k)](1-\alpha_k)}
		= \sum_{k=1}^{\infty}\lim_{t \approach 0}\frac{\beta_k^2}{[1-(\alpha_k + t\beta_k)](1-\alpha_k)}
		\\
		&=\sum_{k=1}^{\infty}\frac{\beta_k^2}{(1-\alpha_k)^2} = ||(I-S)^{-1}V||^2_{\HS} < \infty.
	\end{align*}
	Since $S+|V| \in \HS(\Hcal)$, we have $\alpha_k + |\beta_k| < 1 \equivalent\frac{|\beta_k|}{1-\alpha_k} < 1$ for $k$ sufficiently large. Without loss of generality, assume that $\frac{|\beta_k|}{1-\alpha_k} < 1$ $\forall k \in \Nbb$.
	For $|t| < 1$,
	\begin{align*}
		&	\sum_{k=1}^{\infty}\left|\frac{\beta_k^2}{2(1-\alpha_k)^2} + \frac{t\beta_k^3}{3(1-\alpha_k)^3} + \cdots \right|
		\leq\sum_{k=1}^{\infty} \frac{|\beta_k|^2}{2(1-\alpha_k)^2} + \frac{|\beta_k|^3}{3(1-\alpha_k)^3} + \cdots 
		\\
		& = \sum_{k=1}^{\infty}-\frac{|\beta_k|}{(1-\alpha_k)} - \log\left(1-\frac{|\beta_k|}{(1-\alpha_k)}\right)
		= - \log\dettwo\left(I - (I-S)^{-1}|V|\right) < \infty.
	\end{align*}
	Thus by the Lebesgue Dominated Convergence Theorem
	\begin{align*}
		&\lim_{t \approach 0}\sum_{k=1}^{\infty}\left[\frac{\beta_k^2}{2(1-\alpha_k)^2} + \frac{t\beta_k^3}{3(1-\alpha_k)^3} + \cdots \right]
		\\
		&
		= \sum_{k=1}^{\infty}\lim_{t \approach 0}\left[\frac{\beta_k^2}{2(1-\alpha_k)^2} + \frac{t\beta_k^3}{3(1-\alpha_k)^3} + \cdots \right]
		\\
		&
		= \sum_{k=1}^{\infty}\frac{\beta_k^2}{2(1-\alpha_k)^2} = \frac{1}{2}||(I-S)^{-1}V||^2_{\HS} < \infty.
	\end{align*}
	Combining the previous expressions, we obtain that
	\begin{align*}
		\lim_{t \approach 0}\frac{1}{|t|}||f(S+tV) - f(S) - tg(S)(V)||_{\Lcal^2(\Hcal,\mu)} = 0.
	\end{align*}
	It thus follows by definition of the Fr\'echet derivative that
	\begin{align*}
		&Df(S)(V) = g(S)(V)
		\\
		& = -\frac{1}{2}\sum_{k=1}^{\infty}\left[\frac{\beta_k}{(1-\alpha_k)^2}W^2_{\phi_k} - \frac{\beta_k}{1-\alpha_k}\right]
		= -\frac{1}{2}\sum_{k=1}^{\infty}\frac{\beta_k}{1-\alpha_k} \left[\frac{1}{1-\alpha_k}W^2_{\phi_k} - 1\right].
	\end{align*}
	For $V_1, V_2 \in \Sym(\Hcal) \cap \HS(\Hcal)$, with eigenvalues $\{\beta_k\}_{k \in \Nbb}$, $\{\gamma_k\}_{k \in \Nbb}$, respectively, corresponding to 
	the eigenvectors
	$\{\phi_k\}_{k \in \Nbb}$, by the orthonormality of $\left\{\frac{1}{\sqrt{2}}\left(\frac{1}{1-\alpha_k}W^2_{\phi_k}-1\right)\right\}_{k \in \Nbb}$
	in $\Lcal^2(\Hcal,\mu)$,
	\begin{align*}
		&\int_{\Hcal}Df(S)(V_1)Df(S)(V_2)d\mu(x) 
		\\
		&=\frac{1}{4}\int_{\Hcal}\left(\sum_{k=1}^{\infty}\frac{\beta_k}{1-\alpha_k} \left[\frac{1}{1-\alpha_k}W^2_{\phi_k} - 1\right]\right)
		\left(\sum_{k=1}^{\infty}\frac{\gamma_k}{1-\alpha_k} \left[\frac{1}{1-\alpha_k}W^2_{\phi_k} - 1\right]\right)d\mu(x)
		\\
		&= \frac{1}{2}\sum_{k=1}^{\infty}\frac{\beta_k\gamma_k}{(1-\alpha_k)^2} = \frac{1}{2}\trace[(I-S)^{-2}V_1V_2] = \frac{1}{2}\trace[(I-S)^{-1}V_1(I-S)^{-1}V_2],
	\end{align*}
	where the last equality follows since $S,V_1,V_2$ all commute.
	%
	\qed
\end{proof}

\begin{lemma}
\label{lemma:relation-T-S-2}
Let $A,T,S \in \Lcal(\Hcal)$ be as in Corollary \ref{corollary:relation-T-S}, with $I-T = AA^{*}$ and $(I-S)^{-1} = A^{*}A$.
Then $||S||_{\HS}  = ||T(I-T)^{-1}||_{\HS}$, $||T||_{\HS} = ||S(I-S)^{-1}||_{\HS}$,
$\trace[S^2(I-S)^{-1}] = ||S(I-S)^{-1/2}||^2_{\HS} = ||T(I-T)^{-1/2}||^2_{\HS} = \trace[T^2(I-T)^{-1}]$.
\end{lemma}
\begin{proof}
With $(I-S)^{-1} = A^{*}A$, we  have $I-S = (A^{*}A)^{-1}$, $S = I - (A^{*}A)^{-1}$,
$||S||_{\HS} = ||I - (A^{*}A)^{-1}||_{\HS} = ||I- (AA^{*})^{-1}||_{\HS} = ||I-(I-T)^{-1}||_{\HS} = ||T(I-T)^{-1}||_{\HS}$.
By symmetry, $||T||_{\HS} = ||S(I-S)^{-1}||_{\HS}$.

Similarly, $\trace[S^2(I-S)^{-1}] = \trace[(I-(A^{*}A)^{-1})^2A^{*}A] = \trace[(I-(AA^{*})^{-1})^2AA^{*}]$ 
$=\trace[(I-(I-T)^{-1})^2(I-T)]= \trace[T^2(I-T)^{-1}] = ||T(I-T)^{-1/2}||^2_{\HS}$.
\qed
\end{proof}

\begin{lemma}
	\label{lemma:HS-inner-expansion}
	Let $S,T \in \Sym(\Hcal) \cap \HS(\Hcal)$, with eigenvalues $\{\alpha_k\}_{k \in \Nbb}$, $\{\beta_k\}_{k \in \Nbb}$
	and corresponding orthonormal eigenvectors $\{\phi_k\}_{k \in \Nbb}$, $\{\psi_k\}_{k \in \Nbb}$, respectively.
	Let $A \in \Sym(\Hcal)$.
	Then
	\begin{align}
		\trace(ASAT) =\trace(SATA) = \sum_{k,j=1}^{\infty}\alpha_k \beta_j \la A\phi_k, \psi_j\ra^2.
	\end{align}
If, in addition, $A \in \Sym(\Hcal)\cap \HS(\Hcal)$, then $\trace(AS) = \sum_{k=1}^{\infty}\alpha_k\la \phi_k, A\phi_k\ra$.
\end{lemma}
\begin{proof}
	Since $\{\phi_k\}_{k \in \Nbb}$ and $\{\psi_j\}_{j \in \Nbb}$ are both orthonormal bases in $\Hcal$,
	\begin{align*}
		&
		\trace(ASAT) = \trace(SATA)
		\\
		&= \sum_{k=1}^{\infty}\la \phi_k, SATA\phi_k\ra = \sum_{k=1}^{\infty}\alpha_k \la \phi_k, ATA\phi_k\ra 
		= \sum_{k=1}^{\infty}
		\alpha_k\la A\phi_k, TA\phi_k\ra 
		\\
		&= \sum_{k=1}^{\infty}\alpha_k \sum_{j=1}^{\infty}\la A\phi_k, \psi_j\ra \la TA\phi_k, \psi_j\ra =  \sum_{k=1}^{\infty}\alpha_k \sum_{j=1}^{\infty}\la A\phi_k, \psi_j\ra \la A\phi_k, T\psi_j\ra
		\\
		& = \sum_{k,j=1}^{\infty}\alpha_k\beta_j\la A\phi_k, \psi_j\ra^2.
	\end{align*}
	If we assume in addition that $A \in \HS(\Hcal)$, then
		$\trace(AS) = \sum_{k=1}^{\infty}\la \phi_k, AS\phi_k\ra = \sum_{k=1}^{\infty}\alpha_k \la \phi_k, A\phi_k\ra$.
		\qed
	\end{proof}
\begin{lemma}
	\label{lemma:logdet-2-inverse}
	Let $A \in \SymHS(\Hcal)_{< I}$ with eigenvalues $\{\lambda_k\}_{k=1}^{\infty}$. Then
	\begin{align}
		\log\dettwo[(I-A)^{-1}] &= -\sum_{k=1}^{\infty}\left[\frac{\alpha_k}{1-\alpha_k} + \log(1-\alpha_k)\right]
		\\
		&= - [\log\dettwo(I-A) + \trace(A^2(I-A)^{-1})].
	\end{align}
	In particular, for $A \in \SymTr(\Hcal)_{< I}$, $\log\dettwo[(I-A)^{-1}] = -\log\det(I-A) - \trace[A(I-A)^{-1}]$.
\end{lemma}
\begin{proof}
	Since $(I-A)^{-1} = I+ A(I-A)^{-1}$, by definition $\dettwo[(I-A)^{-1}] = \det[(I-A)^{-1}\exp(-A(I-A)^{-1})] 
	= \prod_{k=1}^{\infty}(1-\alpha_k)^{-1}\exp(-\alpha_k(1-\alpha_k)^{-1})$, from which the first identity follows.
The second identity follows from	
$\log\dettwo(I-A) + \trace[A^2(I-A)^{-1}] = \alpha_k + \log(1-\alpha_k) + \frac{\alpha_k^2}{1-\alpha_k}
= \frac{\alpha_k}{1-\alpha_k} + \log(1-\alpha_k)$.
	\qed
\end{proof}

\begin{proposition}
	\label{proposition:inner-logRN-mu-star}
	Let $S_{*} \in \SymHS(\Hcal)_{<I}$ be fixed but arbitrary. 
	Let $\mu_{*} = \Ncal(0, C_{*})$, where $C_{*} = C_0^{1/2}(I-S_{*})C_0^{1/2}$.
	Let $S,T \in \SymHS(\Hcal)_{< I}$. Let
	$\mu = \Ncal(0,Q)$, $Q = C_0^{1/2}(I-S)C_0^{1/2}$, $\nu = \Ncal(0,R)$, $R = C_0^{1/2}(I-T)C_0^{1/2}$.
	Then 
	\begin{align}
		&\left\la \log\left\{{\frac{d\mu}{d\mu_0}(x)}\right\}, \log\left\{{\frac{d\nu}{d\mu_0}(x)}\right\}\right\ra_{\Lcal^2(\Hcal,\mu_{*})}
		\nonumber
		\\
		& = \frac{1}{2}\trace[(I-S_{*})S(I-S)^{-1}(I-S_{*})T(I-T)^{-1}]
		\nonumber
		\\
		& \quad + \frac{1}{4}(\log\dettwo[(I-S)^{-1}] + \trace[S_{*}S(I-S)^{-1}])
		\nonumber
		\\
		& \quad \quad \times (\log\dettwo[(I-T)^{-1}] + \trace[S_{*}T(I-T)^{-1}]).
		\\
		&\left\| \log\left\{{\frac{d\mu}{d\mu_0}(x)}\right\}\right\|^2_{\Lcal^2(\Hcal,\mu_{*})}=\frac{1}{2} ||(I-S_{*})^{1/2}S(I-S)^{-1}(I-S_{*})^{1/2}||^2_{\HS}
		\nonumber
		\\
		&\quad \quad \quad \quad \quad \quad +\frac{1}{4}(\log\dettwo[(I-S)^{-1}] + \trace[S_{*}S(I-S)^{-1}])^2.
		\\
		&\left\| \log\left\{{\frac{d\mu}{d\mu_0}(x)}\right\}- \log\left\{{\frac{d\nu}{d\mu_0}(x)}\right\}\right\|^2_{\Lcal^2(\Hcal,\mu_{*})}
		\label{equation:norm-difference-logRN-mu-star}
		\\
		&= \frac{1}{2}||(I-S_{*})^{1/2}[S(I-S)^{-1}-T(I-T)^{-1}](I-S_{*})^{1/2}||^2_{\HS}
		\nonumber
		\\
		&  + \frac{1}{4}(\log\dettwo(I-S)^{-1}- \log\dettwo(I-T)^{-1} + \trace[S_{*}[S(I-S)^{-1}-T(I-T)^{-1}]])^2.
		\nonumber
	\end{align}
	Furthermore,
	\begin{align}
		\label{equation:norm-difference-bound-logRN-mu-star}
		&\left\| \log\left\{{\frac{d\mu}{d\mu_0}(x)}\right\}- \log\left\{{\frac{d\nu}{d\mu_0}(x)}\right\}\right\|^2_{\Lcal^2(\Hcal,\mu_{*})}
		\nonumber
		\\
		&\leq \frac{1}{2}\left(||I-S_{*}||^2 + ||S_{*}||_{\HS}^2+ [||S||_{\HS} + ||T||_{\HS} + ||ST||_{\HS}]^2\right)
		\nonumber
		\\
		& \quad \times ||(I-S)^{-1}||^2||(I-T)^{-1}||^2||S-T||^2_{\HS}. 
	\end{align}
\end{proposition}
\begin{proof}[\textbf{of Proposition \ref{proposition:inner-logRN-mu-star}}]
	Let $\{\alpha_k\}_{k \in \Nbb}$ and $\{\beta_k\}_{k \in \Nbb}$ be the eigenvalues of $S$ and $T$, with corresponding orthonormal eigenvectors
	$\{\phi_k\}_{k \in \Nbb}$ and $\{\psi_k\}_{k \in \Nbb}$, respectively.
	Consider the white noise mapping $W: \Hcal \mapto \Lcal^2(\Hcal, \mu_0)$.
	For the Radon-Nikodym densities $\frac{d\mu}{d\mu_0}$ and $\frac{d\nu}{d\mu_0}$,
	\begin{align*}
		\log\left\{{\frac{d\mu}{d\mu_0}(x)}\right\} = -\frac{1}{2}\sum_{k=1}^{\infty}\left[\frac{\alpha_k}{1-\alpha_k}W^2_{\phi_k}(x) + \log(1-\alpha_k)\right],
		\\
		\log\left\{{\frac{d\nu}{d\mu_0}(x)}\right\} = -\frac{1}{2}\sum_{k=1}^{\infty}\left[\frac{\beta_k}{1-\beta_k}W^2_{\psi_k}(x) + \log(1-\beta_k)\right].
	\end{align*}
By Lemma 19 in \cite{Minh:2020regularizedDiv}, for any $a \in \Hcal$,
\begin{align*}
\int_{\Hcal}W^2_{a}(x)d\mu_{*}(x) = \la a, (I-S_{*})a\ra.
\end{align*}
	By Lemma 22 in \cite{Minh:2020regularizedDiv}, for any pair $a,b \in \Hcal$,
	\begin{align*}
		\int_{\Hcal}W^2_a(x)W^2_b(x)d\mu_{*}(x) =\la a, (I-S_{*})a\ra \la b, (I-S_{*})b\ra + 2 \la a,(I-S_{*})b\ra^2.
	\end{align*}
	Thus for each pair $k, j \in \Nbb$,
	\begin{align*}
		&\int_{\Hcal}\left[\frac{\alpha_k}{1-\alpha_k}W^2_{\phi_k}(x) + \log(1-\alpha_k)\right]\left[\frac{\beta_j}{1-\beta_j}W^2_{\psi_j}(x) + \log(1-\beta_j)\right]d\mu_{*}(x)
		\\
		& = \frac{\alpha_k}{1-\alpha_k}\frac{\beta_j}{1-\beta_j}[\la \phi_k, (I-S_{*})\phi_k\ra \la \psi_j, (I-S_{*})\psi_j\ra + 2 \la \phi_k,(I-S_{*})\psi_j\ra^2] 
		\\
		&+ \frac{\alpha_k}{1-\alpha_k}\log(1-\beta_j)\la \phi_k, (I-S_{*})\phi_k\ra
		+ \frac{\beta_j}{1-\beta_j}\log(1-\alpha_k) \la \psi_j, (I-S_{*})\psi_j\ra
		\\
		&\quad + \log(1-\alpha_k)\log(1-\beta_j)
		\\
		& = 2\frac{\alpha_k}{1-\alpha_k}\frac{\beta_j}{1-\beta_j}[\la \phi_k, (I-S_{*})\psi_j\ra^2] 
		\\
		&\quad + \left[\frac{\alpha_k}{1-\alpha_k}\la \phi_k, (I-S_{*})\phi_k\ra + \log(1-\alpha_k)\right]
		\left[\frac{\beta_j}{1-\beta_j}\la \psi_j, (I-S_{*})\psi_j\ra + \log(1-\beta_j)\right]
		\\
		& = 2\frac{\alpha_k}{1-\alpha_k}\frac{\beta_j}{1-\beta_j}[\la \phi_k, (I-S_{*})\psi_j\ra^2] 
		\\
		&\quad + \left[- \frac{\alpha_k}{1-\alpha_k}\la \phi_k, S_{*}\phi_k\ra + \frac{\alpha_k}{1-\alpha_k} + \log(1-\alpha_k)\right]
		\\
		& \quad \quad \times \left[- \frac{\beta_j}{1-\beta_j}\la \psi_j, S_{*}\psi_j\ra + \frac{\beta_j}{1-\beta_j} + \log(1-\beta_j)\right].
	\end{align*}
For each fixed pair $N,M \in \Nbb$, define
\begin{align*}
f = \frac{1}{2}\sum_{k=1}^{\infty}\left[\frac{\alpha_k}{1-\alpha_k}W^2_{\phi_k}(x) + \log(1-\alpha_k)\right],\;
f_N = \frac{1}{2}\sum_{k=1}^{N}\left[\frac{\alpha_k}{1-\alpha_k}W^2_{\phi_k}(x) + \log(1-\alpha_k)\right],
\\
g = \frac{1}{2}\sum_{k=1}^{\infty}\left[\frac{\beta_k}{1-\beta_k}W^2_{\psi_k}(x) + \log(1-\beta_k)\right],\;
g_M = \frac{1}{2}\sum_{k=1}^{M}\left[\frac{\beta_k}{1-\beta_k}W^2_{\psi_k}(x) + \log(1-\beta_k)\right].
\end{align*}
By Proposition 10 in \cite{Minh:2020regularizedDiv}, $f,g \in \Lcal^2(\Hcal,\mu_{*})$, with $\lim\limits_{N \approach \infty}||f_N - f||_{\Lcal^2(\Hcal,\mu_{*})} = 0$, $\lim\limits_{N \approach \infty}||g_M - g||_{\Lcal^2(\Hcal,\mu_{*})} = 0$. By H\"older's inequality,
$\lim\limits_{N,M \approach \infty}||f_Ng_M - fg||_{\Lcal^1(\Hcal,\mu_{*})} = 0$.
	It follows that
	\begin{align*}
		&\left\la \log\left\{{\frac{d\mu}{d\mu_0}(x)}\right\}, \log\left\{{\frac{d\nu}{d\mu_0}(x)}\right\}\right\ra_{\Lcal^2(\Hcal,\mu_{*})}
		= \int_{\Hcal}f(x)g(x)\mu_{*}(dx) 
		\\
		&= \lim_{N, M \approach \infty}\int_{\Hcal}f_N(x)g_M(x)\mu_{*}(dx)
		= \frac{1}{2}\sum_{k,j=1}^{\infty}\frac{\alpha_k}{1-\alpha_k}\frac{\beta_j}{1-\beta_j}\la \phi_k, (I-S_{*})\psi_j\ra^2
		\\
		&\quad + \frac{1}{4}\sum_{k=1}^{\infty}\left[- \frac{\alpha_k}{1-\alpha_k}\la \phi_k, S_{*}\phi_k\ra + \frac{\alpha_k}{1-\alpha_k} + \log(1-\alpha_k)\right]
		\\
		&\quad \quad \times \sum_{j=1}^{\infty}\left[- \frac{\beta_j}{1-\beta_j}\la \psi_j, S_{*}\psi_j\ra +\frac{\beta_j}{1-\beta_j} + \log(1-\beta_j)\right]
		\\
		& = \frac{1}{2}\trace[(I-S_{*})S(I-S)^{-1}(I-S_{*})T(I-T)^{-1}]
		\\
		 & \quad + \frac{1}{4}(\log\dettwo[(I-S)^{-1}] + \trace[S^{*}S(I-S)^{-1}])
		 \\
		 & \quad \quad \times (\log\dettwo[(I-T)^{-1}] + \trace[S^{*}T(I-T)^{-1}]).
	\end{align*}
	where the last line follows from Lemmas \ref{lemma:HS-inner-expansion} and \ref{lemma:logdet-2-inverse}. Thus
	\begin{align*}
		&\left\| \log\left\{{\frac{d\mu}{d\mu_0}(x)}\right\}- \log\left\{{\frac{d\nu}{d\mu_0}(x)}\right\}\right\|^2_{\Lcal^2(\Hcal,\mu_{*})}
		\\
		&= \frac{1}{2}||(I-S_{*})^{1/2}[S(I-S)^{-1}-T(I-T)^{-1}](I-S_{*})^{1/2}||^2_{\HS}
		\\
		&  + \frac{1}{4}(\log\dettwo(I-S)^{-1}- \log\dettwo(I-T)^{-1} + \trace[S^{*}[S(I-S)^{-1}-T(I-T)^{-1}]])^2
		\\
		& \leq \frac{1}{2}||(I-S_{*})^{1/2}[S(I-S)^{-1}-T(I-T)^{-1}](I-S_{*})^{1/2}||^2_{\HS}
		\\
		&\quad + \frac{1}{2}(\log\dettwo(I-S)^{-1}- \log\dettwo(I-T)^{-1})^2 
		+\frac{1}{2}(\trace[S^{*}[S(I-S)^{-1}-T(I-T)^{-1}]])^2.
	\end{align*}
	Using $S(I-S)^{-1} - T(I-T)^{-1} = (I-S)^{-1}(S-T)(I-T)^{-1}$, we obtain
	\begin{align*}
		&||(I-S_{*})^{1/2}[S(I-S)^{-1} - T(I-T)^{-1}](I-S_{*})^{1/2}||_{\HS}
		\\ 
		&\leq ||I-S_{*}||\;||(I-S)^{-1}||\;||(I-T)^{-1}||\;||S-T||_{\HS},
		\\
		&|\trace[S_{*}[S(I-S)^{-1}-T(I-T)^{-1}]| = |\la S_{*}, S(I-S)^{-1}-T(I-T)^{-1}\ra_{\HS}|
		\\
		& \leq ||S_{*}||_{\HS}||(I-S)^{-1}||\;||(I-T)^{-1}||\;||S-T||_{\HS}.
	\end{align*}
	By Theorem 12 in \cite{Minh:2022KullbackGaussian},
	\begin{align*}
		&|\log\dettwo[(I-S)^{-1}] - \log\dettwo[(I-T)^{-1}]| 
		\\
		&\leq ||(I-S)^{-1}||\;||(I-T)^{-1}||[||S||_{\HS} + ||T||_{\HS} + ||ST||_{\HS}]||S-T||_{\HS}.
	\end{align*}
	Combining the last four expressions, we obtain
	\begin{align*}
		&\left\| \log\left\{{\frac{d\mu}{d\mu_0}(x)}\right\}- \log\left\{{\frac{d\nu}{d\mu_0}(x)}\right\}\right\|^2_{\Lcal^2(\Hcal,\mu_{*})}
		\\
		&\leq \frac{1}{2}\left(||I-S_{*}||^2 + ||S_{*}||_{\HS}^2+ [||S||_{\HS} + ||T||_{\HS} + ||ST||_{\HS}]^2\right)
		\\
		& \quad \times ||(I-S)^{-1}||^2||(I-T)^{-1}||^2||S-T||^2_{\HS}.
	\end{align*}
	\qed
\end{proof}

\begin{corollary}
	\label{corollary:inner-logRN-mu0}
	Let $S,T \in \SymHS(\Hcal)_{< I}$. Let
	$\mu = \Ncal(0,Q)$, $Q = C_0^{1/2}(I-S)C_0^{1/2}$, $\nu = \Ncal(0,R)$, $R = C_0^{1/2}(I-T)C_0^{1/2}$.
	Then 
	\begin{align}
		\label{equation:norm-logRN-mu0}
		&\left\la \log\left\{{\frac{d\mu}{d\mu_0}(x)}\right\}, \log\left\{{\frac{d\nu}{d\mu_0}(x)}\right\}\right\ra_{\Lcal^2(\Hcal,\mu_0)} =  \frac{1}{2}\la S(I-S)^{-1}, T(I-T)^{-1}\ra_{\HS}
		\nonumber
		\\
		& \quad \quad \quad \quad \quad \quad + \frac{1}{4}\log\dettwo[(I-S)^{-1}]\log\dettwo[(I-T)^{-1}].
		\\
		&\left\| \log\left\{{\frac{d\mu}{d\mu_0}(x)}\right\}\right\|^2_{\Lcal^2(\Hcal,\mu_0)}=
		\frac{1}{2}||S(I-S)^{-1}||^2_{\HS}
		+ \frac{1}{4}(\log\dettwo[(I-S)^{-1}])^2.
		\\
		&\left\| \log\left\{{\frac{d\mu}{d\mu_0}(x)}\right\}- \log\left\{{\frac{d\nu}{d\mu_0}(x)}\right\}\right\|^2_{\Lcal^2(\Hcal,\mu_0)}=
		\frac{1}{2}||S(I-S)^{-1}-T(I-T)^{-1}||^2_{\HS}
		\nonumber
		\\
		& \quad \quad \quad \quad \quad \quad + \frac{1}{4}(\log\dettwo[(I-S)^{-1}] - \log\dettwo[(I-T)^{-1}])^2.
		\label{equation:norm-difference-logRN-mu-0}
	\end{align}
	Furthermore,
	\begin{align}
		\label{equation:norm-difference-bound-logRN-mu-0}
		&\left\| \log\left\{{\frac{d\mu}{d\mu_0}(x)}\right\}- \log\left\{{\frac{d\nu}{d\mu_0}(x)}\right\}\right\|^2_{\Lcal^2(\Hcal,\mu_0)}
		\\
		&\leq \left(\frac{1}{2} + \frac{1}{4}[||S||_{\HS} + ||T||_{\HS} + ||ST||_{\HS}]^2\right)||(I-S)^{-1}||^2||(I-T)^{-1}||^2||S-T||^2_{\HS}.
		\nonumber
	\end{align}
\end{corollary}
\begin{proof}
	This follows from Proposition \ref{proposition:inner-logRN-mu-star} by letting $S_{*} = 0$. Here,
		the inequality in \eqref{equation:norm-difference-bound-logRN-mu-0} has a smaller constant compared to that in \eqref{equation:norm-difference-bound-logRN-mu-star}. This is obtained 
	by directly bounding the expression in Eq.\eqref{equation:norm-difference-logRN-mu-0}, using the same technique 
	as in the proof of Proposition \ref{proposition:inner-logRN-mu-star}.
	\qed
\end{proof}

\begin{corollary}
\label{corollary:norm-logRN-mu}
Let $S \in \SymHS(\Hcal)_{< I}$ be fixed.
Let $\mu = \Ncal(0,C)$, $C = C_0^{1/2}(I-S)C_0^{1/2}$. Then
\begin{align}
&\left\| \log\left\{{\frac{d\mu}{d\mu_0}(x)}\right\}\right\|^2_{\Lcal^2(\Hcal,\mu)}=\frac{1}{2} ||S||^2_{\HS}
 +\frac{1}{4}[\log\dettwo(I-S)]^2.
%
\end{align}
\end{corollary}
\begin{proof}	
	This is obtained by letting $S_{*}= S$ in Proposition \ref{proposition:inner-logRN-mu-star}
	and noting that $-\log\dettwo(I-S) = \log\dettwo[(I-S)^{-1}] + \trace[S^2(I-S)^{-1}]$ by Lemma
	\ref{lemma:logdet-2-inverse}.
	
	Alternatively, by symmetry, let $T \in \SymHS(\Hcal)_{< I}$ be such that $C_0 = C^{1/2}(I-T)C^{1/2}$.
	Then by Eq.\eqref{equation:norm-logRN-mu0} in Corollary \ref{corollary:inner-logRN-mu0},
	\begin{align*}
&\left\| \log\left\{{\frac{d\mu}{d\mu_0}(x)}\right\}\right\|^2_{\Lcal^2(\Hcal,\mu)} = \left\| \log\left\{{\frac{d\mu_0}{d\mu}(x)}\right\}\right\|^2_{\Lcal^2(\Hcal,\mu)}
\\
&	= \frac{1}{2}||T(I-T)^{-1}||^2_{\HS} + \frac{1}{4}(\log\dettwo[(I-T)^{-1}])^2 
= \frac{1}{2}||S||^2_{\HS} + \frac{1}{4}[\log\dettwo(I-S)]^2,
	\end{align*}
where the last equality follows from Corollary \ref{corollary:relation-T-S} and Lemma \ref{lemma:relation-T-S-2}.
			\qed
\end{proof}

\begin{corollary}
	\label{corollary:logRN-convergence-HS-norm}
	Let $\mu = \mu(S) = \Ncal(0,C)$, where $C = C_0^{1/2}(I-S)C_0^{1/2}$, $S \in \SymHS(\Hcal)_{<I}$.
	Let $S_{*} \in \SymHS(\Hcal)_{<I}$ be fixed but arbitrary. 
	Let $\mu_{*} = \Ncal(0,C_{*})$, $C_{*} = C_0^{1/2}(I-S_{*})C_0^{1/2}$, be the corresponding Gaussian measure.
	Let $f: \SymHS(\Hcal)_{<I} \mapto \Lcal^2(\Hcal,\mu_{*})$ be defined by
	$f(S)(x) = \log\left\{\frac{d\mu}{d\mu_0}(x)\right\}$.
	Let $T \in \SymHS(\Hcal)_{< I}$ be fixed.
	Let $\{T_k\}_{k \in \Nbb}$ in $\SymHS(\Hcal)_{<I}$ be 
	such that $\lim\limits_{k \approach \infty}||T_k - T||_{\HS} = 0$. Then
	\begin{align}
		&||f(T_k) - f(T)||^2_{\Lcal^2(\Hcal,\mu_{*})} 
		\\
		&\leq \frac{1}{2}\left(||I-S_{*}||^2 + ||S_{*}||_{\HS}^2+ [||T_k||_{\HS} + ||T||_{\HS} + ||T_kT||_{\HS}]^2\right)
		\nonumber
		\\
		& \quad \times ||(I-T_k)^{-1}||^2||(I-T)^{-1}||^2||T_k-T||^2_{\HS}.
	\end{align}
In particular, $\lim\limits_{k \approach \infty}||f(T_k) - f(T)||_{\Lcal^2(\Hcal,\mu_{*})} = 0$.
\end{corollary}
\begin{proof}
We first note that given any fixed $T \in \SymHS(\Hcal)_{<I}$, there always exists a sequence  $\{T_k\}_{k \in \Nbb}$ in $\SymHS(\Hcal)_{<I}$
such that $\lim\limits_{k \approach \infty}||T_k - T||_{\HS} = 0$. We can first take, for example, $\{T_k\}_{k \in \Nbb}$ to be a sequence of finite-rank self-adjoint operators that approximate $T$ in the $||\;||_{\HS}$ norm.

Since $I-T > 0$, there exists $M_T > 0$ such that $\la x, (I-T)x\ra \geq M_T||x||^2$ $\forall x \in \Hcal$.
Let $0 < \ep < M_T$ be fixed. Since $\lim\limits_{k \approach \infty}||T_k-T||_{\HS} = 0$,
$\exists N(\ep) \in \Nbb$ such that $||T_k - T||\leq ||T_k - T||_{\HS} < \ep$ $\forall k \geq N(\ep)$.
Thus 
$|\la x, (T_k-T)x\ra| \leq ||T_k-T||\;||x||^2 \leq \epsilon ||x||^2$ $\forall x \in \Hcal$.
Then $\la x, (I-T_k)x\ra = \la x, (I-T)x\ra - \la x, (T_k - T)x\ra \geq (M_T-\ep)||x||^2$ $\forall k \geq N(\ep)$,
so that $T_k \in \SymHS(\Hcal)_{<I}$ $\forall k \geq N(\ep)$.
The result then follows from Proposition \ref{proposition:inner-logRN-mu-star}.
\qed
\end{proof}

\begin{lemma}
	\label{lemma:integral-quadratic-S-star}
	Let $S_{*} \in \SymHS(\Hcal)_{<I}$ be fixed. 
	Let $C_{*} = C_0^{1/2}(I-S_{*})C_0^{1/2}$ and 
	$\mu_{*} = \Ncal(0,C_{*})$ be the corresponding Gaussian measure. For
	any $T \in \Sym(\Hcal) \cap \Tr(\Hcal)$, let $ \la C_0^{-1/2}x, TC_0^{-1/2}x\ra \doteq \lim\limits_{N \approach \infty}\la C_0^{-1/2}P_Nx, TC_0^{-1/2}P_Nx\ra$ in $\Lcal^2(\Hcal,\mu_{*})$. Then
\begin{align}
\int_{\Hcal}\la C_0^{-1/2}x, TC_0^{-1/2}x\ra d\mu_{*}(x) &= \trace[(I-S_{*})T],
\\
\int_{\Hcal}\la C_0^{-1/2}x, TC_0^{-1/2}x\ra^2 d\mu_{*}(x)
&= [\trace((I-S_{*})T)]^2 + 2\trace[((I-S_{*})T)^2].
\end{align}
For any pair $T_1, T_2 \in \Sym(\Hcal) \cap \Tr(\Hcal)$,
\begin{align}
&\int_{\Hcal}\la C_0^{-1/2}x, T_1C_0^{-1/2}x\ra \la C_0^{-1/2}x, T_2C_0^{-1/2}x\ra d\mu_{*}(x)
\nonumber
\\
& = \trace((I-S_{*})T_1)\trace((I-S_{*})T_2) + 2\trace[(I-S_{*})T_1(I-S_{*})T_2].
\end{align}
\end{lemma}
\begin{proof}
(i) 
We have
$\lim\limits_{N \approach \infty}||\la C_0^{-1/2}P_Nx, TC_0^{-1/2}P_Nx\ra - \la C_0^{-1/2}x, TC_0^{-1/2}x\ra||_{\Lcal^1(\Hcal,\mu_{*})} = 0$ by H\"older's inequality.
It follows that
\begin{align*}
&\int_{\Hcal}\la C_0^{-1/2}x, TC_0^{-1/2}x\ra d\mu_{*}(x)
\\
& = \lim_{N \approach\infty}
\int_{\Hcal}\la C_0^{-1/2}P_Nx, TC_0^{-1/2}P_Nx\ra d\Ncal(0,C_{*})(x)
\\
& = \lim_{N \approach\infty}
\int_{\Hcal}\la x,C_0^{-1/2}P_NTC_0^{-1/2}P_Nx\ra d\Ncal(0,C_{*})(x)
\\
& = \lim_{N \approach \infty}\trace[C_0^{1/2}(I-S_{*})C_0^{1/2}(C_0^{-1/2}P_NTC_0^{-1/2}P_N)]
\\
& = \lim_{N \approach \infty}\trace[(I-S_{*})P_NTP_N]
= \trace[(I-S_{*})T].
\end{align*}

(ii) Let $A_N = C_0^{-1/2}P_NTC_0^{-1/2}P_N$. By Lemma \ref{lemma:gaussian-integral-double-quadratic-form},
\begin{align*}
&||\la C_0^{-1/2}P_Nx, TC_0^{-1/2}P_Nx\ra||^2_{\Lcal^2(\Hcal,\mu_{*})}
= \int_{\Hcal}\la C_0^{-1/2}P_Nx, TC_0^{-1/2}P_Nx\ra^2 d\Ncal(0,C_{*})(x)
\\
& = \int_{\Hcal}\la x, A_Nx\ra^2 d\Ncal(0,C_{*})(x)
= [\trace(C_{*}A_N)]^2 + 2\trace[(C_{*}A_N)^2] 
\\
&= [\trace((I-S_{*})P_NTP_N)]^2 + 2\trace[(I-S_{*})P_NTP_N(I-S_{*})P_NTP_N],
\end{align*}
where we have used the identity $C_0^{-1/2}P_NC_0^{1/2} = P_N$. It follows that
\begin{align*}
&||\la C_0^{-1/2}x, TC_0^{-1/2}x\ra||^2_{\Lcal^2(\Hcal,\mu_{*})} = \lim_{N \approach \infty}||\la C_0^{-1/2}P_Nx, TC_0^{-1/2}P_Nx\ra||^2_{\Lcal^2(\Hcal,\mu_{*})}
\\
& = [\trace((I-S_{*})T)]^2 + 2\trace[((I-S_{*})T)^2].
\end{align*}
For a pair $T_1, T_2 \in \Sym(\Hcal) \cap \Tr(\Hcal)$, since
$\lim\limits_{N \approach \infty}||\la C_0^{-1/2}P_Nx, T_jC_0^{-1/2}P_Nx\ra - \la C_0^{-1/2}x, T_jC_0^{-1/2}x\ra||_{\Lcal^2(\Hcal,\mu_{*})} = 0$, $j=1,2$, by H\"older's inequality, we have in the $\Lcal^1(\Hcal,\mu_{*})$ sense,
$\lim\limits_{N \approach \infty}\la C_0^{-1/2}P_Nx, T_1C_0^{-1/2}P_Nx\ra \la C_0^{-1/2}P_Nx, T_2C_0^{-1/2}P_Nx\ra = \la C_0^{-1/2}x, T_1C_0^{-1/2}x\ra \la C_0^{-1/2}x, T_2C_0^{-1/2}x\ra$.
Let $A_N = C_0^{-1/2}P_NT_1C_0^{-1/2}P_N$, $B_N = C_0^{-1/2}P_NT_2C_0^{-1/2}P_N$. By Lemma \ref{lemma:gaussian-integral-double-quadratic-form}, we then have
\begin{align*}
&\int_{\Hcal}\la C_0^{-1/2}x, T_1C_0^{-1/2}x\ra \la C_0^{-1/2}x, T_2C_0^{-1/2}x\ra d\mu_{*}(x)
\\
&= \lim_{N\approach \infty}\int_{\Hcal}\la C_0^{-1/2}P_Nx, T_1C_0^{-1/2}P_Nx\ra \la C_0^{-1/2}P_Nx, T_2C_0^{-1/2}P_Nx\ra d\mu_{*}(x)
\\
& = \lim_{N \approach \infty}\int_{\Hcal}\la x, A_Nx\ra \la x, B_Nx\ra d\Ncal(0,C_{*})(x)
\\
& = \lim_{N \approach \infty}[\trace(C_{*}A_N)\trace(C_{*}B_N) + 2\trace(C_{*}A_NC_{*}B_N)]
\\
& = \lim_{N\approach \infty}[\trace((I-S_{*})P_NT_1P_N)\trace((I-S_{*})P_NT_2P_N)
\\
&
\quad +\lim_{N \approach \infty}2 \trace((I-S_{*})P_NT_1P_N(I-S_{*})P_NT_2P_N)
\\
& = \trace((I-S_{*})T_1)\trace((I-S_{*})T_2) + 2\trace[(I-S_{*})T_1(I-S_{*})T_2].
\end{align*}
\qed
\end{proof}

\begin{lemma}
	\label{lemma:square-difference-DRadonNikodym-derivative}
	Let $\mu = \mu(S) = \Ncal(0, C)$, where $C = C_0^{1/2}(I-S)C_0^{1/2}$, $S \in \SymHS(\Hcal)_{< I}$.
	Let $S_{*} \in \SymHS(\Hcal)_{< I}$ be fixed but arbitrary.
	Let $\mu_{*} = \Ncal(0,C_{*})$, $C_{*} = C_0^{1/2}(I-S_{*})C_0^{1/2}$, be the corresponding Gaussian measure.
	Let $f: \SymHS(\Hcal)_{< I} \mapto \Lcal^2(\Hcal,\mu_{*})$ be defined by $f(S) = \log\left\{\frac{d\mu}{d\mu_0}(x)\right\}$.
	Let $\{S_k\}_{k \in \Nbb} \in \SymTr(\Hcal)_{<I}$. Then $\forall V \in \Sym(\Hcal) \cap \Tr(\Hcal)$, $\forall k \in \Nbb$, 
	$Df(S_k)(V) \in \Lcal^2(\Hcal,\mu_{*})$ and $\forall i,j \in \Nbb$,
	\begin{align}
		&||Df(S_i)(V) - Df(S_j)(V)||^2_{\Lcal^2(\Hcal,\mu_{*})}
		\nonumber
		\\
		&\leq \left(\frac{1}{2} + ||I-S_{*}||^2||(I-S_i)^{-1}||^2||(I-S_j)^{-1}||^2\right)
		\nonumber
		\\
		&\quad \times ||(I-S_i)^{-1}||^2||(I-S_j)^{-1}||^2||S_i-S_j||^2_{\HS}||V||^2_{\HS}.
	\end{align}
	
\end{lemma}
\begin{proof}
By Proposition \ref{proposition:derivative-log-Radon-Nikodym-S-trace-class},
for $S \in \SymTr(\Hcal)_{<I}$ and $V \in \Sym(\Hcal) \cap \Tr(\Hcal)$,
\begin{align*}
Df(S)(V) = \frac{1}{2}\trace\left[(I-S)^{-1}V\right] - \frac{1}{2}\la C_0^{-1/2}x, (I-S)^{-1}V{(I-S)^{-1}}C_0^{-1/2}x\ra.
\end{align*}
Let $A = (I-S_i)^{-1}V(I-S_i)^{-1}$, $B =(I-S_j)^{-1}V(I-S_j)^{-1}$. By Lemma \ref{lemma:integral-quadratic-S-star},
	\begin{align*}
		&\left \la D\log\left\{\frac{d\mu}{d\mu_0}(x) \right\}(S_i)(V), D\log\left\{\frac{d\mu}{d\mu_0}(x) \right\}(S_j)(V)\right\ra_{\Lcal^2(\Hcal,\mu_{*})}
		\\
		& = \frac{1}{4}\trace\left[(I-S_i)^{-1}V\right]\trace\left[(I-S_j)^{-1}V\right]
		\\
		&\quad - \frac{1}{4}\trace\left[(I-S_i)^{-1}V\right]\int_{\Hcal}\la C_0^{-1/2}x, BC_0^{-1/2}x\ra d\mu_{*}(x)
		\\
		&\quad -\frac{1}{4}\trace\left[(I-S_j)^{-1}V\right]\int_{\Hcal}\la C_0^{-1/2}x, AC_0^{-1/2}x\ra d\mu_{*}(x)
\\
		&\quad +\frac{1}{4}\int_{\Hcal}\la C_0^{-1/2}x, AC_0^{-1/2}x\ra\la C_0^{-1/2}x, BC_0^{-1/2}x\ra d\mu_{*}(x)
		\\
		& = \frac{1}{4}\trace\left[(I-S_i)^{-1}V\right]\trace\left[(I-S_j)^{-1}V\right]
		\\
		&\quad - \frac{1}{4}\trace\left[(I-S_i)^{-1}V\right]\trace[(I-S_{*})(I-S_j)^{-1}V(I-S_j)^{-1}]
		\\
		& \quad - \frac{1}{4}\trace\left[(I-S_j)^{-1}V\right]\trace[(I-S_{*})(I-S_i)^{-1}V(I-S_i)^{-1}]
		\\
		& \quad + \frac{1}{4}[\trace[(I-S_{*})(I-S_i)^{-1}V(I-S_i)^{-1}]\trace[(I-S_{*})(I-S_j)^{-1}V(I-S_j)^{-1}]]
		\\
		& \quad+ \frac{1}{2}\trace[(I-S_{*})(I-S_i)^{-1}V(I-S_i)^{-1}(I-S_{*})(I-S_j)^{-1}V(I-S_j)^{-1}]
		\\
		& =\frac{1}{4}\left([\trace((I-S_i)^{-1}V)] - \trace[(I-S_{*})(I-S_i)^{-1}V(I-S_i)^{-1}]\right)
		\\
		&\quad \times \left([\trace((I-S_j)^{-1}V)] - \trace[(I-S_{*})(I-S_j)^{-1}V(I-S_j)^{-1}]\right)
		\\
		&\quad+ \frac{1}{2}\trace[(I-S_{*})(I-S_i)^{-1}V(I-S_i)^{-1}(I-S_{*})(I-S_j)^{-1}V(I-S_j)^{-1}].
	\end{align*}
	It follows that
	\begin{align*}
		&\Delta = \left\| D\log\left\{\frac{d\mu}{d\mu_0}(x) \right\}(S_i)(V)- D\log\left\{\frac{d\mu}{d\mu_0}(x)(S_j)(V)\right\}\right\|^2_{\Lcal^2(\Hcal,\mu_{*})}
		\\
		& = \frac{1}{4}([\trace((I-S_i)^{-1}V)- \trace[(I-S_{*})(I-S_i)^{-1}V(I-S_i)^{-1}]]
		\\
		& \quad  \quad  - [\trace((I-S_j)^{-1}V) - \trace[(I-S_{*})(I-S_j)^{-1}V(I-S_j)^{-1}]])^2
		\\
		& \quad + \frac{1}{2}||(I-S_{*})^{1/2}[(I-S_i)^{-1}V(I-S_i)^{-1} -(I-S_j)^{-1}V(I-S_j)^{-1}](I-S_{*})^{1/2}||^2_{\HS}
		\\
		& \leq \frac{1}{2}(\trace[(I-S_i)^{-1}V] - \trace(I-S_j)^{-1}V])^2
		\\
		& \quad + \frac{1}{2}(\trace[(I-S_{*})(I-S_i)^{-1}V(I-S_i)^{-1}]-\trace[(I-S_{*})(I-S_j)^{-1}V(I-S_j)^{-1}])^2
		\\
		& \quad + \frac{1}{2}||(I-S_{*})^{1/2}[(I-S_i)^{-1}V(I-S_i)^{-1} -(I-S_j)^{-1}V(I-S_j)^{-1}](I-S_{*})^{1/2}||^2_{\HS}
		\\
		& \quad = A_1 + A_2 + A_3.
	\end{align*}
	For the first term, using $(I-S_i)^{-1} - (I-S_j)^{-1} = (I-S_i)^{-1}(S_i-S_j)(I-S_j)^{-1}$,
	\begin{align*}
		A_1 &= \frac{1}{2}(\trace[(I-S_i)^{-1}V] - \trace(I-S_j)^{-1}V])^2 
		\\
		&= \frac{1}{2}(\trace[(I-S_i)^{-1}(S_i-S_j)(I-S_j)^{-1}V])^2
		\\
		& \leq \frac{1}{2}||(I-S_i)^{-1}||^2||(I-S_j)^{-1}||^2||S_i-S_j||^2_{\HS}||V||_{\HS}^2.
	\end{align*}
	For the second term, using $(I-S_i)^{-1}V(I-S_i)^{-1} -(I-S_j)^{-1}V(I-S_j)^{-1} = 
	(I-S_i)^{-1}[V(I-S_i)^{-1}(S_i-S_j) + (S_i-S_j)(I-S_j)^{-1}V](I-S_j)^{-1}$
	\begin{align*}
		A_2 \leq \frac{1}{2}||I-S_{*}||^2||(I-S_i)^{-1}||^4||(I-S_j)^{-1}||^4||S_i-S_j||^2_{\HS}||V||^2_{\HS}.
	\end{align*}
	For the third term,
	\begin{align*}
		A_3 \leq \frac{1}{2}||I-S_{*}||^2||(I-S_i)^{-1}||^4||(I-S_j)^{-1}||^4||S_i-S_j||^2_{\HS}||V||^2_{\HS}.
	\end{align*}
	Combing all three terms $A_1,A_2,A_3$, we obtain
	\begin{align*}
		\Delta &\leq \left(\frac{1}{2} + ||I-S_{*}||^2||(I-S_i)^{-1}||^2||(I-S_j)^{-1}||^2\right)
		\\
		&\quad \times ||(I-S_i)^{-1}||^2||(I-S_j)^{-1}||^2||S_i-S_j||^2_{\HS}||V||^2_{\HS}.
	\end{align*}
\qed
\end{proof}

We now extend Proposition \ref{proposition:derivative-log-Radon-Nikodym-S-trace-class} 
to the general setting of $S \in \SymHS(\Hcal)_{< I} $.
The following shows that $D\log\left\{\frac{d\mu}{d\mu_0}(x) \right\}(S)(V)$
for $S \in \SymHS(\Hcal)_{<I}$, $V \in \Sym(\Hcal) \cap \Tr(\Hcal)$, can be obtained by taking limit
of a sequence $D\log\left\{\frac{d\mu}{d\mu_0}(x) \right\}(S_k)(V)$, with $S_k \in \SymTr(\Hcal)_{<I}$ and
$\lim\limits_{k \approach \infty}||S_k-S||_{\HS} = 0$.

\begin{proposition}
\label{proposition:D-logRN-S-HS-V-Tr}
Let $\mu = \mu(S) = \Ncal(0,C)$, $C = C_0^{1/2}(I-S)C_0^{1/2}$, $S \in \SymHS(\Hcal)_{<I}$.
Let $S_{*} \in \SymHS(\Hcal)_{< I}$ be fixed. Let $\mu_{*} = \Ncal(0,C_{*})$, $C_{*} = C_0^{1/2}(I-S_{*})C_0^{1/2}$, be the corresponding Gaussian measure.
Let $\{S_k\}_{k \in \Nbb}\in \SymTr(\Hcal)_{<I}$ 
be
such that $\lim\limits_{k \approach \infty}||S_k-S_{*}||_{\HS} = 0$.
Then
\begin{align}
D\log\left\{\frac{d\mu}{d\mu_0}(x) \right\}(S_{*})(V) = \lim_{k \approach \infty} D\log\left\{\frac{d\mu}{d\mu_0}(x) \right\}(S_k)(V)
\end{align}
$\forall V \in \Sym(\Hcal) \cap \Tr(\Hcal)$, where the limit is taken in $\Lcal^2(\Hcal,\mu_{*})$.
\end{proposition}
\begin{proof}
Let $f: \SymHS(\Hcal)_{<I} \mapto \Lcal^2(\Hcal,\mu_0)$ be defined by $f(S)(x) = \log\left\{{\frac{d\mu}{d\mu_0}(x)}\right\}$.
Since $S_k \in \SymTr(\Hcal)_{<I}$, $k \in \Nbb$, we have 
\begin{align*}
f(S_k)(x) = - \frac{1}{2}\log\det(I-S_k) - \frac{1}{2}\la C_0^{-1/2}x, S_k(I-S_k)^{-1}C_0^{-1/2}x\ra.
\end{align*}
By Proposition \ref{proposition:derivative-log-Radon-Nikodym-S-trace-class}, 
$Df(S_k): \Sym(\Hcal)\cap\HS(\Hcal) \mapto \Lcal^2(\Hcal,\mu_{*})$, when restricted to $\Sym(\Hcal)\cap\Tr(\Hcal)$, is given by,
$\forall V \in \Sym(\Hcal)\cap\Tr(\Hcal)$,
\begin{align*}
Df(S_k)(V) = \frac{1}{2}\trace\left[(I-S_k)^{-1}V\right] - \frac{1}{2}\la C_0^{-1/2}x, (I-S_k)^{-1}V{(I-S_k)^{-1}}C_0^{-1/2}x\ra.
\end{align*}
Let $V \in \Sym(\Hcal) \cap \Tr(\Hcal)$ be fixed. 
Since $\{S_k\}_{k \in \Nbb}$ is a Cauchy sequence in $\HS(\Hcal)$, by Lemma \ref{lemma:square-difference-DRadonNikodym-derivative}, 
with similar reasoning as in Corollary \ref{corollary:logRN-convergence-HS-norm},
$\{Df(S_k)(V)\}_{k \geq N_0}$ is a Cauchy sequence in $\Lcal^2(\Hcal,\mu_{*})$ for a sufficiently large $N_0 \in \Nbb$, hence must 
converge to a unique limit, denoted by $g(V) \in \Lcal^2(\Hcal,\mu_{*})$.
%
Let $t\in \R$ be sufficiently close to zero so that $I-(S_{*} + tV) > 0$. Let $J \in \Nbb$ be sufficiently large so that $I - (S_k + tV) > 0$ $\forall k \geq J$.
By definition of the Fr\'echet derivative, 
\begin{align*}
\lim_{t \approach 0}\frac{||f(S_k+tV) - f(S_k) - tDf(S_k)(V)||_{\Lcal^2(\Hcal,\mu_{*})}}{|t|\;||V||_{\HS}} = 0.
\end{align*}
By Corollary \ref{corollary:logRN-convergence-HS-norm}, 
$\lim\limits_{k \approach \infty}||f(S_k) - f(S_{*})||_{\Lcal^2(\Hcal,\mu_{*})} = 0$,
$\lim\limits_{k \approach \infty}||f(S_k +tV) - f(S_{*}+tV)||_{\Lcal^2(\Hcal,\mu_{*})} = 0$.
Thus taking limit as $k \approach \infty$ gives us
\begin{align*}
\lim_{t \approach 0}\frac{||f(S_{*}+tV) - f(S_{*}) - tg(V)||_{\Lcal^2(\Hcal,\mu_{*})}}{|t|\;||V||_{\HS}} = 0.
\end{align*}
It follows then by definition that $Df(S_{*})(V) = g(V)$.
\qed
\end{proof}

\begin{proposition}
	\label{proposition:D-logRN-inner-V-Tr}
		Let $\mu = \mu(S) = \Ncal(0,C)$, where $C = C_0^{1/2}(I-S)C_0^{1/2}$, $S \in \SymHS(\Hcal)_{<I}$.
	Let $S_{*} \in \SymHS(\Hcal)_{<I}$ be fixed. 
	Let $\mu_{*} = \Ncal(0,C_{*})$, $C_{*} = C_0^{1/2}(I-S_{*})C_0^{1/2}$, be the corresponding Gaussian measure.
	For any pair $V_i, V_j \in \Sym(\Hcal) \cap \Tr(\Hcal)$,
\begin{align}
&\left \la D\log\left\{\frac{d\mu}{d\mu_0}(x) \right\}(S_{*})(V_i), D\log\left\{\frac{d\mu}{d\mu_0}(x) \right\}(S_{*})(V_j)\right\ra_{\Lcal^2(\Hcal,\mu_{*})}
\nonumber
\\
&=\frac{1}{2}\trace[(I-S_{*})^{-1}V_i(I-S_{*})^{-1}V_j].
\end{align}
In particular, for $V_i = V_j = V$,
\begin{align}
\left\|D\log\left\{\frac{d\mu}{d\mu_0}(x) \right\}(S_{*})(V)\right\|^2_{\Lcal^2(\Hcal,\mu_{*})} = \frac{1}{2}||(I-S_{*})^{-1/2}V(I-S_{*})^{-1/2}||^2_{\HS}.
\end{align}
\end{proposition}
\begin{proof} 
(i) Consider first the case $S_{*} \in \SymTr(\Hcal)_{< I}$. Then
\begin{align*}
&D\log\left\{\frac{d\mu}{d\mu_0}(x) \right\}(S_{*})(V_i) 
\nonumber
\\
&= \frac{1}{2}\trace\left[(I-S_{*})^{-1}V_i\right] - \frac{1}{2}\la C_0^{-1/2}x, (I-S_{*})^{-1}V_i{(I-S_{*})^{-1}}C_0^{-1/2}x\ra.
\end{align*}
Let $A = (I-S_{*})^{-1}V_i{(I-S_{*})^{-1}}$, $B = (I-S_{*})^{-1}V_j{(I-S_{*})^{-1}}$. By Lemma \ref{lemma:integral-quadratic-S-star},
\begin{align*}
&\Delta = \left \la D\log\left\{\frac{d\mu}{d\mu_0}(x) \right\}(S_{*})(V_i), D\log\left\{\frac{d\mu}{d\mu_0}(x) \right\}(S_{*})(V_j)\right\ra_{\Lcal^2(\Hcal,\mu_{*})} 
\\
& = \frac{1}{4}\trace\left[(I-S_{*})^{-1}V_i\right]\trace\left[(I-S_{*})^{-1}V_j\right] 
\\
&\quad - \frac{1}{4}\trace\left[(I-S_{*})^{-1}V_i\right]\int_{\Hcal}\la C_0^{-1/2}x, BC_0^{-1/2}x\ra d\Ncal(0,C_{*})(x) 
\\
&\quad - \frac{1}{4}\trace\left[(I-S_{*})^{-1}V_j\right]\int_{\Hcal}\la C_0^{-1/2}x, AC_0^{-1/2}x \ra d\Ncal(0,C_{*})(x)
\\
& \quad + \frac{1}{4}\int_{\Hcal}\la C_0^{-1/2}x, AC_0^{-1/2}x\ra \la C_0^{-1/2}x, BC_0^{-1/2}x\ra d\Ncal(0,C_{*})(x)
\\
& = \frac{1}{4}\trace\left[(I-S_{*})^{-1}V_i\right]\trace\left[(I-S_{*})^{-1}V_j\right] 
\\
& \quad - \frac{1}{4}\trace\left[(I-S_{*})^{-1}V_i\right]\trace[(I-S_{*})^{-1}V_j] - \frac{1}{4}\trace\left[(I-S_{*})^{-1}V_i\right]\trace[(I-S_{*})^{-1}V_j]
\\
& \quad + \frac{1}{4}[\trace\left[(I-S_{*})^{-1}V_i\right]\trace\left[(I-S_{*})^{-1}V_j\right] + 2\trace[(I-S_{*})^{-1}V_i(I-S_{*})^{-1}V_j]]
\\
& = \frac{1}{2}\trace[(I-S_{*})^{-1}V_i(I-S_{*})^{-1}V_j].
\end{align*}

(ii) Consider now the
case $S_{*}  \in \SymHS(\Hcal)_{< I}$.
Let $\{S_k\}_{k \in \Nbb}$ in $\SymTr(\Hcal)_{<I}$ be such that $\lim\limits_{k \approach \infty}||S_k - S_{*}||_{\HS} = 0$.
By Proposition \ref{proposition:D-logRN-S-HS-V-Tr}, for $V \in \Sym(\Hcal) \cap \Tr(\Hcal)$, $D\log\left\{\frac{d\mu}{d\mu_0}(x) \right\}(S_{*})(V) 
= \lim\limits_{k \approach \infty}D\log\left\{\frac{d\mu}{d\mu_0}(x) \right\}(S_k)(V)$ in $\Lcal^2(\Hcal,\mu_{*})$,
\begin{align*}
&\Delta = \left \la D\log\left\{\frac{d\mu}{d\mu_0}(x) \right\}(S_{*})(V_i), D\log\left\{\frac{d\mu}{d\mu_0}(x) \right\}(S_{*})(V_j)\right\ra_{\Lcal^2(\Hcal,\mu_{*})} 	
\\
& = \lim_{k \approach \infty}\left\la D\log\left\{\frac{d\mu}{d\mu_0}(x) \right\}(S_k)(V_i), D\log\left\{\frac{d\mu}{d\mu_0}(x) \right\}(S_k)(V_j) \right\ra_{\Lcal^2(\Hcal,\mu_{*})}
\\
& = \frac{1}{2}\lim_{k \approach \infty}\trace[(I-S_k)^{-1}V_i(I-S_k)^{-1}V_j]
\\
& = \frac{1}{2}\trace[(I-S_{*})^{-1}V_i(I-S_{*})^{-1}V_j].
\end{align*}
Here we use the fact that $||(I-S_k)^{-1}-(I-S_{*})^{-1}||_{\HS} \leq ||(I-S_k)^{-1}||\;||S_k-S_{*}||_{\HS}||(I-S)^{-1}||
\approach 0$ as $k \approach \infty$.
\qed
\end{proof}

\begin{proposition}
\label{proposition:D-logRN-S-HS-V-HS}
	Let $\mu = \mu(S) = \Ncal(0,C), \mu_0 = \Ncal(0,C_0)$, with $C = C_0^{1/2}(I-S)C_0^{1/2}$, $S \in \SymHS(\Hcal)_{<I}$. 
		Let $S_{*} \in \SymHS(\Hcal)_{<I}$ be fixed. 
	Let $\mu_{*} = \Ncal(0,C_{*})$, $C_{*} = C_0^{1/2}(I-S_{*})C_0^{1/2}$, be the corresponding Gaussian measure.
	For
$V \in \Sym(\Hcal) \cap \HS(\Hcal)$, let $\{V_j\}_{j \in \Nbb} \in \Sym(\Hcal) \cap\Tr(\Hcal)$ be
such that $\lim\limits_{j \approach \infty}||V_j - V||_{\HS} = 0$. Then
\begin{align}
	&D\log\left\{\frac{d\mu}{d\mu_0}(x) \right\}(S_{*})(V) = \lim_{j \approach \infty}D\log\left\{\frac{d\mu}{d\mu_0}(x) \right\}(S_{*})(V_j).
\end{align}
Here the limit is taken in $\Lcal^2(\Hcal,\mu_{*})$.
For a pair $V,W \in \Sym(\Hcal) \cap \HS(\Hcal)$,
\begin{align}
&\left\la D\log\left\{\frac{d\mu}{d\mu_0}(x) \right\}(S_{*})(V), D\log\left\{\frac{d\mu}{d\mu_0}(x) \right\}(S_{*})(W)\right\ra_{\Lcal^2(\Hcal,\mu_{*})}
\nonumber
\\
&=\frac{1}{2}\trace[(I-S_{*})^{-1}V(I-S_{*})^{-1}W].
\end{align}
\end{proposition}
\begin{proof}
(i) 
By Proposition \ref{proposition:D-logRN-S-HS-V-Tr}, $D\log\left\{\frac{d\mu}{d\mu_0}(x) \right\}(S_{*})(V_j) \in \Lcal^2(\Hcal,\mu_{*})$ $\forall j \in \Nbb$
and by Proposition \ref{proposition:D-logRN-inner-V-Tr}, for each pair $i,j \in \Nbb$,
\begin{align*}
&\left\|D\log\left\{\frac{d\mu}{d\mu_0}(x) \right\}(S_{*})(V_i) - D\log\left\{\frac{d\mu}{d\mu_0}(x) \right\}(S_{*})(V_j)\right\|^2_{\Lcal^2(\Hcal,\mu_{*})}
\\
& = \frac{1}{2}||(I-S_{*})^{-1/2}V_i(I-S_{*})^{-1/2}||^2_{\HS} + \frac{1}{2}||(I-S_{*})^{-1/2}V_j(I-S_{*})^{-1/2}||^2_{\HS}
\\
&\quad -\trace[(I-S_{*})^{-1}V_i(I-S_{*})^{-1}V_j]
\\
& = \frac{1}{2}||(I-S_{*})^{-1/2}(V_i-V_j)(I-S_{*})^{-1/2}||^2_{\HS}
 \leq \frac{1}{2}||(I-S_{*})^{-1}||^2||V_i-V_j||^2_{\HS}.
\end{align*}
Since $\{V_j\}_{j \in \Nbb}$ is a Cauchy sequence in $\HS(\Hcal)$, $\left\{D\log\left\{\frac{d\mu}{d\mu_0}(x) \right\}(S_{*})(V_j)\right\}_{j \in \Nbb}$ is a Cauchy sequence in $\Lcal^2(\Hcal,\mu_{*})$ and must converge to a unique limit
in $\Lcal^2(\Hcal,\mu_{*})$.
Let $f:\SymHS(\Hcal)_{< I} \mapto \Lcal^2(\Hcal,\mu_{*})$ be defined by $f(S) = \log\left\{\frac{d\mu}{d\mu_0}(x)\right\}$. 
Let $t$ be sufficiently close to zero so that $S_{*}+tV \in \SymHS(\Hcal)_{<I}$.
Since $\lim\limits_{j \approach \infty}||V_j-V||_{\HS} = 0$, 
let $J \in \Nbb$ be sufficiently large so that $S_{*}+tV_j \in \SymHS(\Hcal)_{< I}$ $\forall j \geq J$. 
By definition of the Fr\'echet derivative,
\begin{align*}
	\lim_{t \approach 0}\frac{||f(S_{*}+tV_j) - f(S_{*}) - tDf(S_{*})(V_j)||_{\Lcal^2(\Hcal,\mu_{*})}}{|t|\;||V_j||_{\HS}} = 0.
\end{align*}
By Corollary \ref{corollary:logRN-convergence-HS-norm},
$\lim\limits_{j \approach \infty}||f(S_{*}+tV_j) - f(S_{*}+tV)||_{\Lcal^2(\Hcal,\mu_{*})} = 0$.
Taking limit as $j \approach \infty$ in $\Lcal^2(\Hcal,\mu_{*})$ thus gives
	$D\log\left\{\frac{d\mu}{d\mu_0}(x) \right\}(S_{*})(V) = \lim\limits_{j \approach \infty} D\log\left\{\frac{d\mu}{d\mu_0}(x) \right\}(S_{*})(V_j)$.
	
	(ii) For any pair $(V,W) \in \Sym(\Hcal) \cap \HS(\Hcal)$,
	let $\{V_j\}_{j \in \Nbb}$, $\{W_k\}_{k \in \Nbb}$ in $\Sym(\Hcal) \cap \Tr(\Hcal)$ be such
	that $\lim\limits_{j \approach \infty}||V_j - V||_{\HS} = 0$, $\lim\limits_{k \approach \infty}||W_k - W||_{\HS} = 0$.
	By part (i) and Proposition \ref{proposition:D-logRN-inner-V-Tr},
	\begin{align*}
	&\left\la D\log\left\{\frac{d\mu}{d\mu_0}(x) \right\}(S_{*})(V), D\log\left\{\frac{d\mu}{d\mu_0}(x) \right\}(S_{*})(W)\right\ra_{\Lcal^2(\Hcal,\mu_{*}) }
	\\
	&	=\lim_{j,k \approach \infty}\left\la D\log\left\{\frac{d\mu}{d\mu_0}(x) \right\}(S_{*})(V_j), D\log\left\{\frac{d\mu}{d\mu_0}(x) \right\}(S_{*})(W_k)\right\ra_{\Lcal^2(\Hcal,\mu_{*}) }
	\\
	& = \lim_{j,k \approach \infty}\frac{1}{2}\trace[(I-S_{*})^{-1}V_j(I-S_{*})^{-1}W_k]
	= \frac{1}{2}\trace[(I-S_{*})^{-1}V(I-S_{*})^{-1}W].
	\end{align*}
\qed
\end{proof}

\subsection{Miscellaneous technical results}
\label{section:Gaussian-integrals}

\begin{lemma}
	\label{lemma:rank-one-operator-norm}
	Let $a,b \in \Hcal$. The rank-one operator $a \otimes b:\Hcal \mapto \Hcal$ defined by $(a\otimes b)x = \la b,x\ra a$ has operator norm
	$||a\otimes b||\leq ||a||\;||b||$.
\end{lemma}
\begin{proof}
	This follows from $||(a \otimes b)x|| = |\la b,x\ra|\;||a|| \leq ||x||\;||b||\;||a||$. \qed
\end{proof}

\begin{lemma}
	\label{lemma:fourth-power-quadratic-bound}
	Let $a, b \in \Hcal$. Then 
	\begin{align}
		(||a||^2-||b||^2)^2 \leq	||a||^4 + ||b||^4 - 2 \la a,b\ra^2 \leq ||a-b||^2[||a||+||b||]^2.
	\end{align}
\end{lemma}
\begin{proof}
	First, by the Cauchy-Schwarz inequality, $||a||^4 + ||b||^4 - 2 \la a,b\ra^2 \geq ||a||^4 + ||b||^4 - 2||a||^2||b||^2 = (||a||^2-||b||^2)^2$. For the right inequality, we note that
	\begin{align*}
		||a||^4 - \la a,b\ra^2 &= (\la a,a\ra - \la a,b\ra)(\la a,a\ra + \la a,b\ra)  = \la a, a-b\ra \la a, a+b\ra
		\\
		& = \la a-b, (a \otimes a)(a+b)\ra,
		\\
		||b||^4 - \la a,b\ra^2 &= - \la a-b, (b \otimes b)(a+b)\ra,
		\\
		||a||^4 + ||b||^4 - 2\la a,b\ra^2 &= \la a-b, (a\otimes a - b\otimes b)(a+b)\ra.
	\end{align*}
	By Lemma \ref{lemma:rank-one-operator-norm}, the operator $a\otimes a - b\otimes b$ satisfies
	\begin{align*}
		||a\otimes a - b \otimes b|| \leq ||a\otimes (a-b)|| + ||(a-b)\otimes b|| \leq ||a-b||[||a|| + ||b||].
	\end{align*}
	Combining all the previous expressions, it  follows that
	\begin{align*}
		||a||^4 + ||b||^4 - 2\la a,b\ra^2 &\leq ||a-b||\;||a\otimes a - b\otimes b||[||a||+||b||] 
		\\
		&\leq ||a-b||^2[||a||+||b||]^2.
	\end{align*}
	\qed
\end{proof}

\begin{lemma}[Lemma 11 in \cite{Minh2020:EntropicHilbert}]
\label{lemma:gaussian-integral-quadratic-form}
Let $A \in \Lcal(\Hcal)$. Then
\begin{align}
\int_{\Hcal}\la x-m, A(x-m)\ra d\Ncal(m,C)(x) = \trace(CA).
\end{align}
\end{lemma}
\begin{lemma}[Lemma 23 in \cite{Minh:2020regularizedDiv}]
\label{lemma:gaussian-integral-double-square}
Let $a,b \in \Hcal$. Then
\begin{align}
\int_{\Hcal}\la x-m, a\ra^2\la x-m, b\ra^2 d\Ncal(m,C)(x) = \la a, Ca\ra\la b,Cb\ra + 2 \la a, Cb\ra^2.
\end{align}
\end{lemma}
\begin{lemma}
	\label{lemma:gaussian-integral-double-quadratic-form}
Let $A,B \in \Lcal(\Hcal)$. Then
\begin{align}
	&\int_{\Hcal}\la x-m, A(x-m)\ra \la x-m, B(x-m)\ra d\Ncal(m,C)(x)
	\nonumber
	\\
	& = 
 \trace(CA)\trace(CB) +\frac{1}{2}\trace(C(A+A^{*})C(B+B^{*})).
\end{align}
For $A,B \in \Sym(\Hcal)$,
\begin{align}
&\int_{\Hcal}\la x-m, A(x-m)\ra \la x-m, B(x-m)\ra d\Ncal(m,C)(x)
\nonumber
\\
& = \trace(CA)\trace(CB) + 2\trace(CACB).
\end{align}
In particular, for $A=B \in \Sym(\Hcal)$,
\begin{align*}
\int_{\Hcal}\la x-m, A(x-m)\ra^2 d\Ncal(m,C)(x) = [\trace(CA)]^2 + 2 \trace[(CA)^2].
\end{align*}
\end{lemma}
\begin{proof} It suffices to prove for $m=0$.

(i)	For $A,B \in \Sym^{+}(\Hcal)$, we have $\la x, Ax\ra = ||A^{1/2}x||^2 = \sum_{j=1}^{\infty}\la A^{1/2}x, e_j\ra^2$ $= \sum_{j=1}^{\infty}\la x, A^{1/2}e_j\ra^2$, $\la x, Bx\ra = \sum_{k=1}^{\infty}\la x, B^{1/2}e_k\ra^2$. Then
\begin{align*}
& \int_{\Hcal}\la x, Ax\ra \la x, Bx\ra d\Ncal(0,C)(x) = \int_{\Hcal}\sum_{j,k=1}^{\infty}\la x, A^{1/2}e_j\ra^2\la x, B^{1/2}e_k\ra^2 d\Ncal(0,C)(x)
\\
& = \sum_{j,k=1}^{\infty}\int_{\Hcal}\la x, A^{1/2}e_j\ra^2\la x, B^{1/2}e_k\ra^2 d\Ncal(0,C)(x)
\\
&\quad \text{by Lebesgue Monotone Convergence Theorem}
\\
& = \sum_{j,k=1}^{\infty}\left[\la A^{1/2}e_j, CA^{1/2}e_j\ra \la B^{1/2}e_k, CB^{1/2}e_k\ra + 2 \la A^{1/2}e_j, CB^{1/2}e_k\ra^2\right]
\\
& = \left(\sum_{j=1}^{\infty}\la e_j, A^{1/2}CA^{1/2}e_j\ra\right)\left(\sum_{k=1}^{\infty}\la e_k, B^{1/2}CB^{1/2}e_k\ra\right) + 2\sum_{j,k=1}^{\infty}\la e_j, A^{1/2}CB^{1/2}e_k\ra^2
\\
& = \trace(A^{1/2}CA^{1/2})\trace(B^{1/2}CB^{1/2}) + 2 ||A^{1/2}CB^{1/2}||^2_{\HS}
\\
& = \trace(CA)\trace(CB) + 2 \trace(B^{1/2}CACB^{1/2}) 
\\
&= \trace(CA)\trace(CB) + 2\trace(CACB).
\end{align*}

(ii) For $A,B \in \Sym(\Hcal)$, we write $A = A_1-A_2$, $B = B_1 - B_2$, where $A_1 =\frac{1}{2} (|A| +A)$, $A_2 = \frac{1}{2}(|A|-A)$,
$B_1 = \frac{1}{2}(|B|+B)$, $B_2 = \frac{1}{2}(|B|-B)$. Here $A_1,A_2,B_1, B_2 \in \Sym^{+}(\Hcal)$. Applying part (i) gives
\begin{align*}
&\int_{\Hcal}\la x, Ax\ra \la x, Bx\ra d\Ncal(0,C)(x) = \int_{\Hcal}\la x, (A_1-A_2)x\ra \la x, (B_1-B_2)x\ra d\Ncal(0,C)(x)
\\
& = \int_{\Hcal}[\la x, A_1x\ra - \la x, A_2x\ra][\la x, B_1x\ra-\la x,B_2x\ra]d\Ncal(0,C)(x)
\\
& = \int_{\Hcal}[\la x, A_1x\ra\la x, B_1x\ra - \la x, A_1x\ra\la x, B_2x\ra]d\Ncal(0,C)(x)
\\
&\quad + \int_{\Hcal}-\la x, A_2x\ra\la x,B_1x\ra + \la x, B_1x\ra\la x, B_2x\ra]d\Ncal(0,C)(x)
\\
& = [\trace(CA_1)\trace(CB_1) + 2\trace(CA_1CB_1)] -[\trace(CA_1)\trace(CB_2) + 2\trace(CA_1)CB_2]
\\
&\quad - \trace[(CA_2)\trace(CB1) + 2\trace(CA_2CB_1)] + [\trace(CA_2)\trace(CB_2) + 2 \trace(CA_2CB_2)]
\\
& = \trace[C(A_1-A_2)]\trace[C(B_1-B_2)] + 2\trace[C(A_1-A_2)C(B_1-B_2)]
\\
& = \trace(CA)\trace(CB) + 2\trace(CACB).
\end{align*}  

(iii) For the general case $A, B \in \Lcal(\Hcal)$, we have $\la x, Ax\ra = \frac{1}{2}\la x, (A+A^{*})x\ra$, $\la x, Bx\ra = \frac{1}{2}\la x, (B+B^{*})x\ra$. Applying part (ii) gives
\begin{align*}
&\int_{\Hcal}\la x, Ax\ra \la x, Bx\ra d\Ncal(0,C)(x) = 
\frac{1}{4}\int_{\Hcal}[\la x, (A+A^{*})x\ra \la x, (B+B^{*})x\ra d\Ncal(0,C)(x)
\\
& = \frac{1}{4}\trace[C(A+A^{*})]\trace[C(B+B^{*})] + \frac{1}{2}\trace(C(A+A^{*})C(B+B^{*}))
\\
& = \trace(CA)\trace(CB) +\frac{1}{2}\trace(C(A+A^{*})C(B+B^{*})).
\end{align*}
Here we have used the cyclic property of the trace and the fact that $\trace(A^{*}) = \trace(A)$ for $A \in \Tr(\Hcal)$.
\qed
\end{proof}

\bibliographystyle{plain}
\bibliography{../cite_RKHS}
\end{document}